\documentclass[11pt,reqno]{amsart}

\usepackage{amssymb, amsmath, amsthm}
\usepackage{hyperref}
\usepackage[alphabetic,lite]{amsrefs}
\usepackage{verbatim}
\usepackage{amscd}   
\usepackage[all]{xy} 
\usepackage{youngtab} 
\usepackage{young} 
\usepackage{ytableau}
\usepackage{tikz}
\usepackage{tikz-cd}
\usepackage{ mathrsfs }
\usepackage{cases}
\usepackage{array}
\usepackage{cellspace}
\usepackage{tabu}
\usepackage{calligra,mathrsfs}
\usepackage{bm}

\DeclareMathOperator{\ShHom}{\mathscr{H}\text{\kern -3pt {\calligra\large om}}\,}

\renewcommand{\a}{\alpha}
\renewcommand{\b}{\beta}
\newcommand{\bw}{\bigwedge}

\newcommand{\D}{\mathcal{D}}

\newcommand{\oo}{\otimes}

\newcommand{\bS}{\mathbb{S}}

\newcommand{\tn}{\textnormal}

\newcommand{\C}{\mathbb{C}}
\newcommand{\charC}{\operatorname{charC}}

\newcommand{\GL}{\operatorname{GL}}

\newcommand{\SL}{\operatorname{SL}}

\newcommand{\Sym}{\operatorname{Sym}}

\newcommand{\codim}{\operatorname{codim}}

\renewcommand{\det}{\operatorname{det}}
\newcommand{\dom}{\operatorname{dom}}

\newcommand{\opmod}{\operatorname{mod}}

\newcommand{\W}{\mathcal{W}}

\newcommand{\bb}[1]{\mathbb{#1}}

\newcommand{\mc}[1]{\mathcal{#1}}
\newcommand{\mf}[1]{\mathfrak{#1}}
\newcommand{\ol}[1]{\overline{#1}}
\newcommand{\op}[1]{\operatorname{#1}}

 \newtheorem{theorem}{Theorem}[section]
 \newtheorem{lemma}[theorem]{Lemma}
 \newtheorem{corollary}[theorem]{Corollary}
 \newtheorem{prop}[theorem]{Proposition}
 
 \newtheorem{remark}[theorem]{Remark}

\numberwithin{equation}{section}

\newcommand{\andras}[1]{{\color{brown} \sf $\clubsuit\clubsuit\clubsuit$ Andr\'as: [#1]}}
\newcommand{\mike}[1]{{\color{blue} \sf $\clubsuit\clubsuit\clubsuit$ Mike: [#1]}}

\author{Andr\'{a}s C. L\H{o}rincz}
\author{Michael Perlman}

\title{Equivariant $\D$-modules on $2\times 2\times n$ hypermatrices}

\begin{document}
\maketitle 

\begin{abstract}
We study $\D$-modules and related invariants on the space of $2\times 2\times n$ hypermatrices for $n\geq 3$, which has finitely many orbits under the action of $G=\GL_2(\C) \times \GL_2(\C) \times \GL_n(\C)$. We describe the category of coherent $G$-equivariant $\D$-modules as the category of representations of a quiver with relations. We classify the simple equivariant $\D$-modules, determine their characteristic cycles and find special representations that appear in their $G$-structures. We determine the explicit $\D$-module structure of the local cohomology groups with supports given by orbit closures. As a consequence, we calculate the Lyubeznik numbers and intersection cohomology groups of the orbit closures. All but one of the orbit closures have rational singularities: we use local cohomology to prove that the one exception is neither normal nor Cohen--Macaulay. While our results display special behavior in the cases $n=3$ and $n=4$, they are completely uniform for $n\geq 5$.
\end{abstract}

\section{Introduction}\label{sec:orbits}

The roots behind hypermatrices can be traced back to Cayley's hyperdeterminant, and this classical theory has continued to generate research since then -- see \cite{gkz} and \cite{landsberg2012tensors} and references therein. The theory has a wide range of applications, since tensors and (multi-)linearity are universal features in the sciences.  In particular, they have recently gained more traction due to their use in complexity theory.

Spaces of hypermatrices have large groups of symmetries arising from natural actions of products of general linear groups, which corresponds to performing changes of bases. While for matrices the rank completely parametrizes the orbits, this is not true for hypermatrices in general, and there can be an infinite number of them. In fact, the ones that have finitely many orbits are classified (see \cite{satokimura, kac}), and correspond to the space of $m\times n$ matrices with $m\geq n \geq 1$, the $2\times 2\times n$  hypermatrices with $n\geq 2$,  and the $2\times 3 \times n$ hypermatrices with $n\geq 3$.

In this paper, we study the geometry of the space of $2\times 2\times n$ hypermatrices (for $n\geq 3$) from a $\D$-module-theoretic point of view. Similar work was undertaken in the case of $m\times n$ matrices \cite{raicuWeymanLocal, lorcla}, as well as in the case of $2\times 2 \times 2$ hypermatrices \cite{perlman2020equivariant}. Since the orbit decomposition yields a natural finite stratification of these spaces, there are a number powerful techniques that can be used (cf. \cite{lHorincz2019categories}) to compute explicit objects and invariants that are, in general, notoriously difficult to determine.

This work is part of an ongoing effort to investigate categories of equivariant coherent  $\D$-modules on spaces with finitely many orbits under actions of linear algebraic groups, such as flag varieties, toric varieties, and nilpotent cones. It follows from Luna's Slice Theorem \cite{luna} that, in the case of a smooth affine variety under the action of a reductive group, this category is equivalent to the corresponding category of equivariant $\D$-modules on the normal space to a point in the closed orbit (see \cite[Proposition 4.6]{lHorincz2019categories}). Thus, the study of categories of equivariant coherent $\D$-modules on such affine varieties reduces to the case of representations of reductive groups with finitely many orbits. The spaces of $2\times 2\times n$ hypermatrices are examples of Vinberg representations, which are the \textit{irreducible} representations with finitely many orbits (see \cite[Section 10.1]{landsberg2012tensors},  cf. also \cite{satokimura, kac}). Such representations are classified and, besides a couple of exceptions, can be described by Dynkin diagrams with a choice of simple root -- the $2\times 2\times n$ case corresponds to the type $\mathbb{D}_{n+2}$ with the choice of the node with degree $3$. Quiver categories of equivariant $\D$-modules on Vinberg representations have been studied in several articles \cite{binary, lHorincz2019categories, perlman2020equivariant,  senary, subexceptional}.

\medskip

We now introduce the basic setup. Let $A=\langle a_1, a_2\rangle$ and $B=\langle b_1, b_2\rangle$ be two-dimensional complex vector spaces, and let $C=\langle c_1,\cdots,c_n\rangle$ be an $n$-dimensional complex vector space with $n\geq 2$. We consider the space $V=A\oo B\oo C$ of $2\times 2\times n$ hypermatrices, endowed with the natural action of $\GL=\GL(A)\times \GL(B)\times \GL(C)$. The $\GL$-orbits of $V$ are described below (see \cite[Table 10.3.1]{landsberg2012tensors} or \cite{parfenov}).
\begin{itemize}
\item The zero orbit $O_0=\{0\}$. 

\item The orbit $O_1$ of dimension $n+2$, with representative $a_1\oo b_1\oo c_1$.

\item The orbit $O_2$ of dimension $n+3$, with representative $a_1\oo b_1\oo c_1+a_2\oo b_2\oo c_1$.

\item The orbit $O_3$ of dimension $2n+1$, with representative $a_1\oo b_1\oo c_1+a_1\oo b_2\oo c_2$.

\item The orbit $O_4$ of dimension $2n+1$, with representative $a_1\oo b_1\oo c_1+a_2\oo b_1\oo c_2$.

\item The orbit $O_5$ of dimension $2n+3$, with representative $a_1\oo(b_1\oo c_1+b_2\oo c_2)+a_2\oo b_1\oo c_2$.

\item The orbit $O_6$ of dimension $2n+4$, with representative $a_1\oo b_1\oo c_1+a_2\oo b_2\oo c_2$.

\item The orbit $O_7$ of dimension $3n+2$, with representative $a_1\oo(b_1\oo c_1+b_2\oo c_3)+a_2\oo b_1\oo c_2$.

\item The orbit $O_8$ of dimension $3n+3$, with representative $a_1\oo (b_1\oo c_1+b_2\oo c_2)+a_2\oo (b_1\oo c_2+b_2\oo c_3)$.

\item The orbit $O_9$ of dimension $4n$, with representative $a_1\oo(b_1\oo c_1+b_2\oo c_3)+a_2\oo (b_1\oo c_2+b_2\oo c_4)$.
\end{itemize}

For $n=2$ the orbit $O_6$ is dense, $\ol{O}_5$  is a hypersurface, and orbits $7-9$ do not appear. For $n=3$ the orbit $O_8$ is dense, $\ol{O}_7$ is a hypersurface, and orbit $9$ does not appear. For $n\geq 4$ all orbits appear, and $\ol{O}_8$ is a hypersurface if and only if $n=4$.

We recall that for positive integers $i\leq 2,j \leq 2,k \leq n$ the subspace variety $\tn{Sub}_{ijk} \subset V$ is the closed subset of tensors $v \in V$ such that there exists subspaces $A' \subset A, B'\subset B, C'\subset C$ of dimensions $i,j,k$ respectively, such that $v \in A' \oo B' \oo C'$. 
In our case $\ol{O}_1 = \tn{Sub}_{111}, \ol{O}_2 = \tn{Sub}_{221},\ol{O}_3 = \tn{Sub}_{122},\ol{O}_4 = \tn{Sub}_{212}, \ol{O}_6 = \tn{Sub}_{222}, \ol{O}_8=\tn{Sub}_{223}, \ol{O}_9 = \tn{Sub}_{224}$. In particular, $\ol{O}_1$ is the affine cone over the Segre variety $\operatorname{Seg}(\bb{P}(A)\times \bb{P}(B)\times \mathbb{P}(C))$. The orbit closures $\ol{O}_5$ and $\ol{O}_7$ have descriptions as affine cones over the tangential and projectively dual varieties of this Segre variety. Using this (or see \cite{parfenov}), the following is the Hasse diagram with respect to containment of orbit closures for $n\geq 4$ (for $n=3$ the diagram is obtained by removing $O_9$):
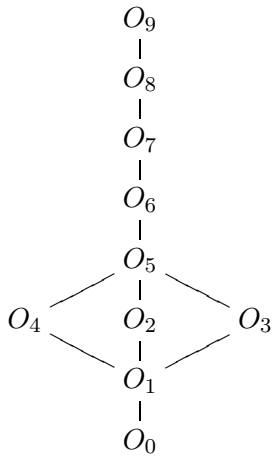
\[\xymatrix@R-1.4pc{
 & O_9 \ar@{-}[d] & \\
 & O_8 \ar@{-}[d]& \\
 & O_7 \ar@{-}[d]& \\
 & O_6 \ar@{-}[d]& \\
 & O_5 \ar@{-}[dl]\ar@{-}[d]\ar@{-}[dr]& \\
O_4 \ar@{-}[dr] & O_2 \ar@{-}[d] & O_3 \ar@{-}[dl] \\
 & O_1 \ar@{-}[d]& \\
 & O_0 & \\
}\]

Due to the exceptional isomorphism $\op{Spin}_4(\C) \cong \SL_2(\C) \times \SL_2(\C)$, the $\GL$-representation $V$ is equivalent to the action of $\op{SO}_4(\C)\times  \GL_n(\C)$ on the space of $4\times n$ matrices $X$, i.e. the action maps into $\GL(V)$ have the same image. As the latter action is sometimes more convenient to handle, we will use this identification several times throughout the article. The orbit closures can be defined set-theoretically by rank conditions imposed on $X$ and $X^t\cdot X$, with the exception of $\ol{O}_3$ and $\ol{O}_4$ as such conditions define their union only.

\smallskip

\subsection*{Summary of results and organization} The paper is organized as follows. In Section \ref{sec:prelim}, we give the necessary background and introduce some techniques that are used subsequently in the article.

\smallskip

In Section \ref{sec:steps}, we derive our preliminary results. In particular, we classify the simple objects in the category of $\GL$-equivariant $\D$-modules $\tn{mod}_{\GL}(\D_V)$ in Proposition \ref{componentGrp} and describe how the (twisted) Fourier transform permutes them (Proposition \ref{prop:fourier}). We further determine the characteristic cycles of the simples in $\tn{mod}_{\GL}(\D_V)$ (Proposition \ref{prop:char566}) and determine the local cohomology modules supported in $\ol{O}_1$ (Theorem \ref{BettiThm}).

\smallskip

In Section \ref{sec:eqcat}, we describe the category of $\GL$-equivariant $\D$-modules $\tn{mod}_{\GL}(\D_V)$ as the category of finite-dimensional representations of a quiver (with relations). While the cases $n=3$ (Theorem \ref{thm:quiv3}) and $n=4$ (Theorem \ref{thm:quiv4}) are more involved as the orbits sit more crammed, the description simplifies and is uniform for $n\geq 5$ (Theorem \ref{quiverN5}).

\smallskip

In Section \ref{sec:o5} we compute the explicit $\D$-module structure of the local cohomology modules supported in $\ol{O}_5$ (Theorem \ref{loc05}), while in Section \ref{sec:O7} those supported in $\ol{O}_7$ (Theorem \ref{loc07}). Furthermore, we explain how the latter result can be used to show that $\ol{O}_7$, unlike the other orbit closures, is not normal nor Cohen--Macaulay, see Corollary \ref{cor:sing}.

\smallskip

In Section \ref{sec:characters} we find special representations (so-called witness weights) for each simple $\GL$-equivariant $\D$-module, see Theorem \ref{thm:witnessngeq4}. This allows the determination of the composition factors of any equivariant $\D$-module through representation-theoretic means, by just finding multiplicities of these witness weights.

\smallskip

Finally, in Section \ref{sec:lyub} we use our results obtained on local cohomology modules to determine the Lyubeznik numbers and intersection cohomology groups of each orbit closure.

\section{Preliminaries}\label{sec:prelim}

\subsection{Background on $\D$-modules and functors}\label{sec:Dmod} Our $\D$-modules are left $\D$-modules. Given a morphism $f:X\to Y$ of smooth varieties, we write $f_+$ for $\D$-module pushforward along $f$ \cite[Section 1.5]{htt}.

Let $X=\mathbb{A}^N$. For a closed subvariety $Z\subset X$, we denote by $L_Z$ the intersection cohomology $\D_X$-module of Brylinksi and Kashiwara \cite[Section 3.4]{htt}, that corresponds to the intersection cohomology complex $IC_Z$ via the Riemann--Hilbert correspondence. Throughout we denote $E=L_{\{0\}}$ and $S=L_X$, where the latter stands for the coordinate ring of $X$. When $X$ is the space of $2\times 2\times n$ hypermatrices, throughout the article we use the notation $D_p = L_{\ol{O}_p}$, for all $p=0,\dots, 9$.

\smallskip

Suppose an algebraic group $G$ acts on $X$, and let $\mf{g}$ denote the Lie algebra of $G$. A $\D_X$-module $M$ is equivariant if we have a $\D_{G\times X}$-isomorphism $\tau: p^*M \rightarrow m^*M$, where $p: G\times X\to X$ denotes the projection and $m: G\times X\to X$ the map defining the action, with $\tau$ satisfying the usual compatibility conditions (see \cite[Definition 11.5.2]{htt}). On the other hand, differentiating the $G$-action on $X$ we obtain a map from $\mf{g}$ to the space of algebraic vector fields on $X$, which in turn yields a map $\mf{g} \to \D_X$. This gives an equivalent way of looking at equivariance: $M$ admits an algebraic $G$-action whose differential coincides with the $\mf{g}$-action induced by the natural map $\mf{g} \to \D_X$. Since $G$ is connected, the category  $\tn{mod}_{G}(\D_X)$ of equivariant $\D$-modules is a full subcategory of the category of all coherent $\D$-modules, which is closed under taking submodules and quotients. For a $G$-stable closed subset $Z$ of $X$, we denote by $ \tn{mod}_{G}^Z(\D_X)$ the full subcategory of $\tn{mod}_{G}(\D_X)$ consisting of equivariant $\D$-modules supported inside $Z$.

When $G$ acts on $X$ with finitely many orbits, each module in  $\tn{mod}_{G}(\D_X)$ is regular and holonomic \cite[Theorem~11.6.1]{htt}. Therefore, by the Riemann--Hilbert correspondence the category $\tn{mod}_{G}(\D_X)$ is equivalent to the category of $G$-equivariant perverse sheaves on $X$. Moreover, this category is equivalent to the category of finitely generated modules over a finite dimensional $\bb{C}$-algebra (see \cite[Theorem 4.3]{vilonen} or  \cite[Theorem 3.4]{lHorincz2019categories}), in other words, to the category of finite-dimensional representations of a quiver with relations.

When $G$ is reductive acting linearly on $X$, another important functor on $\opmod_G(\D_X)$ comes from considering the (twisted) Fourier transform \cite[Section 4.3]{lHorincz2019categories}. This functor gives an involution
\[\mc{F} : \opmod_G(\D_X) \xrightarrow{\sim} \opmod_G(\D_X).\]

For a $G$-stable (locally) closed subset $Z$ of $X$, we can consider for each $i\geq 0$ and $M\in\opmod_G(\D_X)$ the $i$-th local cohomology module $H^i_Z(M)$ of $M$ with support in $Z$, which is an object in  $\tn{mod}_{G}^Z(\D_X)$. Similarly, for a smooth variety $X$ we write $\mathscr{H}^i_Z(\mc{M})$ for the local cohomology sheaves associated to a sheaf of $\D_X$-modules $\mc{M}$. When $Z$ is a divisor, there is a short exact sequence of $\D_X$-modules
\begin{equation}\label{comphypn2}
0\longrightarrow \mc{O}_X\longrightarrow  \mc{O}_X(\ast Z)\longrightarrow \mathscr{H}^1_Z(\mc{O}_X)\longrightarrow 0,
\end{equation}
where we denote by $\mc{O}_X(\ast Z)$ the localization of $\mc{O}_X$ at the divisor $Z$.

\subsection{Local cohomology preliminaries}

Let $Z'\subset Z$ be closed subvarieties of $X=\mathbb{A}^N$, and let $Y=Z\setminus Z'$. We set $c=\operatorname{codim}(Z,\mathbb{A}^N)$. There is a long exact sequence of $\D$-modules (see \cite[Proposition 1.9]{hartshornelocal}): 
\begin{equation}\label{locallyclosedsequence}
0 \to H^c_Z(S) \to H^c_Y(S)\to H^{c+1}_{Z'}(S)\to H^{c+1}_Z(S) \to H^{c+1}_Y(S)\to H^{c+2}_{Z'}(S)\to \cdots 		
\end{equation}

We store the following fact relating local cohomology supported at the origin to intersection cohomology.

\begin{prop}\label{prop:intcoho}
Let $Z \subset \mathbb{A}^N$ be the cone over an irreducible projective variety in $\mathbb{P}^{N-1}$. Put $d=\dim_{\C} Z$, and $c= N -d$. For all $i\in \bb{Z}$, we have
\[\mbox{\large\(H_{\{0\}}^{d-i}(L_Z)\)}=\mbox{\large\(E\)}^{\mbox{\small\(\oplus \dim IH^i(Z)\)}}=\mbox{\large\(H_{\{0\}}^{i+c}(\mc{F}(L_Z))\)}. \]


\end{prop}


\begin{proof}
The first equality follows from \cite[Theorem 2.10.3]{achar}, using that $L_Z$ is self-dual. The second equality follows from  \cite[Proposition 5.1]{subexceptional}.
\end{proof}

We mention a useful tool for calculating local cohomology modules that is based on invariant theory.

\begin{lemma}\label{lem:locinv}
Let $G$ be a reductive group acting on $X$. Take an ideal $I\subset S^G$. Then for each $i\geq 0$ we have a $\D_X^G$-isomorphism
\[ \left(H^i_{S\cdot I}(S)\right)^G \cong H^i_I(S^G).\]
\end{lemma}

\begin{proof} Use the \v{C}ech-type sequence with respect to a set of generators of $I$ to calculate local cohomology, and the fact that taking $G$-invariants is exact.
\end{proof}

\subsection{Representation theory of the general linear group}\label{sec:reps} Let $W$ be an $m$-dimensional complex vector space, and let $\GL(W)$ denote the general linear group, consisting of linear automorphisms of $W$. Irreducible representations of $\GL(W)$ are indexed by dominant weights $\lambda=(\lambda_1\geq \lambda_{2}\geq \cdots \geq \lambda_m)$ where $\lambda\in \mathbb{Z}^m$. We write $\mathbb{Z}^m_{\dom}$ for the set of dominant weights with $m$ entries, and we write $\bS_{\lambda}W$ for the irreducible representation corresponding to $\lambda$, where $\bS_{\lambda}(-)$ is the corresponding Schur functor. For example, for $d\geq 0$ we have $\bS_{(d)}(W)=\Sym^d(W)$ and $\bS_{(1^d)}(W)=\bw^d(W)$. Here, as we often do throughout the article, we omit trailing zeros in dominant weights, e.g. $(d)=(d,0,\cdots,0)$, and we use exponents to denote the multiplicity of an entry, e.g. $(1^m)=(1,\cdots, 1)$. Given a representation $U$ of dimension $d$, we write $\det(U)=\bw^d(U)$ for the highest nonzero exterior power of $U$. Given $\lambda\in \mathbb{Z}^m_{\dom}$ we write $|\lambda|=\lambda_1+\cdots +\lambda_m$, and we say that $\lambda$ is a partition if $\lambda_m\geq 0$. We have an isomorphism $\bS_{\lambda}(W)^{\ast}\cong \bS_{\lambda^{\ast}}(W)$, where $\lambda^{\ast}=(-\lambda_m,\cdots, -\lambda_1)$.

Throughout the article, $A$ and $B$ are two-dimensional complex vector spaces, and $C$ is an $n$-dimensional vector space for some $n\geq 3$. We write $\GL=\GL(A)\times \GL(B)\times \GL(C)$. Irreducible representations of $\GL$ are of the form $\bS_{\alpha}A\oo \bS_{\beta}B\oo \bS_{\gamma}C$, where $\alpha,\beta\in \mathbb{Z}^2_{\dom}$ and $\gamma\in \mathbb{Z}^n_{\dom}$. Given a representation $U$ of $\GL$, we write $[U:\bS_{\alpha}A\oo \bS_{\beta}B\oo \bS_{\gamma}C]$ for the multiplicity of $\bS_{\alpha}A\oo \bS_{\beta}B\oo \bS_{\gamma}C$ as a summand of $U$. We note that the multiplicity of an irreducible $\GL$-representation in a module in $\tn{mod}_{\GL}(\D_V)$ is finite \cite[Proposition 3.14]{lHorincz2019categories}. 

We write $V=A\oo B\oo C$ and $S$ for the polynomial ring $S=\Sym(V)$. It is a difficult problem in general to find the $\GL$-irreducible summands of $S$ (see \cite[pg. 60-61]{weyman}). Nevertheless, one may obtain some information about the $\GL$ structure $S$ via one of the three flattenings:
$$
A\otimes B\otimes C\cong (A\otimes B)\otimes C\cong A\oo(B\oo C)\cong B\oo(A\oo C).
$$
For example, by the Cauchy Formula \cite[Theorem 2.3.2]{weyman} we have the following $\GL$-decomposition of $S$:
$$
\bigoplus_{\gamma=(\gamma_1\geq \gamma_2\geq \gamma_3\geq\gamma_4\geq 0)} \bS_{\gamma}(A\oo B)\otimes \bS_{\gamma}C.
$$
There is no general formula to decompose $\bS_{\gamma}(A\oo B)$, unless $\gamma_3=\gamma_4$ (or dually $\gamma_1=\gamma_2$). Upon twisting by a suitable power of the determinant representation, this is the case when $\gamma$ has at most two rows. If $\gamma$ has one row, we may apply the Cauchy Formula. If $\gamma$ has two rows, one may apply \cite[Corollary 4.3b]{raicu2012secant} to find the multiplicity of $\bS_{\alpha}A\oo \bS_{\beta}B$ in $\bS_{\gamma}(A\oo B)$, which we often do in this article.

\subsection{Equivariant $\D$-modules on spaces of matrices}\label{sec:dmodsMat} Here we discuss some information on equivariant $\D$-modules on spaces of matrices, and local cohomology modules with support in determinantal varieties, which we use repeatedly throughout the article.

We employ the flattening of $V$ to the space of $4\times n$ matrices, and use the known $\D$-module structure of local cohomology with support the determinantal varieties $Z_p$ of matrices with rank $\leq p$. Note that we have
\begin{equation} \label{eq:detorbits}
O_0=Z_0, \; \overline{O}_2=Z_1, \; \overline{O}_6=Z_2, \, \mbox{ and }\, \overline{O}_8=Z_3 \, (\mbox{for } n\geq 3). 
\end{equation}

Let $n=4$ and $\det=\det(x_{i,j})$ denote the $4\times 4$ determinant on $(A\oo B)\oo C$. The localization $S_{\det}$ of the polynomial ring at $\det$ is a holonomic $\D$-module, with composition series as follows \cite[Theorem 1.1]{raicu2016characters}:
$$
0\subsetneq S\subsetneq \D\cdot \det^{-1}\subsetneq \D\cdot \det^{-2}\subsetneq \D\cdot \det^{-3}\subsetneq S_{\det},
$$
with $(\D\cdot \det^{-1})/S=D_8$, $(\D\cdot \det^{-2})/(\D\cdot \det^{-1})=D_6$, $(\D\cdot \det^{-3})/(\D\cdot \det^{-2})=D_2$, and $S_{\det}/(\D\cdot \det^{-3})=D_0$. Following \cite{lorcla} we define
\begin{equation}\label{eq:detQ}
Q_3=S_{\det}/S, \quad Q_2=S_{\det}/(\D\cdot \det^{-1}),\quad Q_1=S_{\det}/(\D\cdot \det^{-2}).
\end{equation}

With this notation, we recall the $\D$-module structure of local cohomology with support in $\overline{O}_2$, $\overline{O}_6$, $\overline{O}_8$ (see \cite[Main Theorem]{raicu2016ocal} and \cite[Theorem 1.6]{lorcla}):

\begin{theorem}\label{localSegre}
The following is true about local cohomology of $S$ with support in $\ol{O}_2$.
\begin{enumerate}
\item If $n=3$ the nonzero local cohomology modules	are given by
$$
H^6_{\ol{O}_2}(S)=D_2,\quad H^7_{\ol{O}_2}(S)=D_0,\quad H^9_{\ol{O}_2}(S)=D_0.
$$
\item If $n=4$ the nonzero local cohomology modules are given by 
$$
H^9_{\ol{O}_2}(S)=Q_1,\quad H^{11}_{\ol{O}_2}(S)=D_0,\quad H^{13}_{\ol{O}_2}(S)=D_0.
$$
\item If $n\geq 5$ the nonzero local cohomology modules are given by
$$
H^{3n-3}_{\ol{O}_2}(S)=D_2,\quad H^{4n-7}_{\ol{O}_2}(S)=D_0,\quad H^{4n-5}_{\ol{O}_2}(S)=D_0,\quad H^{4n-3}_{\ol{O}_2}(S)=D_0.
$$
\end{enumerate}
\end{theorem}

\begin{theorem}\label{thm:localcoho6}
The following is true about local cohomology of $S$ with support in $\ol{O}_6$.	
\begin{enumerate}
\item If $n=3$ the nonzero local cohomology modules	are given by
$$
H^2_{\ol{O}_6}(S)=D_6,\quad H^3_{\ol{O}_6}(S)=D_2,\quad H^4_{\ol{O}_6}(S)=D_0.
$$
\item If $n=4$ the nonzero local cohomology modules are given by 
$$
H^4_{\ol{O}_6}(S)=Q_2,\quad H^6_{\ol{O}_6}(S)=Q_1,\quad H^8_{\ol{O}_6}(S)=D_0.
$$
\item If $n\geq 5$ the local cohomology modules are semi-simple, and we have the following in the Grothendieck group of $\tn{mod}_{\GL}(\D_V)$:
$$
\sum_{j\geq 0}[H^j_{\ol{O}_6}(S)]\cdot q^j=[D_6]\cdot q^{2n-4}+[D_2]\cdot (q^{3n-8}+q^{3n-6})+[D_0]\cdot (q^{4n-12}+q^{4n-10}+q^{4n-8}).
$$
\end{enumerate}
\end{theorem}

\begin{theorem}\label{thm:localcoho8}
The following is true about local cohomology of $S$ with support in $\ol{O}_8$.	
\begin{enumerate}
\item If $n=4$ the only nonzero local cohomology module is $H^1_{\ol{O}_8}(S)=Q_3$.
\item If $n\geq 5$ the nonzero local cohomology modules are given by
$$
H^{n-3}_{\ol{O}_8}(S)=D_8,\quad H^{2n-7}_{\ol{O}_8}(S)=D_6,\quad H^{3n-11}_{\ol{O}_8}(S)=D_2,\quad H^{4n-15}_{\ol{O}_8}(S)=D_0.
$$
\end{enumerate}
\end{theorem}

Via one of the two identifications of $V$ with the space of $2\times 2n$ matrices, the orbit closures $\overline{O}_3$ and $\ol{O}_4$ are the determinantal varieties of matrices of rank $\leq 1$. By \cite[Main Theorem]{raicu2016ocal} we have the following.
\begin{theorem}\label{thm:localcoho34}
Let $i=3,4$. For all $n\geq 2$ the nonzero local cohomology modules of $S$ with support in $\ol{O}_i$ are given by
$$
H^{2n-1}_{\overline{O}_i}(S)=D_i,\quad H^{4n-3}_{\overline{O}_i}(S)=D_0.
$$
\end{theorem}

We recall the characters of the simple modules on spaces of matrices \cite[Section 3.2]{raicu2016characters}. Let $W_1$ be an $n_1$-dimensional vector space, and let $W_2$ be an $n_2$-dimensional vector space, with $n_1\geq n_2$. We consider the space $W_1\oo W_2$ of $n_1\times n_2$ matrices. For $0\leq p\leq n_2$, let $L_{Z_p}$ to be the intersection cohomology $\D$-module associated to the set $Z_p$ of matrices of rank $\leq p$.
\begin{theorem}\label{thm:charMatrices}
For $0\leq p\leq n_2$ consider the following set of dominant weights:
$$
W^p=\{\lambda\in \mathbb{Z}^{n_2}_{\dom}\mid \lambda_p\geq p-n_2,\; \lambda_{p+1}\leq p-n_1\},
$$
and for $\lambda\in W^p$ we define
$$
\lambda(p)=(\lambda_1,\cdots,\lambda_p,(p-n_2)^{n_1-n_2},\lambda_{p+1}+(n_1-n_2),\cdots,\lambda_{n_2}+(n_1-n_2)).
$$
Then the $\GL(W_1)\times \GL(W_2)$-character of $L_{Z_p}$ is given by
$$
\bigoplus_{\lambda\in W^p} \bS_{\lambda(p)}W_1\oo\bS_{\lambda}W_2.
$$
\end{theorem}
Using this formula, we can obtain a large amount of information about the $\GL$-characters of $D_0$, $D_2$, $D_3$, $D_4$, $D_6$, $D_8$, $D_9$, though finding the complete $\GL$-character of these is not feasible (see Section \ref{sec:reps}).
In particular, when $n=3$ we have $W_1=A\oo B$, $W_2=C$, and $L_{Z_0}=E=D_0$, $L_{Z_1}=D_2$, $L_{Z_2}=D_6$, $L_{Z_3}=S=D_8$. When $n\geq 4$ we have $W_1=C$, $W_2=A\oo B$, and $L_{Z_0}=D_0$, $L_{Z_1}=D_2$, $L_{Z_2}=D_6$, $L_{Z_3}=D_8$, $L_{Z_4}=D_9=S$. When $W_1=A\oo C$ and $W_2=B$ we have $L_{Z_1}=D_3$, and when $W_1=B\oo C$ and $W_2=A$ we have $L_{Z_1}=D_4$. 

\subsection{Grassmannians, tautological bundles, and Bott's Theorem} Let $W$ be an $m$-dimensional complex vector space. Given $1\leq k\leq m$ we write $\mathbb{G}(k;W)$ for the Grassmannian of $k$-dimensional quotients of $W$, with tautological quotient sheaf $\mc{Q}$ of rank $k$. We store the following consequence of Bott's Theorem (see \cite[Corollary 4.1.9]{weyman}) for use in our calculations.

\begin{lemma}\label{BottLemma}
Let $\lambda\in \mathbb{Z}^k_{\dom}$. We have that $\mathbb{S}_{\lambda}\mathcal{Q}$ has cohomology in at most one degree.
\begin{enumerate}
\item 	If $\lambda_p\geq p-k$ and $\lambda_{p+1}\leq p-m$ for some $0\leq p\leq k$, then 
$$
H^{(m-k)(k-p)}(\mathbb{G}(k;W),\mathbb{S}_{\lambda}\mathcal{Q})\cong \mathbb{S}_{\lambda(p) }W,
$$
where 
$$
\lambda(p)=(\lambda_1,\cdots,\lambda_p, (p-k)^{m-k},\lambda_{p+1}+(m-k),\cdots, \lambda_k+(m-k)).
$$
\item Conversely, if for some $0\leq p\leq k$ we have
$$
H^{(m-k)(k-p)}(\mathbb{G}(k;W),\mathbb{S}_{\lambda}\mathcal{Q})\neq 0,
$$
then $\lambda_p\geq p-k$ and $\lambda_{p+1}\leq p-m$.
\end{enumerate} 
\end{lemma}

\subsection{Spaces of relative hypermatrices}\label{sec:relmatrices} We continue to write $V=A\oo B\oo C$. Fix positive integers $i\leq 2,j \leq 2,k \leq n$, and consider the product of Grassmannians
$$
\mathbb{G}_{ijk}:=\mathbb{G}(i;A)\times \mathbb{G}(j;B)\times \mathbb{G}(k;C),
$$
where $\mathbb{G}(k;W)$ denotes the Grassmannian of $k$-dimensional quotients of a vector space $W$. On $\mathbb{G}_{ijk}$ we consider the geometric vector bundle (see \cite[Exercise II.5.18]{hartshorne}):
$$
Y_{ijk}:=\underline{\tn{Spec}}_{\mc{O}_{\mathbb{G}_{ijk}}} \Sym(\mc{Q}_A\boxtimes \mc{Q}_B\boxtimes \mc{Q}_C),
$$
where $\mc{Q}_W$ denotes the tautological quotient sheaf of rank $k$ on $\mathbb{G}(k;W)$. 

Fix $i,j,k$ and let $Y=Y_{i,j,k}$. We may think of $Y$ as the space of relative $i\times j\times k$ hypermatrices, fibered over $\bb{G}_{i,j,k}$. Indeed, over an open affine subset of $\bb{G}_{i,j,k}$ the space $Y$ is isomorphic to its product with the space of $i\times j\times k$ hypermatrices, and thus $Y$ has stratification by subsets $O_p^{Y}$ as in Section \ref{sec:orbits}. Given such a locally closed subset, we write $D_p^Y$ for the corresponding intersection cohomology $\D_Y$-module. As $Y$ is a closed subvariety of the trivial bundle $V\times \bb{G}_{i,j,k}$, we often abuse notation and identify the modules $D_p^Y$ with sheaves of graded $S^{Y}=\Sym(\mc{Q}_A\boxtimes \mc{Q}_B\boxtimes \mc{Q}_C)$-modules on $\mathbb{G}_{i,j,k}$ (see \cite[Exercise II.5.17]{hartshorne}).

If we write $V_{ijk}=\C^i \oo \C^j \oo \C^k$, then we have an identification
$$
Y_{ijk}\cong \GL\times_{P_{ijk}} V_{ijk},
$$
where the Levi subgroup $\GL_i(\C) \times \GL_j (\C) \times \GL_k(\C)$ of the parabolic $P_{ijk}$ acts on $V_{ijk}=\C^i \oo \C^j \oo \C^k$ naturally, while its unipotent radical acts trivially.

We recall that for positive integers $i\leq 2,j \leq 2,k \leq n$ the subspace variety $\tn{Sub}_{ijk} \subset V$ is the closed subset of tensors $v \in V$ such that there exists subspaces $A' \subset A, B'\subset B, C'\subset C$ of dimensions $i,j,k$ respectively, with $v \in A' \oo B' \oo C'$. Such spaces admit desingularizations as follows:

\begin{lemma}\label{lem:desing}
We have a $\GL$-equivariant resolution of singularities of $\tn{Sub}_{ijk}$ given by
\[\pi_{ijk}: Y_{ijk} \to \tn{Sub}_{ijk},\]
\end{lemma}

We store the following consequence for use in our local cohomology calculations.

\begin{lemma}\label{pushdownlocallyclosed1}
Using notation as above, let $(Y,\pi,O)$ be a triple such that $\pi:Y\to V$ is a desingularization of $\ol{O}$, and assume that $Z:=Y\setminus O$ is a divisor. Let $\mathbb{G}$ be the product of Grassmannians on which $Y$ is a vector bundle. If we set $c=\operatorname{codim}(O,V)$, then we have the following for all $q\geq 0$:
\begin{equation}\label{pushdownlocallyclosed}
H^q_{O}(S)=\mathcal{H}^{q-c}(\pi_{+}\mathcal{O}_Y(\ast Z)).
\end{equation}
In particular, 
\begin{enumerate}
\item $H^q_{O}(S)\neq 0$ only if $c\leq q\leq  c+\dim \bb{G}$, 
\item $\mathcal{H}^{q}(\pi_{+}\mathcal{O}_Y(\ast Z))	\neq 0$	only if $0\leq q\leq \dim \bb{G}$.
\end{enumerate}

\end{lemma}

In this article, we apply Lemma \ref{pushdownlocallyclosed1} to the triples $(Y_{111},\pi_{111},O_1)$ and $(Y_{222},\pi_{222},O_6)$.

\begin{proof}
Let $j:O\to Y$ and $i:O\to V$ be the natural inclusions, so that $\pi\circ j=i$. We observe that
$$
i_{+}(\mathcal{O}_{O})=\pi_{+}(j_{+}(\mathcal{O}_{O}))=\pi_{+}(\mathcal{O}_Y(\ast Z)).
$$
On the other hand, let $U=V\setminus (\ol{O}\setminus O)$ with open immersion $k:U\to V$. Since $O$ is a smooth closed subvariety of $U$ of codimension $c$, we have (see \cite[Proposition 1.7.1]{htt}):
\begin{equation}\label{lcO6push}
i_{+}(\mathcal{O}_{O})=k_{+}(\mathbb{R}\mathscr{H}^0_{O}(\mc{O}_U)[c])=\mathbb{R}\Gamma_{O}(S)[c],
\end{equation}
where the last equality follows from \cite[Proposition 1.3, Proposition 1.4]{hartshornelocal}.
In particular, (\ref{pushdownlocallyclosed}) holds. Thus, the lower bounds (resp. upper bounds) on $q$ in parts (1) and (2) are equivalent. The lower bound on $q$ in (1) follows from (\ref{lcO6push}), and the upper bound on $q$ in (2) follows from \cite[Proposition 1.5.28(ii)]{htt}, using that $\pi$ factors through the projection from $V\times \bb{G}$ to $V$. 
\end{proof}

We have the following notation for the alternating sum of the cohomology of $\pi_+\mc{O}_Y(\ast Z)$ in the Grothendieck group of representations of $\GL$:
\begin{equation}\label{defofEuler}
\big[\chi(\pi_+\mc{O}_Y(\ast Z))\big]=\sum_{i\in \mathbb{Z}} (-1)^i \cdot \big[\mc{H}^i(\pi_+\mc{O}_Y(\ast Z))\big].
\end{equation}

In Section \ref{sec:characters}, we use (\ref{defofEuler}) in conjunction with the techniques of \cite[Section 2]{raicu2016characters} to get information about the characters of the simple equivariant $\D_V$-modules. We recall some notation for use in these calculations. Given $m\geq 1$ we write $[m]=\{1,2,\cdots, m\}$, and for $0\leq k\leq m$ we write $\binom{[m]}{k}$ for the set of size $k$ subsets of $[m]$. For example $\binom{[3]}{2}=\{\{1,2\},\{1,3\},\{2,3\}\}$. Given $r\in \mathbb{Z}$ and $I\in \binom{[m]}{k}$ we write $(r^I)\in \mathbb{Z}^m$ for the tuple with $r$ in the $i$-th place for all $i\in I$, and zeros elsewhere. For example, if $m=3$ then $(r^{\{1,3\}})=(r,0,r)$. 

Let $W$ be an $m$-dimensional vector space and let $\lambda\in \mathbb{Z}^m$ be a not necessarily dominant tuple of $m$ integers. We introduce the notation $[\bS_{\lambda}W]$ in the Grothendieck group of $\GL(W)$ representations as follows. Let $\rho=(m-1,\cdots ,1,0)$. If $\lambda+\rho$ has a repeated entry, then we define $[\bS_{\lambda}W]=0$. Otherwise, $\lambda+\rho$ has distinct entries, and we write $\operatorname{sort}(\lambda+\rho)$ for the tuple obtained by arranging the entries of $\lambda+\rho$ into strictly decreasing order, and we let $\sigma \in \mathfrak{S}_m$ be the corresponding permutation. If we set $\tilde{\lambda}=\operatorname{sort}(\lambda+\rho)-\rho$, then we make the following definition (cf. \cite[Corollary 4.1.9]{weyman})
$$
\big[\bS_{\lambda}W\big]=\operatorname{sgn}(\sigma)\cdot \big[\bS_{\tilde{\lambda}}W\big].
$$
We generalize this to the Grothendieck group of $\GL=\GL(A)\times \GL(B)\times \GL(C)$ representations as follows. Let $\rho^2=(1,0)$ and $\rho^n=(n-1,\cdots, 1,0)$, and let $\alpha,\beta\in \mathbb{Z}^2$ and $\gamma \in \mathbb{Z}^m$. If any of the tuples $\rho^2+\alpha$, $\rho^2+\beta$, $\rho^n+\gamma$ has a repeated entry, then we set $[\bS_{\alpha}A\oo \bS_{\beta}B\oo \bS_{\gamma}C]=0$. Otherwise, we let $\sigma_{\alpha}$, $\sigma_{\beta}$, $\sigma_{\gamma}$ be the corresponding permutations, and we define
\begin{equation}\label{eq:bottGG}
\big[\bS_{\alpha}A\oo \bS_{\beta}B\oo \bS_{\gamma}C\big]=\operatorname{sgn}(\sigma_{\alpha})\cdot \operatorname{sgn}(\sigma_{\beta})\cdot \operatorname{sgn}(\sigma_{\gamma})\cdot \big[\bS_{\tilde{\alpha}}A\oo \bS_{\tilde{\beta}}B\oo \bS_{\tilde{\gamma}}C\big].
\end{equation}
Given an $m$-dimensional vector space $W$ and $r,k\geq 0$, we write (see \cite[Section 2.1.3, Lemma 2.5]{raicu2016characters})
\begin{equation}\label{eq:pkr}
\big[p_{k,r}(W)\big]=\sum_{I\in \binom{[m]}{k}} \big[\bS_{(r^I)}W\big].
\end{equation}
For $\lambda \in \mathbb{Z}^m$, $I\in \binom{[m]}{k}$ and $r\geq 0$ we define the tuple $(\lambda, r,I)\in\bb{Z}^m$ via (cf. \cite[Lemma 2.3]{raicu2016characters}).
\begin{equation}\label{eq:lambdaplusrI}
(\lambda, r,I)=\lambda+(r^I)
\end{equation}
For example, if $m=3$, $\lambda=(-1,-1,-2)$, $k=2$, $I=\{1,3\}$, and $r\geq 3$ then $[\bS_{(\lambda, r,I)}W]=-[\bS_{(r-1,r-3,0)}W]$.

\section{First steps towards the quiver structure of $\tn{mod}_{\GL}(\D_V)$} \label{sec:steps}

\subsection{Classification of simple objects}

First, we describe the component groups $H/H^0$ (here $H^0$ is the connected component of $H$ containing the identity) of all the stabilizers of orbits $O_i = \GL/H$.

\begin{prop}\label{componentGrp}
The component group of the $\GL$-orbit $O_6$ is $\mathbb{Z}/2\mathbb{Z}$, while all other $\GL$-orbits $O_i$ ($i\neq 6$) have connected stabilizers.
\end{prop}

\begin{proof}
We use notation as in Section \ref{sec:relmatrices}. First, assume that for a $\GL$-orbit $O \subset V$ we have $\ol{O} = \tn{Sub}_{ijk}$ for some $i,j,k \geq 1$. By Lemma \ref{lem:desing} $O \cong \pi^{-1}_{ijk}(O) = \GL \times_{P_{ijk}} O' \subset \GL\times_{P_{ijk}} V_{ijk}$, for some $P_{ijk}$-orbit $O' \subset V_{ijk}$. Since the unipotent radical of $P_{ijk}$ acts trivially, $O' $ is a $\GL_i(\C) \times \GL_j (\C) \times \GL_k(\C)$-orbit in $\C^i \oo \C^j \oo \C^k$  (in fact, the dense orbit). 

The generic $\GL$-stabilizer of $O$ agrees with the generic $P_{ijk}$-stabilizer of $O'$. Since the unipotent radical of $P_{ijk}$ acts trivially on $O'$, this shows that the $\GL$-component group of $O$ agrees with the  $\GL_i(\C) \times \GL_j (\C) \times \GL_k(\C)$-component group of $O' \subset \C^i \oo \C^j \oo \C^k$.

When one of $i,j,k$ is $\leq 1$, the above argument gives the result for $O_1, O_2, O_3, O_4$, as it is reduced to the case of matrices, where the corresponding claim is easy to see. 

As $\ol{O}_6 = \tn{Sub}_{222}$, for $O=O_6$ the result follows \cite[Lemma 3.3]{perlman2020equivariant} or \cite[Remark 4.12]{lHorincz2019categories}. Note that the restriction of $\pi_{222}$ still yields an isomorphism $\pi_{222}^{-1}(O_5) \xrightarrow{\cong} O_5$, hence the argument above follows through and yields the result for $O_5$ using \cite[Lemma 3.2]{perlman2020equivariant}.

Since $\ol{O}_8 = \tn{Sub}_{223}$, by the above argument with $O=O_8$ it is enough to prove the claim in the case $n=3$. Then the orbit $O_8 \subset V$ is dense. There is a castling transform from $V=\mathbb{C}^2\oo \mathbb{C}^2\oo \mathbb{C}^3$ to $\mathbb{C}^2\oo \mathbb{C}^3\oo \mathbb{C}^1$ (see \cite[Section 2]{satokimura}). By \cite[Proposition 7]{satokimura} the generic stabilizer of $V$ is isomorphic to the generic stabilizer of $\mathbb{C}^2\oo \mathbb{C}^3\oo \mathbb{C}^1$, endowed with the natural action of $\GL_2(\mathbb{C})\times \GL_3(\mathbb{C})\times \mathbb{C}^{\ast}$. This space may be identified again with the space of $2\times 3$ matrices, where the corresponding claim is easy to see.

Next, we show that the stabilizer of $O_7$ is connected. Since $\pi_{223}$ yields an isomorphism $\pi_{223}^{-1}(O_7) \xrightarrow{\cong} O_7$, it is enough to consider the case $n=3$ (similar as we did for $O_5$). Then this is a consequence of Lemma \ref{lem:opencat2} below, as there is only one $\GL$-equivariant simple local system on $O_7$.

As $\ol{O}_9 = \tn{Sub}_{224}$, by the above argument with $O=O_9$ we can assume that $n>4$. Then the complement of $O_9$ in $V$ has codimension $\geq 2$, which implies that $O_9$ is simply-connected (see \cite[Remark 4.10]{lHorincz2019categories}), hence its component group must be trivial.
\end{proof}

According to the above result, we denote by $D_6'$ the simple $\D_V$-module corresponding to the non-trivial equivariant local system of rank one on $O_6$. We prove the following:

\begin{prop}\label{prop:fourier}
Let $\mc{F}$ denote the (twisted) Fourier transform on $\tn{mod}_{\GL}(\D_V)$.
\begin{enumerate}
\item If $n=3$ then $\mc{F}$ swaps the modules in each of the pairs
$$
(D_0,D_8),\quad (D_1,D_7),\quad (D_2,D_6),
$$
and and all other simples are fixed.
\item If $n\geq 4$ then $\mc{F}$ swaps the modules in each of the pairs
$$
(D_0,D_9),\quad (D_1,D_7),\quad (D_2,D_8), 
$$
and and all other simples are fixed.
\end{enumerate}
	
\end{prop}

\begin{proof} The fact that $\mc{F}$ swaps $D_0$ and $S$ follows immediately from the definition.
We have that $O_1$ and $O_7$ are dual orbits \cite[Table 10.3.1]{landsberg2012tensors}. The characteristic cycle of $D_1$ is of the form $\tn{charC}(D_1)=[\overline{T^{\ast}_{O_1}V}]+m_0[\overline{T^{\ast}_{O_0}V}]$ for some $m_0\geq 0$. By \cite[Equation 4.15]{lHorincz2019categories}, it follows that $\tn{charC}(\mc{F}(D_1))=[\overline{T^{\ast}_{O_7}V}]+m_0[\overline{T^{\ast}_{V}V}]$. Since there is only one simple with full support, namely $S$, it follows that $m_0=0$. As $D_7$ is the only simple with support $\overline{O}_7$, we have that $\mc{F}(D_1)=D_7$ and $\mc{F}(D_7)=D_1$.

Using the two flattenings of $V$ to the space of $2\times 2n$ matrices, the orbit closures $\overline{O}_3$ and $\overline{O}_4$ are identified with the determinantal variety of matrices of rank $\leq 1$. Therefore $\mc{F}(D_3)=D_3$ and $\mc{F}(D_4)=D_4$. 

Assume that $n=3$. Using the flattening of $V$ to the space of $4\times n$ matrices, the orbit closures $\overline{O}_2$, $\overline{O}_6$, are identified with the determinantal varieties of ranks $\leq 1$, $\leq 2$. Thus, $\mc{F}$ swaps the modules $D_2$ and $D_6$.

Now let $n\geq 4$. Using the flattening of $V$ to the space of $4\times n$ matrices, the orbit closures $\overline{O}_2$, $\overline{O}_6$, $\overline{O}_8$, are identified with the determinantal varieties of matrices of rank $\leq 1$, $\leq 2$, $\leq 3$. Thus, the Fourier transform swaps $D_2$ and $D_8$, and it fixes $D_6$. The characteristic cycle of $D_6'$ is of the form $\tn{charC}(D_6')=[\overline{T^{\ast}_{O_6}V}]+\sum_{i=0}^5 m_i[\overline{T^{\ast}_{O_i}V}]$ for some $m_i\geq 0$. Since $O_5$ is self-dual, by \cite[Equation 4.15]{lHorincz2019categories}, it follows that
$$
\tn{charC}(\mc{F}(D_6'))=[\overline{T^{\ast}_{O_6}V}]+m_5[\overline{T^{\ast}_{O_5}V}]+m_4[\overline{T^{\ast}_{O_4}V}]+m_3[\overline{T^{\ast}_{O_3}V}]+m_2[\overline{T^{\ast}_{O_8}V}]+m_1[\overline{T^{\ast}_{O_7}V}]+m_0[\overline{T^{\ast}_{O_9}V}].
$$
Since $O_7, O_8, O_9$ support unique simples, we have that $m_0=m_1=m_2=0$, which shows $\mc{F}(D_6')=D_6'$. By process of elimination, it follows that $\mc{F}$ fixes $D_5$.
\end{proof}

\subsection{Restricting to open subsets}\label{sec:open}

In this subsection we analyze $\D$-modules on some distinguished open sets in $V$. This helps us patch up the quiver of $\tn{mod}_{\GL}(\D_V)$ in Section \ref{sec:eqcat}.

\begin{lemma}\label{lem:opencat1}
Let $U= V\setminus \ol{O}_2$. Then the quiver of $\tn{mod}_{\GL}^{\ol{O}_6 \setminus \ol{O}_2}(\D_U)$ is given by
\[\xymatrix{
(5) \ar@<0.5ex>[r] & (6)  \ar@<0.5ex>[l]  & (3) \ar@<0.5ex>[r] & \ar@<0.2ex>[l] (6') \ar@<0.5ex>[r] & \ar@<0.2ex>[l] (4)
}\]
where all the 2-cycles are zero. Here the vertices of the equivariant simple $\D_U$-modules are labeled as the corresponding simple $\D_V$-modules via restriction.
\end{lemma}
\begin{proof}
Let $\pi=\pi_{222}$ as in Section \ref{sec:relmatrices}. Note that $\ol{O}_6 \setminus \ol{O}_2 \subset U$ is closed and smooth, as it can be identified with the space of $4\times n$ matrices of rank $2$. Furthermore, the map $\pi$ induces an isomorphism $\pi^{-1} (\ol{O}_6 \setminus \ol{O}_2) \xrightarrow{\,\cong\,} \ol{O}_6 \setminus \ol{O}_2$. For $2\leq i \leq 6$, let $O_i' \subset V_{222}$ denote the $\GL_2 \times \GL_2 \times \GL_2$-orbit having the same representative as $O_i$ in Section \ref{sec:orbits}, let $U' = V_{222} \setminus \ol{O}_2'$ and $j': U' \to V_{222}$ be the open embedding, so that $\pi^{-1} (\ol{O}_6 \setminus \ol{O}_2) = \GL\times_{P_{222}} (U')$. By \cite[Proposition 4.5]{lHorincz2019categories} we have equivalences of categories
\begin{equation}\label{eq:catchain}
\tn{mod}_{\GL}^{\ol{O}_6 \setminus \ol{O}_2}(\D_U) \cong \tn{mod}_{\GL}(\D_{\ol{O}_6 \setminus \ol{O}_2}) =  \tn{mod}_{\GL}(\D_{\GL\!\times_{P_{222}} U'}) \cong  \tn{mod}_{P_{222}}(\D_{U'}) = \tn{mod}_{\GL_2 \!\times \! \GL_2 \! \times \! \GL_2}(\D_{U'}),
\end{equation}
where the last equality follows as the unipotent radical of $P_{222}$ is connected and acts trivially on $V_{222}$. By adjunction, $j'_* $ sends the injective envelopes of the equivariant simple $\D_{U'}$-modules with support $\ol{O}'_i \setminus \ol{O}'_2$ to the injective envelopes of the equivariant simple $\D_{V_{222}}$-modules with support $\ol{O}'_i$ (cf. \cite[Lemma 2.4]{binary}), and $j'^*$ sends them back. Therefore, we readily obtain the quiver of $\tn{mod}_{\GL}^{\ol{O}_6 \setminus \ol{O}_2}(\D_U)\cong \tn{mod}_{\GL_2 \!\times \! \GL_2 \! \times \! \GL_2}(\D_{U'}) $ from \cite[Theorem on the Quiver Structure]{perlmancorrected} (see Remark \ref{rem:mikeerror}).
\end{proof}

Below we use the analogous convention as in Lemma \ref{lem:opencat1}  regarding the labeling of vertices.

\begin{lemma}\label{lem:opencat2}
Let $U=V\setminus \ol{O}_6$. The quiver of $\tn{mod}_{\GL}^{\ol{O}_8 \setminus \ol{O}_6}(\D_{U})$ is given by 
\[\xymatrix{ (7) \ar@<0.5ex>[r]  & (8) \ar@<0.5ex>[l]}\]
where all the 2-cycles are zero.
\end{lemma}

\begin{proof}
We note that $\ol{O}_8\setminus \ol{O}_6$ can be identified as the homogeneous space of $4\times n$ matrices of rank $3$. Let $P \subset \GL_4(\C)$ denote the parabolic subgroup given by $g_{4i} = 0$, for $i=1,2,3$. Choosing the representative $\begin{bmatrix} I_3 & 0 \\ 0 & 0 \end{bmatrix}$, we get that $\ol{O}_8\setminus \ol{O}_6$ is $\GL_4(\C)\times \GL_n(\C)$-isomorphic to $\GL_4(\C)\times \GL_n(\C) / H$, where $H$ is the subgroup of $P \times \GL_n(\C)$  consisting of elements of the form
\[ \left\{ \,\left(\begin{bmatrix} X & a \\ 0 & b \end{bmatrix}, \begin{bmatrix} X & 0 \\ Y & Z \end{bmatrix}\right) \, : \, X\in \GL_3(\C), a\in \C^3, b\in \C^*, Y \in 
\C^{(n-3) \times 3}, Z \in \GL_{n-3}(\C) \right\}. \]
Using \cite[Proposition 4.5]{lHorincz2019categories} repeatedly, we have
\[\tn{mod}_{\GL}^{\ol{O}_8 \setminus \ol{O}_6}(\D_{U})\cong\tn{mod}_{\GL}(\D_{\ol{O}_8\setminus \ol{O}_6}) \cong \tn{mod}_{\GL}(\D_{\GL_4(\C)\times \GL_3(\C) / H}) \cong \tn{mod}_{\GL_2\times \GL_2 \times \GL_n \times H}(\D_{\GL_4(\C)\times \GL_n(\C)}) \cong \]
\[\cong \tn{mod}_{\GL_2\times \GL_2 \times P}(\D_{\GL_4(\C)}) \cong \tn{mod}_{\GL_2\times \GL_2}(\D_{\GL_4(\C)/P}) \cong \tn{mod}_{\GL_2\times \GL_2}(\D_{\mathbb{P}(\bigwedge^3 \C^4)}) \cong\]
\[ \cong \tn{mod}_{\GL_2\times \GL_2}(\D_{\mathbb{P} (\C^2\oo \C^2)})  \cong \tn{mod}_{\GL_2\times \GL_2 \times \C^*}(\D_{\C^2 \otimes \C^2 \setminus \{0\}})  \cong \tn{mod}_{\GL_2\times \GL_2}(\D_{\C^2 \otimes \C^2 \setminus \{0\}}),\]
where the last equivalence follows as the normal subgroup consisting of elements $(a I_2, a I_2, a^{-1}) \subset \GL_2(\C)\times \GL_2(\C) \times \C^*$ (with $a\in \C^*$) acts trivially on $\C^2 \oo \C^2$. Under the action of $\GL_2(\C)\times \GL_2(\C)$, we can identify $\C^2 \otimes \C^2$ with the space of $2\times 2$ matrices, in which case the category is described in \cite[Theorem 5.4]{lHorincz2019categories}. The statement now follows since the injective envelopes on the open $(\C^2 \oo \C^2) \setminus \{0\}$ are restrictions of the corresponding injective envelopes on the $\C^2 \oo \C^2$ (cf. proof of Lemma \ref{lem:opencat1}).
\end{proof} 

\begin{lemma}\label{lem:noD34}
For any $i\neq 2n-3$, the modules $D_3$ and $D_4$ are not composition factors of $H_{\ol{O}_5}^i(S)$.
\end{lemma}

\begin{proof}
We use the notation as in Lemma \ref{lem:opencat1} and its proof. For $k=5,6$, put $Z_k=\ol{O}_k \setminus \ol{O}_2$. Its enough to prove the statement on the open set $U=V\setminus \ol{O}_2$, i.e. that the simple $\mc{D}_U$-modules $\mc{D}_3$ and $\mc{D}_4$ are not composition factors of $\mathscr{H}^i_{Z_5}(\mc{O}_U)$ for any $i\neq 2n-3$.

Since the variety $O_6'$ is affine, the inclusion of $\GL/{P_{222}}$-bundles $\GL\times_{P_{222}} O_6' \to \GL\times_{P_{222}} U'$ is an affine morphism. Applying $\pi_{222}$, we see that the open embedding $O_6 \to Z_6$ is an affine morphism. As the inclusion $j: O_6 \to U$ is the composition of the latter and the closed embedding $Z_6 \to U$, this implies that $j_+(\mc{O}_{O_6})$ has cohomology only in degree zero. As $j_{+}(\mathcal{O}_{O_6})=\mathbb{R}\mathscr{H}^0_{O_6}(\mc{O}_U)[2n-4]$ (see (\ref{lcO6push})), we obtain that $\mathscr{H}^k_{O_6}(\mc{O}_U)=0$ for all $k\neq 2n-4$. Since $Z_6$ is smooth, we have $\mathscr{H}^k_{Z_6}(\mc{O}_U)=0$ for  $k\neq 2n-4$.  Hence, the long exact sequence
\[\cdots \to \mathscr{H}^{i-1}_{O_6}(\mc{O}_U) \to\mathscr{H}^i_{Z_5}(\mc{O}_U) \to \mathscr{H}^i_{Z_6}(\mc{O}_U) \to  \mathscr{H}^i_{O_6}(\mc{O}_U) \to \mathscr{H}^{i+1}_{Z_5}(\mc{O}_U) \to \cdots\]
yields $\mathscr{H}^i_{Z_5}(\mc{O}_U)=0$ for all $i\neq 2n-3$.
\end{proof}

\subsection{Conormal bundles}

We first determine the character cycles of all the simples in $\tn{mod}_{\GL}(\D_V)$. As the characteristic cycles are known in the determinantal case (e.g. \cite[Remark 1.5]{raicu2016characters}), we have $ \tn{charC}(D_i)=[\overline{T^{\ast}_{O_i}V}]$ for $i=0, 2, 3, 4, 6, 8, 9$.

\begin{prop}\label{prop:char566}We have
\[\tn{charC}(D_1)=[\overline{T^{\ast}_{O_1}V}], \quad \tn{charC}(D_7)=[\overline{T^{\ast}_{O_7}V}],\]
\[\tn{charC}(D_5)=[\overline{T^{\ast}_{O_5}V}]+[\overline{T^{\ast}_{O_4}V}]+[\overline{T^{\ast}_{O_3}V}],\]
\[\qquad \mbox{for } n=3: \,\, \tn{charC}(D_6')=[\overline{T^{\ast}_{O_6}V}]+[\overline{T^{\ast}_{O_5}V}]+[\overline{T^{\ast}_{O_2}V}], \]
\[\mbox{for } n\geq 4: \,\, \tn{charC}(D_6')=[\overline{T^{\ast}_{O_6}V}]+[\overline{T^{\ast}_{O_5}V}]. \qquad\]
\end{prop}

\begin{proof}
The claim about $D_1$ and $D_7$ follows from the proof of Proposition \ref{prop:fourier}. As both $D_5$ and $D_6'$ are fixed under $\mc{F}$, it follows for $n\geq 4$ as in the proof of Proposition \ref{prop:fourier} that $[\overline{T^{\ast}_{O_i}V}]$ does not appear in the their characteristic cycle for $i=0,1,2$. When $n=3$, the same argument works for $D_5$, and for $D_6'$ we get that $[\overline{T^{\ast}_{O_i}V}]$ does not appear in its characteristic cycle for $i=0,1$, while $[\overline{T^{\ast}_{O_2}V}]$ does so with multiplicity one, since $\ol{O}_2$ is the projective dual of $\ol{O}_6$ and $\mc{F}(D_6')\cong D_6'$. Thus, to finish the proof, we can now restrict to the open set $U=V\setminus \ol{O}_2$. Since the chain of equivalences in (\ref{eq:catchain}) preserves respective characteristic cycles, the statement follows from the corresponding result for $n=2$ in \cite[Theorem 3.4]{perlmancorrected} (see Remark \ref{rem:mikeerror}).
\end{proof}

\begin{remark}\label{rem:mikeerror}
There is an error in the statement of \cite[Theorem on the Quiver Structure]{perlman2020equivariant}, which has been corrected in \cite{perlmancorrected}. The issue is with the asserted relations of the quiver, namely with the relations $\beta_{p,q,r}\alpha_{i,j,k}$ and $\delta_{p,q,r}\gamma_{i,j,k}$ for $(p,q,r)\neq (i,j,k)$. These are in fact not relations of the quiver, and should be replaced with the following new relations for each $(p,q,r)\neq (i,j,k)$: $\beta_{p,q,r}\alpha_{i,j,k}-\delta_{p,q,r}\gamma_{i,j,k}$. The newer article \cite{perlmancorrected} also has the characteristic cycle calculations cited in the proof of Proposition \ref{prop:char566}.
\end{remark}

For the next two results below we identify the $\GL$-action on $V$ with action of $H=\op{SO}_4(\C) \times \GL_n(\C)$ on the space of $4\times n$ matrices, and the cotangent bundle $V \times V^*$ of $V$ with the space pairs of matrices $(X, Y)$, where $X$ is $4\times n$ and $Y$ is $n\times 4$ with the natural actions of $\op{SO}_4(\C) \times \GL_n(\C)$. Then the equations $YX = 0$ and $(XY)^t = XY$ define (set-theoretically)  $\bigcup_{i=0}^9  \ol{T_{O_i}^* V}$.

\begin{lemma}\label{lem:dense}
The varieties $\ol{T_{O_6}^* V}$ and $\ol{T_{O_8}^* V}$ have a dense $\GL$-orbit.
\end{lemma}

\begin{proof}
By a straightforward calculation, we see that the $\op{Lie}(H)$-stabilizers of the points 
\[\left(\begin{bmatrix} 1 & 0 & 0 & 0 & 0 &\cdots & 0 \\ 0 & 1 & 0 & 0 & 0 &\cdots & 0 \\ 0& 0 & 0 & 0 & 0 &\cdots & 0 \\ 0 & 0 & 0 & 0 & 0 &\cdots & 0 \end{bmatrix}, \begin{bmatrix} 0 & 0 & 0 & 0 & 0 &\cdots & 0 \\ 0 & 0 & 0 & 0 & 0 &\cdots & 0 \\ 0 & 0 & 1 & 0 & 0 &\cdots & 0 \\ 0 & 0 & 0 & 1 & 0 &\cdots & 0  \end{bmatrix}^t\right) \in \ol{T_{O_6}^* V}\] and when $n\geq 4$
\[\left(\begin{bmatrix} 1 & 0 & 0 & 0 & 0 &\cdots & 0\\ 0 & 1 & 0 & 0 & 0 &\cdots & 0 \\ 0& 0 & 1 & 0 & 0 &\cdots & 0 \\ 0 & 0 & 0 & 0 & 0 &\cdots & 0 \end{bmatrix}, \begin{bmatrix} 0 & 0 & 0 & 0 & 0 &\cdots & 0 \\ 0 & 0 & 0 & 0 & 0 &\cdots & 0 \\ 0 & 0 & 0 & 0 & 0 &\cdots & 0 \\ 0 & 0 & 0 & 1 & 0 &\cdots & 0 \end{bmatrix}^t\right) \in \ol{T_{O_8}^* V},\] have dimension $n^2-4n+6$, thus their $H$-orbits are dense in $\ol{T_{O_6}^* V}$ and $\ol{T_{O_8}^* V}$, respectively.
\end{proof}

\begin{lemma}\label{lem:codim1}
For $n\geq 4$, we have $\dim \ol{T_{O_7}^* V} \cap \ol{T_{O_i}^* V} < \dim V - 1$ for $i=1,3,4,5,6$.
\end{lemma}

\begin{proof}
Since $\ol{O}_1$ is the projective dual of $\ol{O}_7$, the variety $\ol{T_{O_7}^* V}$ is defined (set-theoretically) by the equations 
\[YX = 0, \, (XY)^t = XY, \, \op{rank} X \leq 3, \, \op{rank} (X^t X)\leq 2, \, \op{rank} Y \leq 1, \, Y Y^t = 0.\]

First, we prove the statement for $i=6$. As $\ol{O}_6 = Z_2$, the variety $\ol{T_{O_6}^* V}$ is defined by $YX = 0, XY=0, \op{rank} X \leq 2, \op{rank} Y \leq 2$ \cite{strickland}. Thus, the variety  $\ol{T_{O_6}^* V} \cap \ol{T_{O_7}^* V}$ is a subset of the variety $T$ defined by 
\[YX = 0, \, XY=0, \, \op{rank} X \leq 2, \,  \op{rank} Y \leq 1,\]
and this inclusion is strict since it is easy to see that $Y Y^t$ is not identically zero on $T$. Thus, it is enough to show that $T$ is irreducible and has dimension $\leq 4n-1$.

Using the terminology from \cite{node}, the variety $T$ can be realized as a rank variety on the quiver
\[\xymatrix{ 4 \ar@<0.5ex>[r]^Y  & n \ar@<0.5ex>[l]^X}\]
where the two vertices are nodes. Thus, $T$ is a representation variety of a radical square algebra. By \cite[Theorem 3.19]{node} and its proof, $T$ is irreducible and has desingularization by a total space of a vector bundle over (a product of) Grassmannians. This bundle has dimension $3n+3 \leq 4n-1$, proving the claim.

Next, we prove the cases $i=1,3,4,5$ simultaneously. Note that all $\ol{T_{O_7}^* V} \cap \ol{T_{O_i}^* V}$ (for  $i=1,3,4,5$)  are contained in the variety $T_1$ defined by
\[YX = 0, \, (XY)^t = XY, \, \op{rank} X \leq 2, \, \op{rank} (X^t X)\leq 1, \, \op{rank} Y \leq 1.\]
Thus, it is enough to show that $\dim T_1 \leq 4n - 2$. When $n=4$ and $n=5$, this is done by a computer calculation using Macaulay2 \cite{M2}, when the dimensions are 14 and 17, respectively.

Now assume $n\geq 6$. Clearly, $T_1$ is a subvariety of the variety $T_2$ defined by
\[ YX = 0, \, \op{rank} X \leq 2, \, \op{rank} Y \leq 1.\]
and this inclusion is strict. Thus, to finish to proof, it suffices to prove that $T_2$ is irreducible and $\dim T_2 \leq 4n-1$. Again, using the terminology from \cite{node}, the variety $T_2$ can be realized as a rank variety on the quiver
\[\xymatrix{ n \ar[r]^X & 4 \ar[r]^Y & n }\]
where the middle vertex is a node -- thus, $T_2$ is the representation variety of a radical square algebra.  By \cite[Theorem 3.19]{node} and its proof, $T_2$ is irreducible and has a desingularization that is the total space of a vector bundle over (a product of) Grassmannians, the dimension of which is $3n+5 \leq 4n-1$.
\end{proof}

\subsection{Preliminaries for $n=3$} We let $n=3$ and let $f\in S$ be the semi-invariant, which is the defining equation of $\overline{O}_7$ and has weight $(3,3)\times (3,3)\times (2,2,2)$. By \cite[Lemma 2.4]{binary} the $\D$-module $S_f$ is the injective envelope of $S$ in $\tn{mod}_{\tn{GL}}(\D_V)$.

We recall that the representation $V$ of $\GL$ is equivalent to the representation of $\tn{SO}_{4}(\C) \times \GL_3(\C)$ acting on the space of $4\times 3$ matrices. If $X$ denotes a generic $4\times 3$ matrix of variables, then $f=\det(X^t \cdot X)$. Thus, the Bernstein--Sato polynomial $b_f(s)$ of $f$ is given by \cite[Example 2.9]{decomp} (cf. also \cite[Example 9.2]{microlocal}).
\begin{equation}\label{eq:bs3}
b_f(s)= (s+1)^2 (s+3/2)^2 (s+2)^2.
\end{equation}

We need also the following equation for a local $b$-function of $f$, which follows again from \cite[Example 2.9]{decomp}. Let $f_1$ be the $(1,1)$-entry of $X^t \cdot X$, and note that the $\GL$-translates of $f_1$ generate the defining ideal of the highest weight orbit $\ol{O}_1$. Below, the operator $Q$ can be chosen (up to a scalar) by replacing the variables with the corresponding partial derivatives in a $2\times 2$ minor of $X^t \cdot X$.
\begin{lemma}\label{lem:loc3}
There exists $Q \in \D_V$ such that
\[Q \cdot f^{s+1} = (s+1)^2 (s+3/2)^2 \cdot f_{1} \cdot f^s.\]
\end{lemma}

\begin{lemma}\label{lem:witD6}
The $\GL$-representation $\bS_{(-2,-2)}A\oo \bS_{(-2,-2)}B\oo \bS_{(-1,-1,-2)}C$ has multiplicity one in $D_6$.
\end{lemma}

\begin{proof}
Using the flattening of $V$ to the space of $4\times 3$ matrices, our simple $D_6$ is identified with $L_{Z_2}$ supported on the determinantal variety $Z_2$. By Theorem \ref{thm:charMatrices} the $\GL(A\oo B)\times \GL(C)$ representation $\bS_{(-1,-1,-1,-1)}(A\oo B)\oo \bS_{(-1,-1,-2)}C$ belongs to $D_6$, and is the whole $\GL(C)$-isotypic component of $D_6$ corresponding to $\bS_{(-1,-1,-2)}C$. The former representation is isomorphic to $\bS_{(-2,-2)}A\oo \bS_{(-2,-2)}B\oo \bS_{(-1,-1,-2)}C$ as a representation of $\GL$.
\end{proof}

\begin{lemma}\label{charSfn3}
The following is true about subrepresentations of $S_f$:
\begin{enumerate}
\item Representations of the form $\bS_{\alpha}A\oo \bS_{\alpha}B\oo \bS_{(2t,2t,2t)}C$ ($t\in \mathbb{Z}$) appear in $S_f$ with multiplicity one.
\item Representations of the form $\bS_{\alpha}A\oo \bS_{\beta}B\oo \bS_{(2t,2t,2t)}C$ ($t\in \mathbb{Z}$) with $\alpha\neq \beta$ do not appear in $S_f$.
\item Let $\alpha,\beta \in \mathbb{Z}^2_{\dom}$ with $|\alpha|=|\beta|=-4$, $\alpha_1+\beta_1\leq -2$, and $|\alpha_1-\beta_1|\leq 1$. Set $\phi=\min\{\alpha_1,\beta_1\}$. The multiplicity of a representation of the form $\bS_{\alpha}A\oo\bS_{\beta}B\oo \bS_{(-1,-1,-2)}C$ in $S_f$ is $\phi+3$ if $\alpha_1+\beta_1$ is odd, and it is $\phi+2$ if $\alpha_1+\beta_1$ is even.
\item The following representations appear in $S_f$ with multiplicity one:
$$
\bS_{(-1,-3)}A\oo \bS_{(-1,-3)}B\oo \bS_{(-1,-1,-2)}C,\quad \bS_{(-1,-3)}A\oo \bS_{(-2,-2)}B\oo \bS_{(-1,-1,-2)}C.
$$
\item The following representations do not appear in $S_f$:
$$
\bS_{(-2,-2)}A\oo \bS_{(-2,-2)}B\oo \bS_{(-1,-1,-2)}C,\quad \bS_{(-4,-4)}A\oo \bS_{(-4,-4)}B\oo \bS_{(-2,-3,-3)}C,
$$
$$
\bS_{(0,-4)}A\oo \bS_{(-2,-2)}B\oo \bS_{(-1,-1,-2)}C.
$$
\end{enumerate}
\end{lemma}

\begin{proof} We prove (1) and (2) simultaneously.  Let $\alpha,\beta \in \mathbb{Z}^2_{\dom}$, and let $t\in \mathbb{Z}$. Since $f$ has weight $(3,3)\times (3,3)\times (2,2,2)$, the multiplicity of $\bS_{\alpha}A\oo \bS_{\beta}B\oo \bS_{(2t,2t,2t)}C$ in $S_f$ is equal to the multiplicity of $\bS_{(3k+\alpha_1,3k+\alpha_2)}A\oo \bS_{(3k+\beta_1,3k+\beta_2)}B\oo \bS_{(2k+2t,2k+2t,2k+2t)}C$ in $S$ for $k\gg 0$. Using the Cauchy Formula (see Section \ref{sec:reps}) applied to $S=\tn{Sym}((A\oo B)\oo C)$, if $\bS_{(3k+\alpha_1,3k+\alpha_2)}A\oo \bS_{(3k+\beta_1,3k+\beta_2)}B\oo \bS_{(2k+2t,2k+2t,2k+2t)}C$ belonged to $S$, it would have to be a subrepresentation of $\bS_{(2k+2t,2k+2t,2k+2t)}(A\oo B)\oo \bS_{(2k+2t,2k+2t,2k+2t)}C$. Twisting by $\det(A\oo B)^{\oo(-2k-2t)}$ and taking duals, we are interested in the multiplicity of $\bS_{(k+4t-\alpha_2,k+4t-\alpha_1)}A^{\ast}\oo \bS_{(k+4t-\beta_2,k+4t-\beta_1)}B^{\ast}$ in $\bS_{(2k+2t,0,0,0)}(A^{\ast}\oo B^{\ast})$ for $k\gg 0$. By the Cauchy formula, this is one if and only if $\alpha=\beta$.

(3) We want the multiplicity of $\bS_{(3k+\alpha_1,3k+\alpha_2)}A\oo \bS_{(3k+\beta_1,3k+\beta_2)}B\oo \bS_{(2k-1,2k-1,2k-2)}C$ in $S$ for $k\gg 0$. Using the Cauchy formula applied to $S=\tn{Sym}((A\oo B)\oo C)$, we are interested in the multiplicity of $\bS_{(3k+\alpha_1,3k+\alpha_2)}A\oo \bS_{(3k+\beta_1,3k+\beta_2)}B$ in $\bS_{(2k-1,2k-1,2k-2)}(A\oo B)$ for $k\gg 0$. Dualizing and twisting by $\det(A\oo B)^{\otimes 2k-1}$, this is the same as the multiplicity of $\bS_{(k-\alpha_2-2,k-\alpha_1-2)}A^{\ast}\oo \bS_{(k-\beta_2-2,k-\beta_1-2)}B^{\ast}$ in $\bS_{(2k-1,1,0,0)}(A^{\ast}\oo B^{\ast})$.

Let $\phi$ be as in the statement of the lemma. We apply \cite[Corollary 4.3b]{raicu2012secant}. Using the notation there, for $k\gg 0$ we have $f=k-\phi-2$, $e=2k-\alpha_1-\beta_1-3$ and $r=2k$. Since $\alpha_1+\beta_1\leq -2$ we have $e\geq r-1$, and since $|\alpha_1-\beta_1|\leq 1$ we have $e\geq 2f$. We have $e$ is even if and only if $\alpha_1+\beta_1$ is odd, in which case the multiplicity is $\phi+3$. We have $e$ is odd if and only if $\alpha_1+\beta_1$ is even, in which case the multiplicity is $\phi+2$.

(4) This follows from (3).

(5) By (3) we have that $\bS_{(-2,-2)}A\oo \bS_{(-2,-2)}B\oo \bS_{(-1,-1,-2)}C$ does not appear in $S_f$. Applying the Fourier transform, $\bS_{(-4,-4)}A\oo \bS_{(-4,-4)}B\oo \bS_{(-2,-3,-3)}C$ also does not appear. To determine the multiplicity of $\bS_{(0,-4)}A\oo \bS_{(-2,-2)}B\oo \bS_{(-1,-1,-2)}C$, by reasoning of the second paragraph, we are interested in the multiplicity of $\bS_{(k+2,k-2)}A^{\ast}\oo \bS_{(k,k)}B^{\ast}$ in $\bS_{(2k-1,1,0,0)}(A^{\ast}\oo B^{\ast})$. Here, $r=2k$, $f=k$ and $e=2k-1$. Since $e<2f$ it follows from \cite[Corollary 4.3b]{raicu2012secant} that the multiplicity is zero.
\end{proof}

As a consequence of Lemma \ref{lem:witD6} and Lemma \ref{charSfn3} we have the following.
\begin{lemma}\label{D6S0fn3}
The simple $D_6$ is not a composition factor of $S_f$. 
\end{lemma}

Let $P'_2= \D_V \otimes_{U(\mathfrak{gl})} (\bS_{(-4,-4)}A\oo \bS_{(-4,-4)}B\oo \bS_{(-2,-3,-3)}C)$ be the projective equivariant $\D_V$-module, as constructed in \cite[Section 2.1]{lHorincz2019categories}.
\begin{lemma}\label{lem:P2char}
The support of $P'_2$ is contained in $\ol{O}_5$, and the multiplicity of $\ol{T_{O_2}^* V}$ in $\charC(P_2')$ is one.
\end{lemma}
\begin{proof}
The proof of the first part is similar to that of \cite[Lemma 5.14]{lHorincz2019categories}. We again identify $V$ with the space of $4\times 3$ matrices with the action of $\tn{SO}_{4}(\C) \times \GL_3(\C)$. Since $P'_2$ has an explicit presentation, we implemented it into the computer algebra system Macaulay2 \cite{M2}. By computing a partial Gr\"obner basis, we obtained the following generator of the characteristic ideal of $P'_2$ in $\C[V \times V^*]$ (where $x_{ij}$ are the coordinates on $V$, viewed as the space of $4\times 3$ matrices):
\[
x_{12}^2x_{21}^2-2x_{11}x_{12}x_{21}x_{22}+x_{11}^2x_{22}^2+x_{12}^2x_{31}^2+x_{22}^2x_{31}^2-2x_{11}x_{12}x_{31}x_{32}-2x_{21}x_{22}x_{31}x_{32}+x_{11}^2x_{32}^2+x_{21}^2x_{32}^2+x_{12}^2x_{41}^2+\]
\[+x_{22}^2x_{41}^2+x_{32}^2x_{41}^2-2x_{11}x_{12}x_{41}x_{42}-2x_{21}x_{22}x_{41}x_{42}-2x_{31}x_{32}x_{41}x_{42}+x_{11}^2x_{42}^2+x_{21}^2x_{42}^2+x_{31}^2x_{42}^2.\]
Viewed as an element in $\C[V]$, this it is non-zero when evaluated on  $\begin{bmatrix} 1 & 0 & 0 \\ 0 & 1 & 0 \\ 0 & 0 & 0 \\ 0 & 0 & 0 \end{bmatrix}  \in O_6$.

For the second part, we first claim that $\ol{T_{O_2}^* V}$ has a dense $P=\op{SO}_4(\C)\times P'$-orbit, where $P'\subset \GL_3(\C)$ is the parabolic given by $g_{21}=g_{31}=0$. This follows readily by explicitly computing the $\mathfrak{p}=\tn{Lie}(P)$-stabilizer of the point $\left(\begin{bmatrix} 1 & 0 & 0 \\ 0 & 0 & 0 \\ 0 & 0 & 0 \\ 0 & 0 & 0 \end{bmatrix}, \begin{bmatrix} 0 & 0 & 0 \\ 0 & 1 & 0 \\ 0 & 0 & 1 \\ 0 & 0 & 0 \end{bmatrix}\right) \in \ol{T_{O_2}^* V}$, which is seen to be 1-dimensional. 

Since the highest weight vector in the representation $\bS_{(-4,-4)}A\oo \bS_{(-4,-4)}B\oo \bS_{(-2,-3,-3)}C$ is $P$-semi-invariant, the claim now follows analogously to \cite[Lemma 3.12 and Proposition 3.14]{lHorincz2019categories} 
\end{proof}

\subsection{Extensions between $D_0$, $D_1$, $D_2$, $D_3$, $D_4$}

We start by determining the local cohomology of the polynomial ring $S$ with support in $\ol{O}_1$. This result will help determine extensions between $D_1$ and $D_0$.

For the following statement, note that the codimension of $\ol{O}_1$ is $3n-2$, which is equal to $4n-5$ when $n=3$.

\begin{theorem}\label{BettiThm}
The following is true about local cohomology of the polynomial ring $S$ with support in $\ol{O}_1$ (all short exact sequences are non-split).

\begin{enumerate}	
\item If $n=3$ then the only non-zero local cohomology modules are 
$$
0 \to D_1 \to H^{3n-2}_{\ol{O}_1}(S)\to D_0\to 0,\quad H^{4n-3}_{\ol{O}_1}(S)=D_0^{\oplus 2}.
$$
\item If $n\geq 4$ then the only non-zero local cohomology modules are
$$
H^{3n-2}_{\ol{O}_1}(S)=D_1,\quad H^{4n-5}_{\ol{O}_1}(S)=D_0,\quad H^{4n-3}_{\ol{O}_1}(S)=D_0^{\oplus 2}.
$$
\end{enumerate}
	
\end{theorem}

\begin{proof}
For ease of notation, set $d=\dim S=4n$, and $c=\codim(\ol{O}_1, V)=3n-2$. The orbit closure $\ol{O}_1$ is the affine cone over the Segre product $X=\mathbb{P}^1\times \mathbb{P}^1\times \mathbb{P}^{n-1}$, a smooth projective variety. It follows from \cite[Theorem 4.8, Theorem 6.4]{hartshorne2018simple} (see also \cite[Main Theorem 1.2]{switala}, \cite[Theorem]{garsab}, and \cite[Theorem 3.1]{LSW}) that the multiplicity of $D_0$ in $H^{\bullet}_{\ol{O}_1}(S)$ is determined by the de Rham Betti numbers $b_k=\dim H^k_{DR}(X)$ of $X$. More precisely, for $c<j<d-1$ we have $H^j_{\ol{O}_1}(S)=D_0^{\oplus (b_{d-j-1}-b_{d-j-3})}$, and there is a non-split short exact sequence
\begin{equation}\label{eqn:EsinO1}
0\to D_1 \to H^{c}_{\ol{O}_1}(S)\to D_0^{\oplus (b_{n+1}-b_{n-1})}\to 0.
\end{equation}
Furthermore, $H^j_{\ol{O}_1}(S)=0$ for $j\geq d-1$. By the K\"{u}nneth formula, the Poincar\'{e} polynomial of $X$ is
$$
P_X(q)=\binom{2}{1}_{q^2}\cdot \binom{2}{1}_{q^2}\cdot \binom{n}{1}_{q^2}=(1+q^2)^2\cdot \sum_{i=0}^{n-1} q^{2i}=(1+2q^2+q^4)\cdot \sum_{i=0}^{n-1}q^{2i},
$$
so $b_k$ is the coefficient of $q^k$ in $P_X(q)$.

We first consider $H^c_{\ol{O}_1}(S)$. Notice that $P_X(q)$ is supported in even degrees, so
that if $n$ is even, then $b_{n+1}= b_{n-1} = 0$, and the claim follows from (\ref{eqn:EsinO1}). For $n$ odd, the coefficient of $q^{n+1}$ in $P_X(q)$
is $4$, so $b_d = 4$. If $n = 3$, the coefficient of $q^{n-1}$ in $P_X(q)$ is $3$, whereas if $n\geq 5$, the coefficient of
$q^{n-1}$ in $P_X(q)$ is $4$. Therefore, the difference $b_{n+1}- b_{n-1}$ is as claimed.

Now we consider cohomological degrees $c<j<d-1$. The values of the desired differences of Betti numbers are described by the following table (if $j$ is even then $b_{d-j-1}=b_{d-j-3}=0$).

\begin{center}
\begin{tabular}{||c ||  c | c | c | c ||} 
 \hline
 $j$ &  $2n+1\leq j \leq 4n-7$ & $4n-5$ & $4n-3$ & $4n-1$ \\ [0.5ex] 
 \hline\hline
 
 $b_{d-j-1}-b_{d-j-3}$ &  $0$ & $1$ & $2$ & $1$ \\ [1ex] 
 \hline
\end{tabular}
\end{center}
We are interested in when $c<j<4n-1$ and $b_{d-j-1}-b_{d-j-3}\neq 0$, namely the columns corresponding to $4n-5$ and $4n-3$. If $n=3$, then $c=3n-2=4n-5$, so we are only interested in $j=4n-3$, and $H^{4n-3}_{\ol{O}_1}(S)=D_0^{\oplus 2}$. If $n\geq 4$, then $4n-5>c$, so $H^{4n-5}_{\ol{O}_1}(S)=D_0$. For all $n\geq 3$ we have $4n-3>c$ so $H^{4n-3}_{\ol{O}_1}(S)=D_0^{\oplus 2}$.
\end{proof}

Throughout $\tn{Ext}$ denotes the group of extensions in the subcategory $\tn{mod}_{\GL}(\D_V)$. 

\begin{lemma}\label{ext01}
Let $c=\textnormal{codim}(\overline{O}_1,V)=3n-2$. The module $H^c_{\overline{O}_1}(S)$ is the injective envelope in $\tn{mod}_{\tn{GL}}^{\overline{O}_1}(\D_V)$ of $D_1$. If $n\neq 3$ then $\tn{Ext}^1(D_1,D_0)=\tn{Ext}^1(D_0,D_1)=0$. If $n=3$, then
$$
\dim_{\mathbb{C}} \tn{Ext}^1(D_1,D_0)=\dim_{\mathbb{C}} \tn{Ext}^1(D_0,D_1)=1.
$$	
\end{lemma}

\begin{proof}
By \cite[Lemma 3.11]{lHorincz2019categories} there is an exact sequence of $\D$-modules
$$
0 \longrightarrow H^c_{\overline{O}_1}(S)\longrightarrow H^c_{O_1}(S)\longrightarrow H^{c+1}_{O_0}(S),
$$
and $H^c_{O_1}(S)$ is the injective envelope in $\tn{mod}_{\tn{GL}}^{\overline{O}_1}(\D_V)$ of $D_1$. Since $H^{c+1}_{O_0}(S)=0$, there is an isomorphism $H^c_{\overline{O}_1}(S)\cong H^c_{O_1}(S)$, proving the first assertion.

Since $H^c_{\overline{O}_1}(S)$ is the injective envelope in $\tn{mod}_{\tn{GL}}^{\overline{O}_1}(\D_V)$ of $D_1$, Theorem \ref{BettiThm} implies that $\tn{Ext}^1(D_0,D_1)$ is as claimed. Since $D_1$ and $D_0$ are self-dual, yields the assertion about $\tn{Ext}^1(D_1,D_0)$.
\end{proof}

Now we turn to studying extensions between $D_2$ and $D_1$, $D_0$. 
Using the above results and \cite[Lemma 3.11]{lHorincz2019categories} we prove the following:

\begin{lemma}\label{lemmaExt12}
We have
\begin{enumerate}
\item If $n=3$ or $n\geq 5$, the injective envelope of $D_2$ in $\tn{mod}_{\tn{GL}}^{\overline{O}_2}(\D_V)$ is a non-trivial extension of $D_1$ by $D_2$. Thus, for these values of $n$, we have $\tn{Ext}^1(D_2,D_0)=\tn{Ext}^1(D_0,D_2)=0$ and the vector spaces $\tn{Ext}^1(D_2,D_1)$, $\tn{Ext}^1(D_1,D_2)$ are one-dimensional.

\item If $n=4$ and $I$ is the injective envelope of $D_2$ in $\tn{mod}_{\tn{GL}}^{\overline{O}_2}(\D_V)$, there is a non-split short exact sequence
$$
0\longrightarrow D_2\longrightarrow I\longrightarrow D_0\oplus D_1\longrightarrow 0.
$$
In particular, for $n=4$ the vector spaces $\tn{Ext}^1(D_2,D_1)$, $\tn{Ext}^1(D_1,D_2)$, $\tn{Ext}^1(D_2,D_0)$, $\tn{Ext}^1(D_0,D_2)$ are one-dimensional.
\end{enumerate}
	
\end{lemma}

\begin{proof}
Similar to the proof of Lemma \ref{ext01}, let $c=\tn{codim}(\overline{O}_2,V)=3n-3$. By \cite[Lemma 3.11]{lHorincz2019categories} there is an exact sequence of $\D$-modules
\begin{equation}\label{injLocalSequence}
0\longrightarrow H^c_{\overline{O}_2}(S)\longrightarrow H^c_{O_2}(S)\longrightarrow H^{c+1}_{\overline{O}_1}(S)\longrightarrow H^{c+1}_{\overline{O}_2}(S),
\end{equation}
and $H^c_{O_2}(S)$ is the injective envelope of $D_2$ in $\tn{mod}_{\tn{GL}}^{\overline{O}_2}(\D_V)$.

To prove (1), first assume $n\geq 6$. Using Theorem \ref{localSegre} we see that $H^c_{\overline{O}_2}(S)=D_2$ and $H^{c+1}_{\overline{O}_2}(S)=0$. By Lemma \ref{ext01} we have $H^{c+1}_{\overline{O}_1}(S)=D_1$, and the result follows from (\ref{injLocalSequence}). If $n=5$, Theorem \ref{localSegre} implies that $H^c_{\overline{O}_2}(S)=D_2$ and $H^{c+1}_{\overline{O}_2}(S)=D_0$, and Lemma \ref{ext01} implies that $H^{c+1}_{\overline{O}_1}(S)=D_1$. Thus, the final map in (\ref{injLocalSequence}) goes from $D_1$ to $D_0$, so it must be zero. Let $n=3$. By Theorem \ref{localSegre} the exact sequence (\ref{injLocalSequence}) becomes
\begin{equation}\label{inj2n3}
0\longrightarrow D_2\longrightarrow H^c_{O_2}(S)\longrightarrow H^{c+1}_{\overline{O}_1}(S)\longrightarrow D_0,
\end{equation}
and $H^{c+1}_{\overline{O}_1}(S)$ is the unique extension of $D_0$ by $D_1$ by Lemma \ref{ext01}. Thus, the claim is that the final map in (\ref{inj2n3}) is surjective, which we prove using the Fourier transform. Let $(\mc{Q},\mc{I})$ be the quiver with relations of $\tn{mod}_{\tn{GL}}(\D_V)$. The number of paths in $\mc{Q}$ (modulo relations) from $D_0$ from $D_2$ is equal to the number of paths from $\mc{F}(D_0)=D_8$ to $\mc{F}(D_2)=D_6$. By Lemma \ref{D6S0fn3}, there are no such paths. Applying the holonomic duality functor, there are no paths in $\mc{Q}$ from $D_2$ to $D_0$. This proves the lemma for $n=3$.

Finally, we prove (2). Let $n=4$. By Theorem \ref{localSegre} we have that $H^c_{\overline{O}_2}(S)$ is a non-trivial extension of $D_0$ by $D_2$, and $H^{c+1}_{\overline{O}_2}(S)=0$. By Lemma \ref{ext01} we have $H^{c+1}_{\overline{O}_1}(S)=D_1$. Thus, the exact sequence (\ref{injLocalSequence}) becomes
$$
0\longrightarrow H^c_{\overline{O}_2}(S)\longrightarrow H^c_{O_2}(S)\longrightarrow D_1\longrightarrow 0.
$$
Since there are no extensions between $D_1$ and $D_0$ for $n=4$, we obtain the desired result.
\end{proof}

\begin{lemma}\label{lemmaExt34}
Let $i=3,4$. If $n\geq 3$ the simple $D_i$ is injective in $\tn{mod}_{\tn{GL}}^{\overline{O}_i}(\D_V)$. In particular, the vector spaces $\tn{Ext}^1(D_i,D_1)$, $\tn{Ext}^1(D_1,D_i)$, $\tn{Ext}^1(D_i,D_0)$, $\tn{Ext}^1(D_0,D_i)$ are zero.	
\end{lemma}

\begin{proof}
Let $i$ be $3$ or $4$. Let $c=\tn{codim}(\overline{O}_i,V)=2n-1$. By Theorem \ref{thm:localcoho34} the only nonvanishing local cohomology is $H^c_{\overline{O}_i}(S)=D_i$ and $H^{c+2n-2}_{\overline{O}_i}(S)=D_0$. Again by \cite[Lemma 3.11]{lHorincz2019categories} we have an exact sequence
$$
0\longrightarrow D_i\longrightarrow H^c_{O_i}(S)\longrightarrow H^{c+1}_{\overline{O}_1}(S)\longrightarrow H^{c+1}_{\overline{O}_i}(S),
$$
and $H^c_{O_i}(S)$ is the injective envelope of $D_i$ in $\tn{mod}_{\tn{GL}}^{\overline{O}_i}(\D_V)$. If $n\geq 3$ then $\tn{codim}(\overline{O}_1,V)-c>1$, so that $H^{c+1}_{\overline{O}_1}(S)=0$. Therefore $D_i\cong H^c_{O_i}(S)$ for $n\geq 3$.
\end{proof}

\section{Categories of equivariant $\D$-modules}\label{sec:eqcat}

In this section, we determine the quiver of the category $\tn{mod}_{\GL}(\D_V)$ of $\GL$-equivariant $\D_V$-modules for the space $V$ of $2\times 2 \times n$ tensors, according to the cases $n=3, n=4,$ and $n\geq 5$.

\subsection{The case $n=3$}\label{sec:n3}

\begin{theorem}\label{thm:quiv3}
Let $n=3$. The quiver of $\tn{mod}_{\GL}(\D_V)$ is given by
\[\xymatrix@R+1pc@C+1.6pc{
& (3) \ar@<0.5ex>[rr]^{\a_3} & & \ar@<0.2ex>[ll]^{\b_3} \ar@<0.2ex>[dll]^{\b_{1'}} (6') \ar@<0.5ex>[rr]^{\a_4} \ar@<0.5ex>[drr]^{\a_{7'}}& & \ar@<0.2ex>[ll]^{\b_4} (4) & \\
(0) \ar@<0.3ex>[r]^{\a_1} & \ar@<0.5ex>[l]^{\b_1}  (1) \ar@<0.3ex>[r]^{\a_2} \ar@<0.5ex>[urr]^{\a_{1'}}& \ar@<0.5ex>[l]^{\b_2} (2) \ar@<0.3ex>[r]^{\a_5} & \ar@<0.5ex>[l]^{\b_5} (5) \ar@<0.3ex>[r]^{\a_6} & \ar@<0.5ex>[l]^{\b_6} (6) \ar@<0.3ex>[r]^{\a_7} & \ar@<0.5ex>[l]^{\b_7} \ar@<0.2ex>[ull]^{\b_{7'}} (7) \ar@<0.3ex>[r]^{\a_8} & \ar@<0.5ex>[l]^{\b_8} (8)
}\]
with the following relations: 
\[\a_i \b_i, \b_i \a_i 
, \a_1 \a_2, \b_2 \b_1, \a_2 \a_5, \b_5\b_2, \a_5\a_6, \b_6 \b_5, \a_6 \a_7, \b_7 \b_6, \a_7 \a_8, \b_8 \b_7, \b_2\a_{1'}, \b_{1'}\a_2, \a_7\b_{7'}, \a_{7'}\b_7,\]
\[\a_{1'}\b_3, \a_{3}\b_{1'}, \a_{1'}\a_4, \b_4\b_{1'}, \b_{7'}\b_3, \a_3\a_{7'}, \b_{7'}\a_4, \b_4\a_{7'}.\]
The Fourier transform acts on the quiver as the reflection along the line between $(5)$ and $(6')$.
\end{theorem}

\begin{proof}
By \cite[Lemma 2.4]{binary}, $S_f$ is the injective envelope of $S=D_8$. Since $S_f/S=H_f^1(S)$, and $D_7$ is its unique simple submodule, we conclude that there is a single arrow $\a_8$ from $(7)$ to $(8)$, and there are no other arrows into $(8)$. Via holonomic duality $\bb{D}$, there is a single arrow $\b_8$ from $(8)$ to $(7)$, and the vertex $(8)$ has no further arrows connected to it. Since $S$ has multiplicity one in $S_f$, we have $\b_8 \a_8=0$.  By Lemma \ref{lem:opencat2} we also have $\a_8 \b_8 =0$, as the restriction of the injective envelope of $D_7$ to $V\setminus \ol{O}_6$ (which is again an injective envelope there) has only one copy of the simple $\D_{V\setminus \ol{O}_6}$-module at corresponding to the trivial local system on $O_7$. Furthermore, since $\a_8\b_8 =0$, all relations in the quiver of the subcategory  $\tn{mod}_{\GL}^{\ol{O}_7}(\D_V)$ stay relations in the quiver of $\tn{mod}_{\GL}(\D_V)$. In particular, by \cite[Corollary 3.9]{lHorincz2019categories} all $2$-cycles at $(7)$ must be zero in the quiver of  $\tn{mod}_{\GL}(\D_V)$. By Fourier transform $\mc{F}$, there are two arrows $\a_1,\b_1$ between $(0)$ and $(1)$ with $\a_1 \b_1=0$, there are no further arrows connected to $(0)$, and all $2$-cycles at $(1)$ must be zero. 

By Lemma \ref{lemmaExt12}, there are exactly two arrows $\a_2, \b_2$ between $(1), (2)$, and by Lemma \ref{lemmaExt34} there are no arrows between $(3)$ (resp. $(4)$) and $(1)$ or $(2)$. Applying $\mc{F}$, we have that there are exactly two arrows $\a_7,\b_7$ between $(6),(7)$, and there are no arrows between $(3)$ (resp. $(4)$) and $(6)$ or $(7)$. By Lemma \ref{D6S0fn3} we have $\a_7 \a_8=0$, and by self-duality of $D_6$, $D_7$, $D_8$, we have $\b_8 \b_7=0$.

Let $U= V\setminus \ol{O}_2$ and $j: U \to V$ be the open embedding. By Lemma \ref{lem:opencat1} and using the argument with injective envelopes as we did in its proof, we get arrows $\a_3, \b_3, \a_4, \b_4, \a_6,\b_6$ in the quiver of $\tn{mod}_{\GL}(\D_V)$ as required, and there are no further arrows between $(3)$, $(4)$, $(5)$, $(6)$ and $(6')$. Further, all compositions between the former arrows are zero in the subcategory $\tn{mod}_{\GL}^{\ol{O}_6}(\D_V)$. 

As $\mc{F}(D_3) = D_4$, we must have $\mc{F}(D_5) = D_5$ and $\mc{F}(D'_{6}) = D'_6$. In particular, we have unique arrows $\a_5, \b_5$ between $(2)$ and $(5)$. As there are no more arrows involving the vertices $(3)$ and $(4)$, this implies that the compositions between $\a_3, \b_3, \a_4, \b_4$ are zero in $\tn{mod}_{\GL}(\D_V)$, and not just in the subcategory $\tn{mod}_{\GL}^{\ol{O}_6}(\D_V)$.

We have the following exact sequence
\begin{equation}\label{eq:O7}
0 \to H^1_f(S) \to H^1_{O_7}(S) \to H^2_{\ol{O}_6}(S) \to 0.
\end{equation}
By Lemma \ref{D6S0fn3}, $D_6$ does not appear as a composition factor in $H^1_f(S)$. We have already showed that there is an arrow $\a_7$ from $(6)$ to $(7)$. Together with \cite[Lemma 3.11]{lHorincz2019categories}, this shows that $D_6$ is a direct summand of $H^1_{O_7}(S)/D_7$, and therefore $\a_6 \a_7 = 0$. Via $\mathbb{D}$ and $\mc{F}$, we obtain also $\b_7 \b_6 = \a_2 \a_5 = \b_5 \b_2=0$.

The $b$-function (\ref{eq:bs3}) of $f$ yields according to \cite[Proposition 4.9]{lHorincz2019categories} a filtration
\[S \subsetneq \D_V f^{-1} \subsetneq \D_V f^{-2} = S_f.\]
By \cite[Lemma 2.8]{lHorincz2019categories} the modules $\D_V f^{-1}$ and $S_f$ have unique simple quotients, say $L^{-1}$ and $L^{-2}$, with $L^{-i}$ having a non-zero semi-invariant element of weight $\sigma^{-i}$, for $i=1,2$, where $\sigma$ stands for the weight of $f$. 

By \cite[Proposition 4.6]{microlocal}, the characteristic variety of $S_f$ contains the conormal $T_{\{0\}}^* V$ at the origin. Since $S$ is the only simple equivariant $\D_V$-module with full support, via $\mc{F}$ we obtain that $D_0$ is the only simple equivariant $\D_V$-module whose characteristic variety contains $T_{\{0\}}^* V$. Therefore, $D_0$ must be a composition factor of $S_f$. From \cite[Equation 4.14]{lHorincz2019categories}, we see that $D_0$ has a non-zero semi-invariant of weight $\sigma^{-2}$. Since the weight space of $\sigma^{-2}$ in $S_f$ is one-dimensional (as the action is prehomogeneous), this shows that $L^{-2} \cong D_0$.

We have an exact sequence
\begin{equation}\label{eq:O6}
0 \to H^2_{\ol{O}_6}(S) \to H^2_{O_6}(S) \to H^3_{\ol{O}_5}(S) \to H^3_{\ol{O}_6}(S).
\end{equation}
By Theorem \ref{thm:localcoho6} we have $H^2_{\ol{O}_6}(S)=D_6$ and $H^3_{\ol{O}_6}(S)=D_2$. By \cite[Lemma 3.11]{lHorincz2019categories} there are no arrows connected to $(6)$ other than $\a_6, \b_6, \a_7, \b_7$. Via $\mc{F}$, there are exactly $4$ arrows connected to $(2)$ as required. 

At this point only the vertices $(1), (5), (6'), (7)$ can have potentially additional arrows among themselves. We have seen that there are no arrows between $(5)$ and $(6')$, and now we show that there is no arrow from $(1)$ to $(5)$. Assume by contradiction that there is such an arrow $\gamma$. Then from the sequence (\ref{eq:O6}) and \cite[Lemma 3.11]{lHorincz2019categories} we deduce that the composition $\gamma \a_6$ is a non-zero path from $(1)$ to $(6)$. Via $\mc{F}$ and $\mathbb{D}$, we obtain a non-zero path from $(2)$ to $(7)$ starting with $\a_5$. Since $\a_2 \a_5 =0$, and in the subcategory $\tn{mod}_{\GL}^{\ol{O}_6}(\D_V)$ we have $\b_5 \a_5 =0 $ (and there are no other arrows into $(2)$), we see by \cite[Lemma 3.11]{lHorincz2019categories} that $D_2$ is a quotient of $H^1_{O_7}(S)/D_7 \, \in \tn{mod}_{\GL}^{\ol{O}_6}(\D_V)$. From (\ref{eq:O7}), this in turn implies that $D_2$ is a quotient of $S_f$. This is a contradiction, since $D_0=L^{-2}$ is the unique quotient of $S_f$. Therefore, no such arrow $\gamma$ exists. Via $\mc{F}$ and $\mathbb{D}$, we see that the vertex $(5)$ can not have other arrows besides $\a_5,\b_5,\a_6,\b_6$. This further implies that $\a_6\b_6=0$ (and not just in the subcategory $\tn{mod}_{\GL}^{\ol{O}_6}(\D_V)$), and via $\mc{F}$ also $\b_5\a_5=0$.

As $S_f$ is the injective envelope of $S$, and $D_0$ is one of its composition factors, there is a non-trivial path from $(0)$ to $(8)$. Since $\a_1$ is the only arrow starting from $(0)$, this shows that $D_1$ is a composition factor of $S_f$. By  \cite[Equation 4.14]{lHorincz2019categories} we see that $\mc{F}(L^{-1})$ also has a non-zero semi-invariant element of weight $\sigma^{-1}$. As both $D_7$ and $D_1 = \mc{F}(D_7)$ are composition factors of $S_f$, and its weight space corresponding to $\sigma^{-1}$ is 1-dimensional, we must have $D_7 \ncong L^{-1} \ncong D_1$. From Theorem \ref{thm:charMatrices} we see readily that $\sigma^{-1}$ does not appear in the character formula for $D_2$ (it has no non-zero $\SL_3(\C)$-invariants), thus $L^{-1} \ncong D_2$. 

By Lemma \ref{lem:loc3} (for $s=-2$), we have that $\tn{supp} \, S_f/\D_Vf^{-1} \subset \ol{O}_1$. We have an exact sequence of the form
\[ 0 \to K \to S_f \to D_0 \to 0.\]
Clearly, $\D_V f^{-1} \subset K$, and by the above, $D_1$ is a quotient of $K$. Since $L^{-1}$ is the unique quotient of $\D_V f^{-1}$ and $L^{-1} \ncong D_1$, this shows that $D_1$ is a composition factor of $S_f / D_V f^{-1}$. Since the latter has $D_0$ as a unique quotient, and has support contained in $\ol{O}_1$, it is given by the (unique) non-split sequence 
\[ 0 \to D_1 \to S_f / \D_V f^{-1} \to D_0 \to 0.\]
As the module $\D_V f^{-1}$ (resp. $S_f$) has the unique quotient $L^{-1}$ (resp. $D_0$), we must have an arrow from $(1)$ to the vertex corresponding to the simple $L^{-1}$. As $D_2 \ncong L^{-1} \ncong D_7$, the only possibility is that $D'_6 \cong L^{-1}$. This gives the arrow $\a_{1'}$, and via $\mc{F}$ and $\bb{D}$ also $\b_{1'}, \a_{7'}, \b_{7'}$.

The module $D_7$ is the unique simple submodule of $H_f^1(S)$. Therefore, $D_7 \subset \D_V f^{-1}/S$, and we denote by $Q$ their quotient. Due to the arrow $\a_{7'}$, we deduce from (\ref{eq:O7}) and \cite[Lemma 3.11]{lHorincz2019categories} that $D'_6$ is a submodule of $Q$. Since $D'_6$ is also the unique quotient of $Q$, this shows that $Q \cong D'_6$. Thus, we have determined the structure of $S_f$, and the corresponding representation of the quiver is non-zero along the path (with one-dimensional spaces at each vertex, and identity map at each arrow) 
\begin{equation}\label{eq:Sf}
(0) \xrightarrow{\a_1} (1) \xrightarrow{\a_{1'}} (6') \xrightarrow{\a_{7'}} (7) \xrightarrow{\a_8} (8).
\end{equation}
With this and from (\ref{eq:O7}) we can now deduce that $D_3$ and $D_4$ are not composition factors of $H^1_{O_7}(S)$, therefore $\a_3 \a_{7'} = \b_4 \a_{7'} =0$. Via $\mathbb{D}$ and $\mc{F}$, we obtain $\b_{7'} \b_3 = \b_{7'} \a_4 = \a_{1'} \b_3 = \a_{1'} \a_4 = \a_3 \b_{1'} = \b_4 \b_{1'}=0$. Further, (\ref{eq:Sf}) and (\ref{eq:O7}) show that there are no arrows from $(1)$ to $(7)$, and $\a_{6'}$ is the only arrow from $(6') \to (7)$. Via $\mathbb{D}$ and $\mc{F}$, we can conclude at this point that we have obtained all the arrows of the quiver.

Let $P_2$ be the projective cover of $D_2$ in $\tn{mod}_{\GL}(\D_V)$, and recall the module $P_2'$ from Lemma \ref{lem:P2char}. Using Lemma \ref{lem:witD6} together with \cite[Lemma 2.1 and Equation 4.14]{lHorincz2019categories}, we see that $P_2$ is a direct summand of $P_2'$, therefore by Lemma \ref{lem:P2char} we must have $\b_2 \a_{1'}= \b_2\a_2 = \b_5 \a_5=\a_5 \a_6=0$. Via $\mc{F}$ and $\mathbb{D}$, we get also $\b_{1'} \a_2 = \a_7 \b_{7'} = \a_{7'} \b_7=\b_6\b_5=\b_6\a_6=\a_7\b_7=0 $

We are left to show that all $2$-cycles at $(6')$ are zero. The projective cover $P_{6}'$ of $D_6'$ in $\tn{mod}_{\GL}(\D_V)$ is a direct summand of the projective module $P(\sigma^{-1})$ as constructed in \cite[Lemma 2.1]{lHorincz2019categories}. By Lemma \ref{lem:dense} the variety $\ol{T_{O_6}^* V}$ has a dense $\GL$-orbit, and therefore the multiplicity of $\ol{T_{O_6}^* V}$ in $P(\sigma^{-1})$ is one, which can be seen as in the proof of Lemma \ref{lem:P2char}. This implies that all $2$-cycles at $(6')$ are zero, finishing the proof. 
\end{proof}

\subsection{The case $n=4$}\label{sec:n4}

\begin{theorem}\label{thm:quiv4}
Let $n=4$. The quiver of $\tn{mod}_{\GL}(\D_V)$ is given by
\[\xymatrix@R-0.5pc@C+1.6pc{
& (3) \ar@<0.5ex>[r] & \ar@<0.5ex>[l] (6') \ar@<0.5ex>[r] & \ar@<0.5ex>[l] (4) & \\
(0) \ar@<0.5ex>[r] & \ar@<0.5ex>[l] (2) \ar@<0.5ex>[r] \ar@<0.5ex>[d] & \ar@<0.5ex>[l] \ar@<0.5ex>[d] (6) \ar@<0.5ex>[r] & \ar@<0.5ex>[l] \ar@<0.5ex>[d] (8) \ar@<0.5ex>[r] & \ar@<0.5ex>[l] (9) \\
& (1) \ar@<0.5ex>[u] & (5) \ar@<0.5ex>[u] & (7) \ar@<0.5ex>[u] & 
}\]
with relations: all 2-cycles, and all paths of length two starting or ending at vertices $(1), (5), (7)$.
\end{theorem}

\begin{proof}
As in this case the semi-invariant $f$ is just the $4\times 4$ determinant, from the structure of $S_f$  together with duality $\bb{D}$ we obtain the following arrows of the quiver (cf. \cite[Theorem 5.4]{lHorincz2019categories})
\[\xymatrix{
(0) \ar@<0.5ex>[r] & \ar@<0.5ex>[l] (2) \ar@<0.5ex>[r]  & \ar@<0.5ex>[l] (6) \ar@<0.5ex>[r] & \ar@<0.5ex>[l] (8) \ar@<0.5ex>[r] & \ar@<0.5ex>[l] (9)
}\]

Since the codimension of $O_1, O_5, O_7$ in $O_2, O_6, O_8$ is one, respectively, we obtain as in (\ref{eq:O6}) the arrows between the vertices $(1)$ and $(2)$, $(5)$ and $(6)$, $(7)$ and $(8)$, as desired.

As in the proof of Theorem \ref{thm:quiv3}, we obtain by Lemma \ref{lem:opencat1} the arrows between $(3)$ and $(6')$, $(4)$ and $(6')$. Further, the paths of length two from $(3)$ to $(4)$ and from $(4)$ to $(3)$ are non-zero.

We are left to obtain the relations and to show that there are no more arrows. Since $S_f$ is the injective envelope of $D_9$ (see \cite[Lemma 2.4]{binary}), there are no more arrows connected to $(9)$, the 2-cycle at $(9)$ is zero, and the path of length two from $(7)$ to $(9)$ is zero. By $\mc{F}$ and $\bb{D}$, we see using Proposition \ref{prop:fourier} that there are no more arrows connected to $(0)$, the 2-cycle at $(0)$ is zero, the paths of length two from $(9)$ to $(7)$, $(0)$ to $(1)$, $(1)$ to $(0)$ are all zero.

By \cite[Lemma 3.11]{lHorincz2019categories}, the following short exact sequence shows that there are no more arrows connected to vertex $(8)$:
\begin{equation}\label{eq:O8}
0 \to H^1_{\ol{O}_8}(S) \to H^1_{O_8}(S) \to H^2_{\ol{O}_7}(S) \to 0.
\end{equation}
Hence, via $\mc{F}$ there are no more arrows connected to $(2)$.

As we did at the end of the proof of Theorem \ref{thm:quiv3}, we see by Lemma \ref{lem:dense} that all the 2-cycles at $(8)$ and $(6)$ are zero. Via $\mc{F}$, all the 2-cycles at $(2)$ are also zero. As there is a non-zero path from $(6)$ to either of the vertices $(0), (2), (6), (8), (9)$, we see by a similar argument that there is no non-zero path between either of these vertices and $(6')$. This also implies by (\ref{eq:O8}) that there is no arrow between $(7)$ and either $(6)$ or $(6')$, and so by $\mc{F}$ between $(1)$ and either $(6')$ or $(6)$.

By Lemma \ref{lem:opencat1}, we see that there are no more arrows among the vertices $(3), (4), (5), (6), (6')$. Together with the above, this implies that there are no more arrows connected to the vertices $(6)$ or $(6')$, and also that the 2-cycles at $(6')$ are zero. 

By Theorem \ref{thm:localcoho34} we have $H^7_{\overline{O}_3}(S)=D_3$ and $H^7_{\overline{O}_4}(S)=D_4$. In particular, by \cite[Lemma 3.11]{lHorincz2019categories} there is no arrow from $(1)$ to either $(3)$ or $(4)$, and via $\mc{F}$ no arrow from $(7)$ to either $(3)$ or $(4)$. Hence, there are no more arrows connected to either $(3)$ or $(4)$, and the 2-cycles at these vertices are zero.

By Lemma \ref{lem:codim1} and \cite[Theorem 6.7]{macvil} (see also \cite{kk}), we see via $\mc{F}$ there are no arrows among the vertices $(1), (5), (7)$. Hence, at this stage we have obtained all the arrows of the quiver. 

Using Lemma \ref{lem:opencat2}, we see that the 2-cycle at $(7)$ is zero. Via $\mc{F}$, the 2-cycle $(1)$ is also zero. By \ref{lem:opencat1}, the 2-cycle at $(5)$ is also zero. From (\ref{eq:O8}), we see that the path $(5) \to (6) \to (8)$ is zero. Similarly, using Theorem \ref{thm:localcoho6} we see that the path $(1)\to (2) \to (6)$ is zero. Via $\mc{F}$  we see that the paths $(5) \to (6) \to (2)$ and $(7)\to (8) \to (6)$ are zero as well. By applying also $\bb{D}$, we have now obtained all the relations in the quiver.
\end{proof}

\subsection{The case $n\geq 5$}\label{sec:n5}

\begin{theorem}\label{quiverN5}
Let $n\geq 5$. The quiver of $\tn{mod}_{\GL}(\D_V)$ is given by
\[\xymatrix@R-1pc@C+1.6pc{
& & & (3) \ar@<0.5ex>[r] & \ar@<0.5ex>[l] (6') \ar@<0.5ex>[r] & \ar@<0.5ex>[l] (4) & \\
(0) & (1) \ar@<0.5ex>[r] & (2) \ar@<0.5ex>[l] & (5) \ar@<0.5ex>[r] & (6) \ar@<0.5ex>[l] & \ar@<0.5ex>[r] (7) & \ar@<0.5ex>[l] (8) & (9)
}\]
where all the 2-cycles are zero.
\end{theorem}

\begin{proof}
As the arguments here are very similar to the cases above, we only sketch the proof. The required arrows and relations between $(3), (4), (5), (6), (6')$ are obtained using Lemma \ref{lem:opencat1}. The arrows and relations between $(7), (8)$ are obtained using Lemma \ref{lem:opencat2}, which give also the arrows between $(1)$ and $(2)$ via $\mc{F}$.

We are left to show that there are no more arrows. Since $\op{codim}(\ol{O}_8, V) \geq 2$, it follows from \cite[Lemma 2.4]{binary} that $(9)$ is an isolated vertex, and via $\mc{F}$, so is $(0)$. Theorem \ref{thm:localcoho8} together with the exact sequence analogous to (\ref{eq:O8}) shows that there are no more arrows connected to $(8)$, and using $\mc{F}$, no more arrows connected to $(2)$. By Theorem \ref{thm:localcoho34}and \cite[Lemma 3.11]{lHorincz2019categories} there is no arrow from $(1)$ to either $(3)$ or $(4)$, and via $\mc{F}$ no arrow from $(7)$ to either $(3)$ or $(4)$. Finally, by Proposition \ref{prop:char566}, Lemma \ref{lem:codim1} and \cite[Theorem 6.7]{macvil} (see also \cite{kk}), we see via $\mc{F}$ there are no arrows among the vertices $(1), (5), (6), (6'), (7)$.
\end{proof}

\begin{remark}\label{rem:reptype}
By \cite[Theorem 2.13]{lHorincz2019categories}, the quiver in Theorem \ref{quiverN5} is of finite representation type, i.e. it has finitely many indecomposable representations (up to isomorphism). On the other hand, it is easy to see that the quivers in Theorems \ref{thm:quiv3} and \ref{thm:quiv4} are of wild representation type.
\end{remark}

\section{Local cohomology with support in $\ol{O}_5$}\label{sec:o5}

The goal of this section is to prove the following.

\begin{theorem}\label{loc05}
The following is true about local cohomology of the polynomial ring $S$ with support in $\ol{O}_5$ (all short exact sequences are non-split).	
\begin{enumerate}
\item If $n=3$ then the only non-zero local cohomology module is
$$
0\to D_5 \to H^3_{\ol{O}_5}(S)\to D_2\to 0.
$$
\item If $n=4$ then the non-zero local cohomology modules are	
$$
H^5_{\ol{O}_5}(S)=D_5,\quad 0\to D_2\to H^6_{\ol{O}_5}(S)\to D_0 \to 0.
$$
\item If $n\geq 5$ then the non-zero local cohomology modules are \
$$
H^{2n-3}_{\ol{O}_5}(S)=D_5,\quad H^{3n-6}_{\ol{O}_5}(S)=D_2,\quad H^{4n-10}_{\ol{O}_5}(S)=D_0.
$$
\end{enumerate}
\end{theorem}

The representation $V$ of $\GL$ is equivalent to the action of $\tn{SO}_4(\C) \times \GL_n(\C)$ on the space of $4\times n$ matrices. Note that we could consider also the action of the bigger group $\tn{O}_4(\C) \times \GL_n(\C)$ on $V$. 

\subsection{The case $n=3$} In this subsection, we prove part (1)  of Theorem \ref{loc05}. 

\begin{lemma}\label{lem:D2inv}
We have $D_0^{\tn{SO}_4(\C)} \neq 0$ and $D_2^{\tn{SO}_4(\C)} \neq 0$.
\end{lemma}

\begin{proof}
Since $S^{\tn{SO}_4(\C)} \neq 0$ and $\mc{F}(S)=D_0$, we also have $D_0^{\tn{SO}_4(\C)} \neq 0$ by \cite[(4.14)]{lHorincz2019categories}. By Theorem \ref{thm:charMatrices} the $\GL_4(\C)\times \GL_3(\C)$-representation $\bS_{(-2,-2,-2,-2)}\C^4\oo \bS_{(-2,-3,-3)}\C^3$ belongs to $D_2$, which is $\tn{SO}_4(\C)$-invariant. 
\end{proof}

\begin{lemma}\label{eq:locO5}
For all $i\neq 3$ we have that $H^i_{\ol{O}_5}(S)$ is a direct sum of copies of $D_1$. 
\end{lemma}

\begin{proof}
 
We identify $V$ with the space of $4\times 3$ matrices under the action of $G\times \GL_3(\C)$, where $G=\tn{SO}_4(\C)$. Then $\ol{O}_5$ is set-theoretically defined by  $\op{rank} X^t \cdot X \leq 1$, where $X$ is the $4\times 3$ generic matrix of variables. 

By the First Fundamental Theorem of Invariant Theory for orthogonal groups, $R=S^G$ is a polynomial ring generated by the entries of $X^t \cdot X$. Let $I$ be the ideal in $R$ generated by the $2\times 2$ minors of $X^t \cdot X$. By Lemma \ref{lem:locinv}, we have for all $i$
\[(H^i_{\ol{O}_5}(S))^G \cong H^i_{I}(R).\]
We identify $I$ with the ideal of $2\times 2$ minors of the $3\times 3$ symmetric matrix of variables. By  \cite[(1.5)]{raicu2016ocal} we see that for all $i\neq 3$, we have $H^i_{I}(R)=0$, hence also $(H^i_{\ol{O}_5}(S))^G=0$. By Lemma \ref{lem:D2inv}, we have $D_0^G\neq 0$ and $D_2^G\neq 0$. Thus, $D_0$ and $D_2$ cannot be a composition factors of $H^i_{\ol{O}_5}(S)$ for any $i\neq 3$. The conclusion now follows from Lemma \ref{lem:noD34}.  
\end{proof}

The main result in this subsection is the following. 

\begin{prop}\label{locorbit06n3}
The nonzero local cohomology modules of the polynomial ring $S$ with support in the orbit $O_6$ are (the short exact sequence is non-split):
$$
0\to D_6 \to H^2_{O_6}(S)\to D_5\to 0,\quad H^4_{O_6}(S)=D_0.
$$	
\end{prop}

Using this, we prove part (1) of Theorem \ref{loc05}.

\begin{proof}
[Proof of Theorem \ref{loc05}(1)]	 
The assertion about $H^3_{\ol{O}_5}(S)$ follows from \cite[Lemma 3.11]{lHorincz2019categories} and Theorem \ref{thm:quiv3}. Indeed, there is an exact sequence
$$
0\to H^3_{\ol{O}_5}(S)\to H^3_{O_5}(S)\to H^4_{Z}(S)\to \cdots
$$
where $Z=\ol{O}_2\cup \ol{O}_3\cup \ol{O}_4$. Since $\operatorname{codim}(Z,V) = 5$, we have $H^4_{Z}(S)=0$, so $H^3_{\ol{O}_5}(S)\cong H^3_{O_5}(S)$.

By Theorem \ref{thm:localcoho6} we have
\begin{equation}\label{cohoO6}
H^2_{\ol{O}_6}(S)=D_6,\quad H^3_{\ol{O}_6}(S)=D_2,\quad H^4_{\ol{O}_6}(S)=D_0,
\end{equation}
and all other local cohomology modules vanish. By Proposition \ref{locorbit06n3} there is an exact sequence
$$
0\to H^2_{\ol{O}_6}(S)\to H^2_{O_6}(S)\to H^3_{\ol{O}_5}(S)\to H^3_{\ol{O}_6}(S)\to H^3_{O_6}(S)\to H^4_{\ol{O}_5}(S)\to H^4_{\ol{O}_6}(S)\to H^4_{O_6}(S)\to H^5_{\ol{O}_5}(S)\to 0,
$$
and $H^q_{\ol{O}_5}(S)=0$ for $q\geq 6$. By Lemma \ref{eq:locO5}, we have that $H^4_{\ol{O}_5}(S)$ and $H^5_{\ol{O}_5}(S)$ are direct sums of copies of $D_1$. Using (\ref{cohoO6}) and Proposition \ref{locorbit06n3}, we conclude that these modules are zero. 
\end{proof}

We now turn to the proof of Proposition \ref{locorbit06n3}. Using notation from Section \ref{sec:relmatrices}, we let $Y=Y_{222}$ and $\pi=\pi_{222}$, which gives a desingularization of $\ol{O}_6$. We note that $\pi$ is an isomorphism on $\ol{O}_6\setminus \ol{O}_2$, and it is small (see \cite[Definition 8.2.29]{htt}) if $n=3,4$ (but not for $n\geq 5$).  

For $i=0,\cdots ,6,$ we write $D_i^Y$ for the simple $\D_Y$-module corresponding locally to $D_i$ in the case $n=2$ (see Section \ref{sec:relmatrices}). Let $Y_5=\pi^{-1}(\ol{O}_5)$ be the hypersurface defined by the relative $2\times 2\times 2$ hyperdeterminant, and let $\mathcal{O}_Y(\ast Y_5)$ denote the localization of $\mc{O}_Y$  at $Y_5$. It fits into a non-split short exact sequence (\ref{comphypn2}) with $Z=Y_5$, and by \cite{perlman2020equivariant} we further have a non-split short exact sequence
\begin{equation}\label{eq:n2bundle}
0 \to D_5^Y \to \mathscr{H}^1_{Y_5}(\mc{O}_Y) \to D_0^Y \to 0.
\end{equation}

By Lemma \ref{pushdownlocallyclosed1}, the following implies Proposition \ref{locorbit06n3}.

\begin{lemma}\label{push222to3}
Let $n=3$, $Y=Y_{222}$, $\pi=\pi_{222}$, and $Y_5=\pi^{-1}(\ol{O}_5)$. We have the following.
\begin{enumerate}
\item The non-zero cohomology modules of $\pi_{+}D^Y_5$ are
$$
\mc{H}^{-1}(\pi_{+}D^Y_5)=D_0,\quad \mc{H}^{0}(\pi_{+}D^Y_5)=D_5,\quad \mc{H}^{1}(\pi_{+}D^Y_5)=D_0.
$$
\item The non-zero cohomology modules of $\pi_{+}\mathcal{O}_Y(\ast Y_5)$ are
$$
0\to D_6 \to \mc{H}^0(\pi_{+}\mathcal{O}_Y(\ast Y_5))\to D_5\to 0,\quad \mc{H}^2(\pi_{+}\mathcal{O}_Y(\ast Y_5))=D_0,
$$
where the short exact sequence is non-split.
\end{enumerate}
\end{lemma}

\begin{proof}
By Lemma \ref{pushdownlocallyclosed1}(2) we have $\mc{H}^q(\pi_{+}\mathcal{O}_Y(\ast Y_5))=0$ for $q<0$. The morphism $\pi$ is small and a desingularization of $\ol{O}_6$, so $\pi_{+}\mc{O}_Y=D_6$. By (\ref{comphypn2}) we obtain an  exact sequence
\begin{equation}\label{lesprooflemmapush3}
0\to \mc{H}^{-1}(\pi_{+}\mathscr{H}^1_{Y_5}(\mc{O}_Y))\to D_6\to \mc{H}^0(\pi_{+}\mathcal{O}_Y(\ast Y_5))\to \mc{H}^{0}(\pi_{+}\mathscr{H}^1_{Y_5}(\mc{O}_Y))\to 0,
\end{equation}
and $\mc{H}^{q}(\pi_{+}\mathscr{H}^1_{Y_5}(\mc{O}_Y))\cong \mc{H}^q(\pi_{+}\mathcal{O}_Y(\ast Y_5))$ for $q\neq -1,0$. Since $\mc{H}^{-1}(\pi_{+}\mathscr{H}^1_{Y_5}(\mc{O}_Y))$ is supported inside $\ol{O}_5$, it follows from the sequence above that $\mc{H}^{-1}(\pi_{+}\mathscr{H}^1_{Y_5}(\mc{O}_Y))=0$.

By (\ref{eq:n2bundle}), we conclude that $\mc{H}^{q}(\pi_{+}D^Y_5)\cong \mc{H}^{q-1}(\pi_{+}D^Y_0)$ for $q<0$. The space $Y$ is a geometric vector bundle on $\mathbb{P}^2$, and $D_0^Y$ is the intersection cohomology module associated to the trivial local system on the zero section. It follows from \cite[Proposition 1.5.28]{htt} that $\mc{H}^{i-2}(\pi_{+}D_0^Y)=D_0^{\oplus H^{i}_{DR}(\mathbb{P}^2)}$ for $i\in\mathbb{Z}$. Thus, we have
\begin{equation}\label{pushdo3}
\mc{H}^{-2}(\pi_{+}D^Y_0)=D_0,\quad \mc{H}^{0}(\pi_{+}D^Y_0)=D_0,\quad \mc{H}^{2}(\pi_{+}D^Y_0)=D_0.
\end{equation}
We conclude that, in cohomological degrees less than zero, the only non-zero cohomology module of $\pi_{+}D_5^Y$ is  $\mc{H}^{-1}(\pi_{+}D^Y_5)=D_0$. 

To complete the proof of (1), it remains to verify the cohomology of $\pi_{+}D_5^Y$ in non-negative cohomological degrees. By (\ref{pushdownlocallyclosed}) and Theorem \ref{thm:quiv3} we know that $\mc{H}^0(\pi_{+}\mathcal{O}_Y(\ast Y_5))$ is a non-trivial extension of $D_5$ by $D_6$, so that by (\ref{lesprooflemmapush3}) we conclude that $\mc{H}^{0}(\pi_{+}\mathscr{H}^1_{Y_5}(\mc{O}_Y))=D_5$. As $\mathscr{H}^1_{Y_5}(\mc{O}_Y)$ is an extension of $D_0^Y$ by $D_5^Y$, it follows from (\ref{pushdo3}) that $\mc{H}^{0}(\pi_{+}D^Y_5)=D_5$. Since $\pi$ is proper and $D_5^Y$ is self-dual, it follows that $\pi_{+}D_5^Y$ is self-dual \cite[Theorem 2.7.2]{htt}, so the cohomology of $\pi_{+}D_5^Y$ must be as claimed.

To complete the proof of (2), it remains to show that the cohomology $\pi_{+}\mathcal{O}_Y(\ast Y_5)$ is as claimed in positive cohomological degrees. As explained above, the cohomology of $\pi_{+}\mathcal{O}_Y(\ast Y_5)$ is the same as that of $\pi_{+}\mathscr{H}^1_{Y_5}(\mc{O}_Y)$ in these degrees, and there is a long exact sequence
$$
0\to \mc{H}^{0}(\pi_{+}D^Y_0)\to \mc{H}^{1}(\pi_{+}D^Y_5)\to \mc{H}^{1}(\pi_{+}\mathscr{H}^1_{Y_5}(\mc{O}_Y))\to \mc{H}^{1}(\pi_{+}D^Y_0)\to  
$$ 
$$
\to \mc{H}^{2}(\pi_{+}D^Y_5)\to \mc{H}^{2}(\pi_{+}\mathscr{H}^1_{Y_5}(\mc{O}_Y))\to \mc{H}^{2}(\pi_{+}D^Y_0)\to 0.
$$
Since $\mc{H}^{1}(\pi_{+}D^Y_0)=\mc{H}^{2}(\pi_{+}D^Y_5)=0$ and the first map is an isomorphism, we conclude that $\mc{H}^{1}(\pi_{+}\mathscr{H}^1_{Y_5}(\mc{O}_Y))=0$ and that $\mc{H}^{2}(\pi_{+}\mathscr{H}^1_{Y_5}(\mc{O}_Y))=D_0$.
\end{proof}

\subsection{The case $n=4$}\label{sec:o5n4} In this subsection, we prove part (2)  of Theorem \ref{loc05}. Recall the $\D_V$-modules $Q_p$ introduced in (\ref{eq:detQ}). The main result here is the following. 

\begin{prop}\label{locorbit06n4}
Let $n=4$. The following is true about local cohomology of the polynomial ring $S$ with support in the orbit $O_6$ and the orbit closure $\ol{O}_5$.
\begin{enumerate}
\item The non-zero local cohomology modules with support in $O_6$ are (the short exact sequence is non-split):
$$
0\to Q_2 \to H^4_{O_6}(S)\to D_5\to 0,\quad H^8_{O_6}(S)=D_0.
$$	
\item We have that $H^q_{\ol{O}_5}(S)=0$ for $q\geq 8$.
\end{enumerate}	
\end{prop}

Using this, we prove part (2) of Theorem \ref{loc05}.

\begin{proof}
[Proof of Theorem \ref{loc05}(2)]	 
The assertion about $H^5_{\ol{O}_5}(S)$ follows from \cite[Lemma 3.11]{lHorincz2019categories} and Theorem \ref{thm:quiv4}. By Theorem \ref{thm:localcoho6} we have
\begin{equation}\label{lcmatrix}
H^4_{\ol{O}_6}(S)=Q_2,\quad H^6_{\ol{O}_6}(S)=Q_1,\quad H^8_{\ol{O}_6}(S)=D_0,
\end{equation}
and all other local cohomology modules vanish. By Proposition \ref{locorbit06n4} there is an exact sequence
$$
0\to H^5_{O_6}(S)\to H^6_{\ol{O}_5}(S)\to H^6_{\ol{O}_6}(S)\to H^6_{O_6}(S)\to H^7_{\ol{O}_5}(S)\to 0,
$$
and $H^q_{\ol{O}_5}(S)=0$ for $q\geq 8$. By Proposition \ref{locorbit06n4}(1) and (\ref{lcmatrix}) we conclude that $H^6_{\ol{O}_5}(S)\cong H^6_{\ol{O}_6}(S)=Q_1$ and $H^7_{\ol{O}_5}(S)=0$, completing the proof.\end{proof}

Next, we prove Proposition \ref{locorbit06n4}(2).

\begin{proof}[Proof of Proposition \ref{locorbit06n4}(2)]
Let $Y=Y_{223}$, $\pi=\pi_{223}$, and $Y_5=\pi^{-1}(\ol{O}_5)$. Since $\ol{O}_8$ is the determinant hypersurface, it has rational singularities (e.g. see \cite[Corollary 6.1.5]{weyman}). Thus, we have $\mathbb{R}\pi_{\ast}\mc{O}_Y\cong \C[\ol{O}_8]$, so by Theorem \ref{loc05}(1) there is a degenerate spectral sequence $\mathbb{R}^i\pi_{\ast}\mathscr{H}^j_{Y_5}(\mc{O}_Y)\implies H^{i+j}_{\ol{O}_5}(\C[\ol{O}_8])$, which yields
\begin{equation}\label{pushDown2231}
\mathbb{R}^i\pi_{\ast}\mathscr{H}^3_{Y_5}(\mc{O}_Y)	= H^{i+3}_{\ol{O}_5}(\C[\ol{O}_8]).
\end{equation}
The coordinate ring $\C[\ol{O}_8]$ has equivariant minimal free resolution as follows
\begin{equation}\label{resDet1}
0\longrightarrow S\otimes \det \longrightarrow S\longrightarrow 	\C[\ol{O}_8]\longrightarrow 0.
\end{equation}
As $Y$ is a vector bundle on the Grassmannian $\operatorname{Gr}(3,4)\cong \mathbb{P}^3$, by Lemma \ref{BottLemma} we have that $\mathbb{R}^i\pi_{\ast}\mathscr{H}^3_{Y_5}(\mc{O}_Y)$ is zero for $i\geq 4$. Thus, by (\ref{pushDown2231}) and the long exact sequence of local cohomology of (\ref{resDet1}), we conclude that for $j>3+4$ multiplication by $\det$ is an isomorphism on $H^j_{\ol{O}_5}(S)$. Since $H^j_{\ol{O}_5}(S)$ has support inside $\ol{O}_5$, it follows that $H^i_{\ol{O}_5}(S)=0$ for $j\geq 8$.
	\end{proof}

By Lemma \ref{pushdownlocallyclosed1}, the following implies Proposition \ref{locorbit06n4}(1). 

\begin{lemma}\label{lem:decompn4}
Let $n=4$, $Y=Y_{222}$, $\pi=\pi_{222}$, and $Y_5=\pi^{-1}(\ol{O}_5)$. We have the following.
\begin{enumerate}
\item The non-zero cohomology modules of $\pi_{+}D^Y_5$ are
$$
\mc{H}^{-3}(\pi_{+}D^Y_5)=D_0,\quad \mc{H}^{-1}(\pi_{+}D^Y_5)=D_0,\quad \mc{H}^{0}(\pi_{+}D^Y_5)=D_5\oplus D_2,\quad \mc{H}^{1}(\pi_{+}D^Y_5)=D_0,\quad \mc{H}^{3}(\pi_{+}D^Y_5)=D_0.
$$
\item The non-zero cohomology modules of $\pi_{+}\mathcal{O}_Y(\ast Y_5)$ are
$$
0\to Q_2 \to \mc{H}^0(\pi_{+}\mathcal{O}_Y(\ast Y_5))\to D_5\to 0,\quad \mc{H}^4(\pi_{+}\mathcal{O}_Y(\ast Y_5))=D_0,
$$
where the short exact sequence is non-split.
\end{enumerate}
\end{lemma}

\begin{proof}
The argument uses Lemma \ref{pushdownlocallyclosed1}, and is similar to the proof of Lemma \ref{push222to3}. The space $Y$ is a geometric vector bundle on $\mathbb{G}(2,4)$, and $D_0^Y$ is the intersection cohomology module associated to the trivial local system on the zero section. It follows from \cite[Proposition 1.5.28]{htt} that the non-zero cohomology modules of $\pi_{+}D_0^Y$ are (see, for example, \cite[Theorem 3.10]{3264} for the cohomology of $\mathbb{G}(2,4)$)
\begin{equation}\label{eq:decompD0}
\mc{H}^{-4}(\pi_{+}D^Y_0)=D_0,\quad \mc{H}^{-2}(\pi_{+}D^Y_0)=D_0,\quad \mc{H}^{0}(\pi_{+}D^Y_0)=D_0^{\oplus 2},\quad \mc{H}^{2}(\pi_{+}D^Y_0)=D_0,\quad \mc{H}^{4}(\pi_{+}D^Y_0)=D_0.
\end{equation}
We first prove (1). The short exact sequence $0\to D^Y_5\to \mathscr{H}^1_{Y_5}(\mc{O}_Y) \to D^Y_0\to 0$ induces the long exact sequence
$$
0\to \mc{H}^{-4}(\pi_{+}D^Y_5)\to \mc{H}^{-4}(\pi_{+}\mathscr{H}^1_{Y_5}(\mc{O}_Y))\to \mc{H}^{-4}(\pi_{+}D^Y_0)\to \mc{H}^{-3}(\pi_{+}D^Y_5)\to \mc{H}^{-3}(\pi_{+}\mathscr{H}^1_{Y_5}(\mc{O}_Y))\to   
$$
$$
\to \mc{H}^{-3}(\pi_{+}D^Y_0)\to \mc{H}^{-2}(\pi_{+}D^Y_5)\to \mc{H}^{-2}(\pi_{+}\mathscr{H}^1_{Y_5}(\mc{O}_Y))\to \mc{H}^{-2}(\pi_{+}D^Y_0)\to \mc{H}^{-1}(\pi_{+}D^Y_5)\to 
$$
$$
\to \mc{H}^{-1}(\pi_{+}\mathscr{H}^1_{Y_5}(\mc{O}_Y))\to \mc{H}^{-1}(\pi_{+}D^Y_0)\to \mc{H}^{0}(\pi_{+}D^Y_5)\to \mc{H}^{0}(\pi_{+}\mathscr{H}^1_{Y_5}(\mc{O}_Y))\to \mc{H}^{0}(\pi_{+}D^Y_0) \to  \mc{H}^{1}(\pi_{+}D^Y_5).
$$
Since $\pi$ is a small resolution of $\ol{O}_6$, we have that $\pi_+\mc{O}_Y=D_6$. Thus, the cohomology of $\pi_{+}\mathscr{H}^1_{Y_5}(\mc{O}_Y)$ is the same as that of $\pi_{+}\mathcal{O}_Y(\ast Y_5)$ in nonzero cohomological degrees, and there is an exact sequence
\begin{equation}\label{eq:sesLocO5}
0\to D_6 \to \mc{H}^0(\pi_{+}\mathcal{O}_Y(\ast Y_5))\to \mc{H}^{0}(\pi_{+}\mathscr{H}^1_{Y_5}(\mc{O}_Y))\to 0.
\end{equation}
By Lemma \ref{pushdownlocallyclosed1} we have that the cohomology of $\pi_{+}\mathscr{H}^1_{Y_5}(\mc{O}_Y)$ is zero in negative degrees, so from (\ref{eq:decompD0}) and the long exact sequence above we conclude that the cohomology of $\pi_{+}D^Y_5$ is as claimed in negative degrees. Since $\pi$ is proper and $D_5^Y$ is self-dual, it follows that $\pi_{+}D_5^Y$ is self-dual \cite[Theorem 2.7.2]{htt}. Thus, we have proven (1) in nonzero cohomological degrees.

 By Lemma \ref{pushdownlocallyclosed1}, \cite[Lemma 3.11]{lHorincz2019categories}, and Theorem \ref{thm:quiv4} we know that $\mc{H}^0(\pi_{+}\mathcal{O}_Y(\ast Y_5))$ is a non-trivial extension of $D_5$ by $Q_2$. Thus, by (\ref{eq:sesLocO5}) we have that $\mc{H}^{0}(\pi_{+}\mathscr{H}^1_{Y_5}(\mc{O}_Y))$ is a direct sum of $D_5$ and $Q_1$. Since $\mc{H}^{-1}(\pi_+D_0^Y)=0$, we have that $\mc{H}^{0}(\pi_{+}D^Y_5)$ is a submodule of $\mc{H}^{0}(\pi_{+}\mathscr{H}^1_{Y_5}(\mc{O}_Y))=D_5\oplus Q_1$. As $\mc{H}^{0}(\pi_{+}D^Y_0)=D_0^{\oplus 2}$ and $ \mc{H}^{1}(\pi_{+}D^Y_5)=D_0$, the last map in the long exact sequence above cannot be injective. Since $Q_1$ is a non-trivial extension of $D_0$ by $D_2$, it follows that $D_0$ cannot be a composition factor of $\mc{H}^{0}(\pi_{+}D^Y_5)$, so that $\mc{H}^{0}(\pi_{+}D^Y_5)=D_5\oplus D_2$, as required to prove (1).

To complete the proof of (2), it remains to show that the cohomology $\pi_{+}\mathcal{O}_Y(\ast Y_5)$ is as claimed in positive cohomological degrees. As explained above, the cohomology of $\pi_{+}\mathcal{O}_Y(\ast Y_5)$ is the same as that of $\pi_{+}\mathscr{H}^1_{Y_5}(\mc{O}_Y)$ in these degrees, and there is a long exact sequence 
$$
\to \mc{H}^{0}(\pi_{+}\mathscr{H}^1_{Y_5}(\mc{O}_Y))\to \mc{H}^{0}(\pi_{+}D^Y_0)\to \mc{H}^{1}(\pi_{+}D^Y_5)\to \mc{H}^{1}(\pi_{+}\mathscr{H}^1_{Y_5}(\mc{O}_Y))\to \mc{H}^{1}(\pi_{+}D^Y_0)\to  
$$ 
$$
\to \mc{H}^{2}(\pi_{+}D^Y_5)\to \mc{H}^{2}(\pi_{+}\mathscr{H}^1_{Y_5}(\mc{O}_Y))\to \mc{H}^{2}(\pi_{+}D^Y_0)\to \mc{H}^{3}(\pi_{+}D^Y_5)\to \mc{H}^{3}(\pi_{+}\mathscr{H}^1_{Y_5}(\mc{O}_Y))\to 0.
$$
Furthermore, $\mc{H}^{4}(\pi_{+}\mathscr{H}^1_{Y_5}(\mc{O}_Y))=\mc{H}^{4}(\pi_{+}D^Y_0)=D_0$. It remains to show that $\mc{H}^{q}(\pi_{+}\mathscr{H}^1_{Y_5}(\mc{O}_Y))=0$ for $q=1,2,3$. Since $\mc{H}^{0}(\pi_{+}D^Y_0)$ has two copies of $D_0$ and  $\mc{H}^{0}(\pi_{+}\mathscr{H}^1_{Y_5}(\mc{O}_Y))$ only has one copy of $D_0$, it follows that the map $\mc{H}^{0}(\pi_{+}D^Y_0)\to \mc{H}^{1}(\pi_{+}D^Y_5)$ is a surjection. As $\mc{H}^{1}(\pi_{+}D^Y_0)=0$, we conclude that $\mc{H}^{1}(\pi_{+}\mathscr{H}^1_{Y_5}(\mc{O}_Y))=0$.

To prove the remaining two vanishings, we need to show that the map $\mc{H}^{2}(\pi_{+}D^Y_0)\to \mc{H}^{3}(\pi_{+}D^Y_5)$ is nonzero (and hence an isomorphism). Consider the exact sequence
\[H^7_{\ol{O}_6}(S) \to H^7_{O_6}(S)\to H^8_{\ol{O}_5}(S)\to H^8_{\ol{O}_6}(S).\]
 By (\ref{lcmatrix}) and Proposition \ref{locorbit06n4}(2), we have $H^7_{\ol{O}_6}(S)=H^8_{\ol{O}_5}(S)=0$, hence $H^7_{O_6}(S)=0$. By Lemma \ref{pushdownlocallyclosed1} we have that $\mc{H}^{3}(\pi_{+}\mathscr{H}^1_{Y_5}(\mc{O}_Y))=H^7_{O_6}(S)=0$. Thus, the map $\mc{H}^{2}(\pi_{+}D^Y_0)\to \mc{H}^{3}(\pi_{+}D^Y_5)$ is an isomorphism.
 \end{proof}

\begin{corollary}\label{cor:weightboundsD5n4}
Let $n=4$. All subrepresentations $\bS_{\alpha}A\oo \bS_{\beta}B\oo \bS_{\gamma}C$ of $D_5$ satisfy $\gamma_2\geq -2$ and $\gamma_3\leq -2$.	
\end{corollary}

\begin{proof}
We show that $\gamma_3\leq -2$ for all such subrepresentations. The other inequality then follows from the fact that $\mc{F}(D_5)=D_5$. By Theorem \ref{loc05}(2) we have $H	^5_{\ol{O}_5}(S)=D_5$. By the proof of Proposition \ref{locorbit06n4}(2) we have that the kernel of the multiplication by $\det$ map $D_5\oo \det\to D_5$ is $H^4_{\ol{O}_5}(\C[\ol{O}_8])\cong \mathbb{R}^1\pi_{\ast}\mathscr{H}^3_{Y_5}(\mc{O}_Y)$, where $Y=Y_{223}$, $\pi=\pi_{223}$, and $Y_5=\pi^{-1}(\ol{O}_5)$. By Lemma \ref{BottLemma} all subrepresentations $\bS_{\lambda}A\oo \bS_{\mu}B\oo \bS_{\nu}C$ of $\mathbb{R}^1\pi_{\ast}\mathscr{H}^3_{Y_5}(\mc{O}_Y)$ satisfy $\nu_3=-1$. Suppose for contradiction that $\bS_{\alpha}A\oo \bS_{\beta}B\oo \bS_{\gamma}C$ in $D_5$ satisfies $\gamma_{3}\geq -1$, and let $v$ be its highest weight vector. Then $v\otimes \det^p=(v\oo \det^{p-1})\oo \det$ is not in the kernel of the multiplication map $D_5\otimes \det \to D_5$ for all $p\geq 1$. Thus, $v$ is not supported inside $\ol{O}_8$, a contradiction.
\end{proof}

\subsection{The case $n\geq 5$}\label{sec:o5ngeq5} In this subsection, we let $\pi'=\pi_{224}$ and $Y'=Y_{224}$. We set $Y'_5=\pi'^{-1}(\ol{O}_5)$. Then we have $\mathbb{R}\pi'_{\ast}\mathbb{R}\mathscr{H}^0_{Y_5'}(\mc{O}_{Y'})=\mathbb{R}\Gamma_{\ol{O}_5}(S)$, so there is a spectral sequence
\begin{equation}\label{localO5ngeq5}
E_2^{i,j}=\mathbb{R}^i\pi'_{\ast}\mathscr{H}^j_{Y_5'}(\mc{O}_{Y'})\implies H^{i+j}_{\ol{O}_5}(S).
\end{equation}
We use the case $n=4$, namely Theorem \ref{loc05}(2), to obtain information about the case $n\geq 5$. To determine the $E_2$-page, we calculate the higher direct images along $\pi'$ of the modules $D^{Y'}_5$, $D^{Y'}_2$, and $D^{Y'}_0$. By Theorem \ref{thm:charMatrices} and Lemma \ref{BottLemma} we have that $\mathbb{R}\pi_{\ast}'D_2^{Y'}$ only has cohomology in degree $3n-12$, and $\mathbb{R}\pi_{\ast}'D_0^{Y'}$ only has cohomology in degree $4n-16$. Similarly, by Corollary \ref{cor:weightboundsD5n4} we have that $\mathbb{R}\pi_{\ast}'D_5^{Y'}$ only has cohomology in degree $2n-8$. By Theorem \ref{loc05}(2) the nonzero modules on the $E_2$ page are given by
$$
\quad E_2^{2n-8,5}=\mathbb{R}^{2n-8}\pi'_{\ast}D_5^{Y'},\quad E_2^{3n-12,6}=\mathbb{R}^{3n-12}\pi'_{\ast}D_2^{Y'},\quad E_2^{4n-16,6}=\mathbb{R}^{4n-16}\pi'_{\ast}D_0^{Y'}.
$$
The arrows on the $E_2$-page go from column six to column five, and increase cohomological degree by two. From this, we conclude that there are no nonzero differentials in the spectral sequence, and it degenerates on the $E_2$ page. Since $D_2=L_{Z_1}$ and $D_0=L_{Z_0}$, it follows from \cite{raicu2016characters, raicuWeymanLocal} (see \cite[Lemma 3.4]{socledegs} for the explicit statement) that $E_2^{3n-12,6}\cong D_2$, and $E_2^{4n-16,6}\cong D_0$. By \cite[Lemma 3.11]{lHorincz2019categories} and Theorem \ref{quiverN5}, we have $E_2^{2n-8,5}=D_5$, as required to complete the proof of Theorem \ref{loc05}.

\section{Local cohomology with support in $\ol{O}_7$}\label{sec:O7}

Recall that since for $n=3$, $\ol{O}_7$ is defined by $f$ so that its only non-zero local cohomology module is $H^1_{\ol{O}_7}(S)=S_f/S$. The main result of this section is the following. 

\begin{theorem}\label{loc07}
\begin{enumerate}
\item Let $n=3$ and let $f$ be the semi-invariant. There is a filtration
$$
0\subsetneq S \subsetneq \D_Vf^{-1}\subsetneq S_f,
$$
such that the quotients are described by the following non-split exact sequences:
$$
0\to D_7 \to \D_Vf^{-1}/S \to D_6'\to 0,\quad\quad 0\to D_1 \to S_f/\D_Vf^{-1}\to D_0 \to 0.
$$
\item Let $n\geq 4$. The nonzero local cohomology modules of the polynomial ring $S$ with support in $\ol{O}_7$ are given by:
$$
H^{n-2}_{\ol{O}_7}(S)=D_7,\quad H^{2n-5}_{\ol{O}_7}(S)=D_6',\quad H^{3n-8}_{\ol{O}_7}(S)=D_1,\quad H^{4n-11}_{\ol{O}_7}(S)=D_0,
$$	
\end{enumerate}

\end{theorem}

We note that, when $n=4$, the cohomological degrees above are: $2$, $3$, $4$, $5$. Theorem \ref{loc07}(1) was proven during the proof of Theorem \ref{thm:quiv3}.

The appearance $D_6'$ in its local cohomology detects the bad singularities of $\ol{O}_7$.

\begin{corollary}\label{cor:sing}
For $n\geq 3$, the variety $\ol{O}_7$ is not normal, and for $n\geq 4$ it is not Cohen--Macaulay.
\end{corollary}

\begin{proof}
Let $I$ denote the defining ideal of $\ol{O}_7$. By \cite[Proposition 3.2]{varbaro} and Theorem \ref{loc07}(2), $\op{depth}(S/I)\leq 2n+5 < 3n+2 = \dim(S/I)$ when $n\geq 4$.

Now we prove that $\ol{O}_7$ is not normal. When $n=3$, this follows from Theorem \ref{loc07}(1), since the support of $H^1_f(S)/D_7$ is $\ol{O}_6$, which is contained in the singular locus of $\ol{O}_7$ but has codimension 1. Now we can extrapolate the result readily to the case $n\geq 4$ using \cite[Theorem 1.4]{collapse}(1).
\end{proof}

\begin{remark}
In fact, the above shows that  $\ol{O}_7$ is not set-theoretically Cohen--Macaulay, i.e. there is no Cohen--Macaulay ideal with zero-set $\ol{O}_7$. Furthermore, using \cite[Theorem  3.1]{lovett} we see that all the other orbit closures are Cohen--Macaulay. We claim that they have also rational singularities. For the determinantal varieties this is known  \cite[Corollary 6.1.5]{weyman}. For $\ol{O}_1$ this follows from \cite[pg. 158, Exercise 8.]{weyman}. We are left with $\ol{O}_5$: for $n=2$ this follows from the description of the Bernstein--Sato polynomial of the hyperdeterminant (e.g. see \cite[Section 3.2]{perlman2020equivariant}) by \cite[Theorem 0.4]{saito}, which extends to the case $n\geq 3$ by \cite[Theorem 1.4]{collapse}(2). 
\end{remark}

\subsection{The case $n=4$}\label{sec:o7n4}  We first assume $n=4$.

\begin{lemma}\label{lem:firstFactsO7}
We have the following.
\begin{enumerate}
\item $H^2_{\ol{O}_7}(S)=D_7$.
\item $H^3_{\ol{O}_7}(S)	$ and $H^5_{\ol{O}_7}(S)$ are nonzero.
\end{enumerate}
	
\end{lemma}

\begin{proof}
(1) We have $H^2_{\ol{O}_7}(S)=D_7$ by \cite[Lemma 3.11]{lHorincz2019categories} and Theorem \ref{thm:quiv4}.

(2) Let $G=\tn{O}_4(\C)$. Similar to the proof of Lemma \ref{eq:locO5}, by Lemma \ref{lem:locinv} we have for all $i$
\[(H^i_{\ol{O}_7}(S))^G \cong H^i_{I}(R),\]
where $R=S^G$ is the polynomial ring on $4\times 4$ symmetric matrices, and $I$ is the ideal of $3\times 3$ minors. By \cite[(1.5)]{raicu2016ocal} we have that $H^i_{I}(R)\neq 0$ if and only if $i=3$ or $i=5$. Thus, $H^3_{\ol{O}_7}(S)	$ and $H^5_{\ol{O}_7}(S)$ are nonzero.
\end{proof}

\begin{remark}
In Section \ref{sec:o5} we have used Lemma \ref{lem:locinv} in a stronger way, working with representations of the special orthogonal group explicitly. However, in the proof above this would be less straightforward, as $\tn{O}_4(\C)$ is disconnected. This means that in the category of $\tn{O}_4(\C)\times \GL_n(\C)$-equivariant $\D_V$-modules some new objects may appear with non-trivial action of the component group of $\tn{O}_4(\C)$, i.e. different objects can have the same $\D$-module (but different equivariant) structure.
\end{remark}

\begin{lemma}\label{lem:Varbaro}
For all $j\geq 4$, $D_6$ and $D_6'$ are not composition factors of $H^j_{\ol{O}_7}(S)$.
\end{lemma}

\begin{proof}
The variety $\ol{O}_7$ is defined by the rank conditions $\op{rank} X \leq 3, \, \op{rank} (X^t X)\leq 2$ (cf. proof of Lemma \ref{lem:codim1}). Using Macaulay2 \cite{M2}, we obtain that the minimal free resolution of the ideal generated by the corresponding minors has length 3, hence depth 13 by the Auslander--Buchsbaum Formula (see, for example, \cite[Theorem 1.2.7]{weyman}). The result now follows from \cite[Proposition 3.2]{varbaro}.
\end{proof}

Let $Y=Y_{223}$, $\pi=\pi_{223}$, and $Y_7=\pi^{-1}(\ol{O}_7)$. Since $\ol{O}_8$ is the determinant hypersurface, it has rational singularities. Thus, $\mathbb{R}\pi_{\ast}\mc{O}_Y\cong \C[\ol{O}_8]$, so the spectral sequence $\mathbb{R}^i\pi_{\ast}\mathscr{H}^j_{Y_7}(\mc{O}_Y)\implies H^{i+j}_{\ol{O}_7}(\C[\ol{O}_8])$ yields
\begin{equation}\label{pushDown223}
\mathbb{R}^i\pi_{\ast}\mathscr{H}^1_{Y_7}(\mc{O}_Y)	= H^{i+1}_{\ol{O}_7}(\C[\ol{O}_8]).
\end{equation}
The coordinate ring $\C[\ol{O}_8]$ has equivariant minimal free resolution as follows
\begin{equation}\label{resDet}
0\longrightarrow S\otimes \det \longrightarrow S\longrightarrow 	\C[\ol{O}_8]\longrightarrow 0.
\end{equation}
Using the above information, we deduce the following.

\begin{lemma}
We have that $H^i_{\ol{O}_7}(S)=0$ for $i\geq 6$.
\end{lemma}

\begin{proof}
By Lemma \ref{BottLemma} we have that $\mathbb{R}^i\pi_{\ast}\mathscr{H}^1_{Y_7}(\mc{O}_Y)$ is zero for $i\geq 4$. Thus, by (\ref{pushDown223}) and the long exact sequence of local cohomology of (\ref{resDet}), we conclude that for $i\geq 6$ multiplication by $\det$ is an isomorphism on $H^i_{\ol{O}_7}(S)$. Since $H^i_{\ol{O}_7}(S)$ has support inside $\ol{O}_7$, it follows that $H^i_{\ol{O}_7}(S)=0$ for $i\geq 6$.
\end{proof}

We obtain the following long exact sequence of local cohomology.

$$
0\to H^1_{\ol{O}_7}(\C[\ol{O}_8])\to H^2_{\ol{O}_7}(S)\otimes \det \to H^2_{\ol{O}_7}(S)\to H^2_{\ol{O}_7}(\C[\ol{O}_8])\to H^3_{\ol{O}_7}(S)\otimes \det \to H^3_{\ol{O}_7}(S)\to H^3_{\ol{O}_7}(\C[\ol{O}_8])\to 
$$
$$
H^4_{\ol{O}_7}(S)\otimes \det \to H^4_{\ol{O}_7}(S)\to H^4_{\ol{O}_7}(\C[\ol{O}_8])\to H^5_{\ol{O}_7}(S)\otimes \det \to H^5_{\ol{O}_7}(S)\to 0
$$

For the remainder of the subsection we use the above long exact sequence to determine $H^i_{\ol{O}_7}(S)$ for $i=3,4,5$. Namely, we will use that $H^{i-1}_{\ol{O}_7}(\C[\ol{O}_8])$ contains the kernel of $H^i_{\ol{O}_7}(S)\otimes \det \to H^i_{\ol{O}_7}(S)$ as a subrepresentation, and it contains the cokernel of $H^{i-1}_{\ol{O}_7}(S)\otimes \det \to H^{i-1}_{\ol{O}_7}(S)$ as a subrepresentation.

\begin{lemma}\label{lem:upperBoundsOnWeights}
For $i=2,3,4,5$, we have that all subrepresentations $\bS_{\alpha}A\oo \bS_{\beta}B\oo \bS_{\gamma}C$ of $H^i_{\ol{O}_7}(S)$ satisfy $\gamma_{6-i}\leq -i+1$	.
\end{lemma}

\begin{proof}
Consider the exact sequence:
$$
H^{i-1}_{\ol{O}_7}(\C[\ol{O}_8])\to H^i_{\ol{O}_7}(S)\otimes \det \to H^i_{\ol{O}_7}(S).
$$
By (\ref{pushDown223}) we have $\mathbb{R}^{i-2}\pi_{\ast}\mathscr{H}^1_{Y_7}(\mc{O}_Y)\cong H^{i-1}_{\ol{O}_7}(\C[\ol{O}_8])$. Since $\mathscr{H}^1_{Y_7}(\mc{O}_Y)$ is a direct sum of bundles of the form $\bS_{\lambda}A\oo \bS_{\mu}B\oo \bS_{\nu}\mc{Q}$, it follows from  Lemma \ref{BottLemma} that subrepresentations $\bS_{\lambda}A\oo \bS_{\mu}B\oo \bS_{\nu}C$ of $H^{i-1}_{\ol{O}_7}(\C[\ol{O}_8])$ satisfy $\nu_{6-i}=-i+2$. Suppose for contradiction that $\bS_{\alpha}A\oo \bS_{\beta}B\oo \bS_{\gamma}C$ in $H^i_{\ol{O}_7}(S)$ satisfies $\gamma_{6-i}\geq-i+2$, and let $v$ be its highest weight vector. Then $v\otimes \det^p=(v\oo \det^{p-1})\oo \det$ is not in the kernel of the multiplication map $H^i_{\ol{O}_7}(S)\otimes \det \to H^i_{\ol{O}_7}(S)$ for all $p\geq 1$. Thus, $v$ is not supported inside $\ol{O}_8$, a contradiction.
\end{proof}

\begin{lemma}\label{lem:kerCharlemma}
Let $M$ be a simple module in $\tn{mod}_{\GL}(\D_V)$, and let $K$ denote the kernel of the multiplication by determinant map: $M\otimes \det \to M$. If $M\neq S$ then $K\neq 0$. Furthermore, the following is true:
\begin{enumerate}
\item If $M$ is equal to $D_3$, $D_4$, or $D_5$, then all subrepresentations $\bS_{\alpha}A\oo \bS_{\beta}B\oo \bS_{\gamma}C$ of $K$ satisfy $\gamma_3=-1$.
\item If $M$ is equal to $D_1$ or $D_2$, then all subrepresentations $\bS_{\alpha}A\oo \bS_{\beta}B\oo \bS_{\gamma}C$ of $K$ satisfy $\gamma_2=-2$.
\item If $M$ is equal to $D_0$, then all subrepresentations $\bS_{\alpha}A\oo \bS_{\beta}B\oo \bS_{\gamma}C$ of $K$ satisfy $\gamma_1=-3$. 

\end{enumerate}
\end{lemma}

\begin{proof}
Let $i=0,1,3,4,5$, and let $c_i=\codim(\ol{O}_i,V)$. We have $H^{c_i}_{\ol{O}_i}(S)=D_i$ in these cases, so that
$$
K=H^{c_i-1}_{\ol{O}_i}(\C[\ol{O}_8]).
$$
If $i=2$, then $H^{c_i}_{\ol{O}_i}(S)$ is a non-trivial extension of $D_0$ by $D_2$, so that $K$ is a submodule of $H^{c_i-1}_{\ol{O}_i}(\C[\ol{O}_8])$.
Let $Y=Y_{223}$ and $\pi=\pi_{223}$ as above, and set $Y_i=\pi^{-1}(\ol{O}_i)$. Again, since $\ol{O}_8$ has rational singularities, we have $\mathbb{R}\pi_{\ast}\mc{O}_Y\cong \C[\ol{O}_8]$, so there is a spectral sequence $\mathbb{R}^p\pi_{\ast}\mathscr{H}^q_{Y_i}(\mc{O}_Y)\implies H^{p+q}_{\ol{O}_i}(\C[\ol{O}_8])$.

By Lemma \ref{BottLemma}, we have that $\mathbb{R}^p\pi_{\ast}\mathscr{H}^q_{Y_i}(\mc{O}_Y)$ satisfies:
\begin{enumerate}
\item if $p=1$ then all subrepresentations $\bS_{\alpha}A\oo \bS_{\beta}B\oo \bS_{\gamma}C$ of  $\mathbb{R}^p\pi_{\ast}\mathscr{H}^q_{Y_i}(\mc{O}_Y)$	satisfy $\gamma_3=-1$,
\item if $p=2$ then all subrepresentations $\bS_{\alpha}A\oo \bS_{\beta}B\oo \bS_{\gamma}C$ of  $\mathbb{R}^p\pi_{\ast}\mathscr{H}^q_{Y_i}(\mc{O}_Y)$	satisfy $\gamma_2=-2$,
\item if $p=3$ then all subrepresentations $\bS_{\alpha}A\oo \bS_{\beta}B\oo \bS_{\gamma}C$ of  $\mathbb{R}^p\pi_{\ast}\mathscr{H}^q_{Y_i}(\mc{O}_Y)$	satisfy $\gamma_1=-3$.
\end{enumerate}
Let $c_i^Y=\codim(Y_i,Y)$. If $i=0$, then $c_i^Y=12$ and $c_i=16$, so $p=16-12-1=3$. If $i=1$, then $c_i^Y=7$ and $c_i=10$, so $p=10-7-1=2$. If $i=2$, then $c_i^Y=6$ and $c_i=9$, so $p=9-6-1=2$. If $i=3,4$, then $c_i^Y=5$ and $c_i=7$, so $p=7-5-1=1$. If $i=5$, then $c_i^Y=3$ and $c_i=5$, so $p=5-3-1=1$.
\end{proof}

\begin{lemma}\label{lem:2notsub7}
We have that $D_2$ is not a submodule of $H^4_{\ol{O}_7}(S)$, and $D_6$ is not a submodule of $H^3_{\ol{O}_7}(S)$.	
\end{lemma}

\begin{proof}
We continue to write $Y=Y_{223}$ and $\pi=\pi_{223}$. By Theorem \ref{thm:charMatrices} we have that $\bS_{(-6,-6)}A\oo \bS_{(-6,-6)}B\oo \bS_{(-3,-3,-3,-3)}C$ is a subrepresentation of $D_2$, so that $\bS_{(-4,-4)}A\oo \bS_{(-4,-4)}B\oo \bS_{(-2,-2,-2,-2)}C$	 is a subrepresentation of $D_2\otimes \det$. By Theorem \ref{thm:charMatrices}, $D_2$ has no subrepresentations $\bS_{\alpha}A\oo \bS_{\beta}B\oo \bS_{\gamma}C$ satisfying $\gamma_2=-2$, so it follows that $\bS_{(-4,-4)}A\oo \bS_{(-4,-4)}B\oo \bS_{(-2,-2,-2,-2)}C$ belongs to the kernel of the multiplication map $D_2\otimes \det\to D_2$. Similarly, $\bS_{(-2,-2)}A\oo \bS_{(-2,-2)}B\oo \bS_{(-1,-1,-1,-1)}C$ belongs to the kernel of $D_6\oo \det \to D_6$.

It suffices to show that $\bS_{(-4,-4)}A\oo \bS_{(-4,-4)}B\oo \bS_{(-2,-2,-2,-2)}C$ is not in $H^3_{\ol{O}_7}(\C[\ol{O}_8])\cong \mathbb{R}^2\pi_{\ast}\mathscr{H}^1_{Y_7}(\mc{O}_Y)$ and that $\bS_{(-2,-2)}A\oo \bS_{(-2,-2)}B\oo \bS_{(-1,-1,-1,-1)}C$ does not belong to $H^2_{\ol{O}_7}(\C[\ol{O}_8])\cong \mathbb{R}^1\pi_{\ast}\mathscr{H}^1_{Y_7}(\mc{O}_Y)$. By Bott's Theorem (Lemma \ref{BottLemma}) it suffices to show that neither $\bS_{(-4,-4)}A\oo \bS_{(-4,-4)}B\oo \bS_{(-2,-3,-3)}\mc{Q}$ nor $\bS_{(-2,-2)}A\oo \bS_{(-2,-2)}B\oo \bS_{(-1,-1,-2)}\mc{Q}$ are subbundles of $\mc{O}_Y(\ast Y_7)$. This is proven in Lemma \ref{charSfn3}.
\end{proof}

\begin{lemma}\label{lem:reduceD1D5}
The following is true about $H^i_{\ol{O}_7}(S)$ for $i=4,5$.
\begin{enumerate}
\item $H^4_{\ol{O}_7}(S)$ is a direct sum of copies of $D_1$.
\item We have that $H^5_{\ol{O}_7}(S)=D_0$.	
\end{enumerate}
	
\end{lemma}

\begin{proof}
We continue to write $Y=Y_{223}$ and $\pi=\pi_{223}$. 

(2) By Lemma \ref{lem:firstFactsO7} we have that $H^5_{\ol{O}_7}(S)$ is nonzero, and by Lemma \ref{lem:upperBoundsOnWeights} we have that all subrepresentations $\bS_{\alpha}A\oo \bS_{\beta}B\oo \bS_{\gamma}C$ of $H^5_{\ol{O}_7}(S)$ satisfy $\gamma_{1}\leq -4$. Thus, by Pieri's Rule \cite[Corollary 2.3.5]{weyman} we see that $H^5_{\ol{O}_7}(S)$ must be a nonzero module supported on the origin, so it must be a direct sum of copies of $D_0$. By Theorem \ref{thm:charMatrices}, $D_0$ has no subrepresentations $\bS_{\alpha}A\oo \bS_{\beta}B\oo \bS_{\gamma}C$ satisfying $\gamma_1=-3$, so the representation $\bS_{(-6,-6)}A\oo \bS_{(-6,-6)}B \oo \bS_{(-3,-3,-3,-3)}C$ appears in the kernel of the multiplication map $D_0\oo \det \to D_0$. Since $\bS_{(-6,-6)}A\oo \bS_{(-6,-6)}B \oo \bS_{(-4,-4,-4)}\mc{Q}$ appears in $\mc{O}_Y(\ast Y_7)$ with multiplicity one (as a power of the semi-invariant), we have by Lemma \ref{BottLemma} that $\bS_{(-6,-6)}A\oo \bS_{(-6,-6)}B \oo \bS_{(-3,-3,-3,-3)}C$ appears with multiplicity one in $H^4_{\ol{O}_7}(\C[\ol{O}_8])$. We conclude that $H^5_{\ol{O}_7}(S)=D_0$.	

(1) By Lemma \ref{lem:upperBoundsOnWeights} and Lemma \ref{lem:kerCharlemma}, neither $D_3$, $D_4$ nor $D_5$ are submodules of $H^4_{\ol{O}_7}(S)$. Again, since $\bS_{(-6,-6)}A\oo \bS_{(-6,-6)}\oo\bS_{(-3,-3,-3,-3)}C$ belongs to the kernel of the multiplication map $D_0\oo \det \to D_0$, if $D_0$ were a submodule of $H^4_{\ol{O}_7}(S)$, then $\bS_{(-6,-6)}A\oo \bS_{(-6,-6)}\oo\bS_{(-3,-3,-3,-3)}C$ would  belong to $\mathbb{R}^2\pi_{\ast}\mathscr{H}^1_{Y_7}(\mc{O}_Y)=H^{3}_{\ol{O}_7}(\C[\ol{O}_8])$. By Lemma \ref{BottLemma}, this is impossible, as there is no bundle $\bS_{\lambda}\mc{Q}$ on $\mathbb{G}(3;C)$ with $\mathbb{R}^2\pi_{\ast}\bS_{\lambda}\mc{Q}\cong \bS_{(-3,-3,-3,-3)}C$. By Lemma \ref{lem:2notsub7} we have that $D_2$ is not a submodule of $H^4_{\ol{O}_7}(S)$. By Lemma \ref{lem:Varbaro}, neither $D_6$ nor $D_6'$ are composition factors of $H^4_{\ol{O}_7}(S)$. We conclude that the only possible simple submodule is $D_1$, and by Theorem \ref{thm:quiv4},  $H^4_{\ol{O}_7}(S)$ is supported on the quiver
\[\xymatrix@C+1.6pc{(0) \ar@<0.5ex>[r] & \ar@<0.5ex>[l] (2) \ar@<0.5ex>[r] & \ar@<0.5ex>[l] (1) }\]
with all paths of length two being zero. Suppose $M$ is a non-zero indecomposable summand of $H^4_{\ol{O}_7}(S)$ (with unique simple submodule $D_1$). By \cite[Theorem 2.13]{lHorincz2019categories}, the only possible composition factors of $M$ (and hence $H^4_{\ol{O}_7}(S)$)  are $D_1$ and $D_2$, and if $D_2$ appeared, it would have to be a quotient.

If $D_2$ were a quotient of $H^4_{\ol{O}_7}(S)$ then the cokernel of the multiplication map $D_2\oo \det\to D_2$ would be a subrepresentation of $H^4_{\ol{O}_7}(\C[\ol{O}_8])$. Since $\bS_{(-6,-6)}A\oo \bS_{(-6,-6)}B \oo \bS_{(-3,-3,-3,-3)}C$ belongs to this cokernel, we would then have that another copy of $\bS_{(-6,-6)}A\oo \bS_{(-6,-6)}B \oo \bS_{(-3,-3,-3,-3)}C$ belongs to $H^4_{\ol{O}_7}(\C[\ol{O}_8])$. This is a contradiction, as in part (2) above we showed that this representation appears in $H^4_{\ol{O}_7}(\C[\ol{O}_8])$ with multiplicity one, and comes from the kernel of $D_0\oo\det\to D_0$.
\end{proof}

\begin{lemma}\label{lem:infoD17}
The following is true about weights in $D_1$.
\begin{enumerate}
\item For all $t>0$ the representation $\bS_{(t-6,-t-6)}A \oo \bS_{(t-6,-t-6)}B\oo \bS_{(-3,-3,-3,-3)}C$ appears with multiplicity one in the cokernel of the multiplication map $D_1\oo \det \to D_1$. In particular, it	appears in $D_1$.
\item We have that that $H^4_{\ol{O}_7}(S)=D_1$.

\end{enumerate}
	
\end{lemma}


\begin{proof} 
We continue to write $Y=Y_{223}$ and $\pi=\pi_{223}$. We consider the cokernel of the map $H^4_{\ol{O}_7}(S)\otimes \det \to H^4_{\ol{O}_7}(S)$, which is a subrepresentation of $H^4_{\ol{O}_7}(\C[\ol{O}_8])\cong \mathbb{R}^3\pi_{\ast}\mathscr{H}^1_{Y_7}(\mc{O}_Y)$. By Lemma \ref{lem:reduceD1D5} we know that $H^4_{\ol{O}_7}(S)$ is a direct sum of copies of $D_1$. Since the representations $\bS_{(t-6,-t-6)}A \oo \bS_{(t-6,-t-6)}B\oo \bS_{(-3,-3,-3,-3)}C$ ($t>0$) do not appear in $D_0\oo \det$ (Theorem \ref{thm:charMatrices}), it suffices to show that they appear with multiplicity one in $\mathbb{R}^3\pi_{\ast}\mathscr{H}^1_{Y_7}(\mc{O}_Y)$. By Lemma \ref{BottLemma}, we need to show that $\bS_{(t-6,-t-6)}A \oo \bS_{(t-6,-t-6)}B\oo \bS_{(-4,-4,-4)}\mc{Q}$ appears with multiplicity one in $\mathscr{H}^1_{Y_7}(\mc{O}_Y)$. This is a consequence of Lemma \ref{charSfn3}.
\end{proof}

\begin{corollary}\label{cor:weightBoundsn4}
The following is true about subrepresentations of $D_1$ and $D_7$.	

\begin{enumerate}
\item if $\bS_{\alpha}A\oo \bS_{\beta}B\oo \bS_{\gamma}C$ belongs to $D_1$ then $\gamma_1\geq -3$ and $\gamma_2\leq -3$,
\item if $\bS_{\alpha}A\oo \bS_{\beta}B\oo \bS_{\gamma}C$ belongs to $D_7$ then $\gamma_3\geq -1$ and $\gamma_4\leq -1$.
\end{enumerate}

\end{corollary}

\begin{proof}
The upper bounds on $\gamma_i$ follow from Lemma \ref{lem:upperBoundsOnWeights} and the fact that $H^2_{\ol{O}_7}(S)=D_7$ and $H^4_{\ol{O}_7}(S)=D_1$. The lower bounds on $\gamma_i$ then follow from applying the Fourier transform, using that $\mc{F}(D_7)=D_1$.
\end{proof}

Finally, we turn our attention to $H^3_{\ol{O}_7}(S)$, the remaining ingredient to prove the case $n=4$ of Theorem \ref{loc07}.

\begin{lemma}\label{lem:wit17n4}
Let $n=4$. The following is true about the characters of $D_1$ and $D_7$.

\begin{enumerate}
\item Let $\gamma_1\geq -3$ and $\gamma_2=\gamma_3=\gamma_4\leq -3$, and let $\alpha,\beta \in \mathbb{Z}^2_{\dom}$ with $|\alpha|=|\beta|=|\gamma|$. Representations of the form $\bS_{\alpha}A\oo \bS_{\beta}B\oo \bS_{\gamma}C$ appear in $D_1$ if and only if  $\alpha=\beta$ and $\alpha_2\leq 2\gamma_2-1$. If $\alpha_2= 2\gamma_2-1$, then the multiplicity is one.	
\item Let $\gamma_4\leq -1$ and $\gamma_1=\gamma_2=\gamma_3\geq -1$, and let $\alpha,\beta \in \mathbb{Z}^2_{\dom}$ with $|\alpha|=|\beta|=|\gamma|$. Representations of the form $\bS_{\alpha}A\oo \bS_{\beta}B\oo \bS_{\gamma}C$ appear in $D_7$ if and only if  $\alpha=\beta$ and $\alpha_1\geq 2\gamma_1+1$. If $\alpha_1= 2\gamma_1+1$, then the multiplicity is one.	

\end{enumerate}
	
\end{lemma}

\begin{proof}
Using $\mc{F}$, it suffices to prove (1). Let $\pi=\pi_{111}$, $Y=Y_{111}$, and $Y_1=\pi^{-1}(\ol{O}_1)$. By Lemma \ref{pushdownlocallyclosed1} we have $\pi_{+}\mc{O}_Y(\ast Y_1)=\mathbb{R}\Gamma_{O_1}(S)[10]$. By Theorem \ref{BettiThm}, and the long exact sequence
$$
\cdots \to H^j_{\ol{O}_1}(S)\to H^j_{O_1}(S)\to H^{j-1}_{\ol{O}_0}(S)\to \cdots,
$$	
we have the following in the Grothendieck group of representations of $\GL$ (see Section \ref{sec:relmatrices}): 
\begin{equation}
[\chi(\pi_+ \mc{O}_Y(\ast Y_1))]=[D_1]-4[D_0].
\end{equation}
By Theorem \ref{thm:charMatrices} representations of the form  $\bS_{\alpha}A\oo \bS_{\beta}B\oo \bS_{\gamma}C$ with $\gamma_1\geq -3$ and $\gamma_2=\gamma_3=\gamma_4\leq -3$ do not belong to $D_0$. Thus, for representations of this form, we have 
$$
m_{\alpha,\beta,\gamma}:=\big[ \chi(\pi_+ \mc{O}_Y(\ast Y_1)):\bS_{\alpha}A\oo \bS_{\beta}B \oo \bS_{\gamma}C\big]=\big[ D_1:\bS_{\alpha}A\oo \bS_{\beta}B \oo \bS_{\gamma}C\big].
$$
For the remainder of the proof, we freely use notation from Section \ref{sec:relmatrices}. By \cite[Lemma 2.5, Proposition 2.10]{raicu2016characters} we have that 
$$
\big[\chi(\pi_+ \mc{O}_Y(\ast Y_5))\big]=\lim_{k\to \infty} \big[p_{1,k}(A)\oo p_{1,k}(B)\oo p_{1,k}(C)\oo 
 S^{\ast}\otimes \det(A^{\ast}\oo B^{\ast}\oo C^{\ast})\big].
$$
Thus, by \cite[Lemma 2.3]{raicu2016characters} we have that $m_{\alpha,\beta,\gamma}$ is equal to
$$
\lim_{k\to \infty}\left(\sum_{I,J\in \binom{[2]}{1},K\in \binom{[4]}{1}}\big[S:\bS_{((-\alpha_2-8,-\alpha_1-8),k,I)}A\oo \bS_{((-\beta_2-8,-\beta_1-8),k,J)}B \oo \bS_{((-\gamma_4-4, -\gamma_3-4, -\gamma_2-4, -\gamma_1-4),k,K)}C \big]\right).
$$
Note $\bS_{\lambda}A\oo \bS_{\mu}B\oo \bS_{\nu}C$ appears in $S$ only if $\lambda,\mu,\nu$ are partitions. Since $\gamma_1\geq -3$ we have $-\gamma_1-4\leq -1$, so for $k\gg 0$ we have $((-\gamma_4-4, -\gamma_3-4, -\gamma_2-4, -\gamma_1-4),k,K)$ is a partition if and only if $K=\{4\}$, and
$$
((-\gamma_4-4, -\gamma_3-4, -\gamma_2-4, -\gamma_1-4),k,\{4\})=(k-\gamma_1-7, -\gamma_4-3, -\gamma_3-3, -\gamma_2-3).
$$
Therefore, by the Cauchy Formula applied to $\Sym((A\oo B)\oo C)$, we have that $m_{\alpha,\beta,\gamma}$ is the limit as $k$ goes to infinity of the following quantity:
$$
\big[\bS_{(k-\gamma_1-7, -\gamma_4-3, -\gamma_3-3, -\gamma_2-3)}(A\oo B):\bS_{(k-\alpha_2-8,-\alpha_1-8)}A\oo \bS_{(k-\beta_2-8,-\beta_1-8)}B \big]
$$
$$
-2\big[\bS_{(k-\gamma_1-7, -\gamma_4-3, -\gamma_3-3, -\gamma_2-3)}(A\oo B):\bS_{(k-\alpha_2-8,-\alpha_1-8)}A\oo \bS_{(k-\beta_1-9,-\beta_2-7)}B \big]
$$
$$
+\big[\bS_{(k-\gamma_1-7, -\gamma_4-3, -\gamma_3-3, -\gamma_2-3)}(A\oo B):\bS_{(k-\alpha_1-9,-\alpha_2-7)}A\oo \bS_{(k-\beta_1-9,-\beta_2-7)}B \big],
$$
where the first term corresponds to $I=J=\{1\}$, the middle term corresponds to $I=\{1\}$ and $J=\{2\}$ (twice, by symmetry), and the third term corresponds to $I=J=\{2\}$. Since $\gamma_2=\gamma_3=\gamma_4$, we can twist by $\det(A\oo B)^{\otimes \gamma_2+3}$ and apply the Cauchy Formula. Since $\alpha_1\geq \alpha_2$ and $\beta_1\geq \beta_2$ we have $(k-\alpha_2-8,-\alpha_1-8)\neq (k-\beta_1-9,-\beta_2-7)$, so by the Cauchy Formula we have that the middle term above is zero. Similarly, the first and third term are nonzero only if $\alpha=\beta$. Therefore, $m_{\alpha,\beta,\gamma}$ is nonzero only if $\alpha=\beta$ and is equal to the limit as $k$ goes to infinity of the following quantity:
$$
\big[\bS_{(k+\gamma_2-\gamma_1-4)}(A\oo B):\bS_{(k+2\gamma_2-\alpha_2-2,2\gamma_2-\alpha_1-2)}A\oo \bS_{(k+2\gamma_2-\alpha_2-2,2\gamma_2-\alpha_1-2)}B \big]
$$
$$
+\big[\bS_{(k+\gamma_2-\gamma_1-4)}(A\oo B):\bS_{(k+2\gamma_2-\alpha_2-2,2\gamma_2-\alpha_1-2)}A\oo \bS_{(k+2\gamma_2-\alpha_1-3,2\gamma_2-\alpha_2-1)}B \big].
$$
The first term is nonzero (and equal to one) if and only if $\alpha_1\leq 2\gamma_2-2$. Similarly, the second term is nonzero (and equal to one) if and only if $\alpha_2\leq 2\gamma_2-1$. If $\alpha_2=2\gamma_2-1$, then since $|\alpha|=|\gamma|=\gamma_1+3\gamma_2$ we have $\alpha_1=\gamma_1+\gamma_2+1>2\gamma_2-2$, so $m_{\alpha,\beta,\gamma}=1$. 
\end{proof}

\begin{lemma}\label{lem:factsD5lc7}
The following is true about weights in $D_5$ and $D_7$.
\begin{enumerate}
\item For all $t>0$ the representation $\bS_{(t-2,-t-2)}A \oo \bS_{(t-2,-t-2)}B\oo \bS_{(-1,-1,-1,-1)}C$ appears with multiplicity one in the kernel of the multiplication map $D_5\oo \det \to D_5$. 
\item For all $t>0$ the representation $\bS_{(t-2,-t-2)}A \oo \bS_{(t-2,-t-2)}B\oo \bS_{(-1,-1,-1,-1)}C$ appears in the cokernel of the multiplication map $D_7\oo \det \to D_7$.
\item We have that $D_5$ is not a submodule of  $H^3_{\ol{O}_7}(S)$.
\end{enumerate}

\end{lemma}

\begin{proof}

(1)  Let $Y=Y_{223}$ and $\pi=\pi_{223}$. By Lemma \ref{lem:charD5n3}(1) we have that $\bS_{(t-2,-t-2)}A \oo \bS_{(t-2,-t-2)}B\oo \bS_{(-1,-1,-2)}\mc{Q}$ belongs to $D_5^{Y}$ with multiplicity one. It follows from Theorem \ref{thm:charMatrices} that $\bS_{(t-2,-t-2)}A \oo \bS_{(t-2,-t-2)}B\oo \bS_{(-1,-1,-2)}\mc{Q}$ does not belong to $D_2^Y$, so by Theorem \ref{loc05}(1) we have that $\bS_{(t-2,-t-2)}A \oo \bS_{(t-2,-t-2)}B\oo \bS_{(-1,-1,-2)}\mc{Q}$ belongs to $\mathscr{H}^3_{Y_5}(\mc{O}_Y)$ with multiplicity one. By the proof of Lemma \ref{lem:kerCharlemma}, we have that $\mathbb{R}^1\pi_{\ast}\mathscr{H}^3_{Y_5}(\mc{O}_Y)$ is the kernel of the multiplication map $D_5\oo \det \to D_5$. By Bott's Theorem (Lemma \ref{BottLemma}) applied to $\bS_{(t-2,-t-2)}A \oo \bS_{(t-2,-t-2)}B\oo \bS_{(-1,-1,-2)}\mc{Q}$, we obtain (1). 

(2) By Lemma \ref{lem:wit17n4} we have that $\bS_{(t-2,-t-2)}A \oo \bS_{(t-2,-t-2)}B\oo \bS_{(-1,-1,-1,-1)}C$ appears in $D_7$, and by Corollary \ref{cor:weightBoundsn4} we have that $\bS_{(t-4,-t-4)}A \oo \bS_{(t-4,-t-4)}B\oo \bS_{(-2,-2,-2,-2)}C$ does not appear in $D_7$. Thus, $\bS_{(t-2,-t-2)}A \oo \bS_{(t-2,-t-2)}B\oo \bS_{(-1,-1,-1,-1)}C$ is in the cokernel of the multiplication map $D_7\oo \det \to D_7$.

(3) Since $H^2_{\ol{O}_7}(S)=D_7$, by (1) and (2) it suffices to show that $\bS_{(-1,-3)}A \oo \bS_{(-1,-3)}B\oo \bS_{(-1,-1,-1,-1)}C$ appears with multiplicity one in $H^2_{\ol{O}_7}(\C[\ol{O}_8])$. By Lemma \ref{BottLemma} is suffices to show that $\bS_{(-1,-3)}A \oo \bS_{(-1,-3)}B\oo \bS_{(-1,-1,-2)}\mc{Q}$ appears with multiplicity one in $\mc{O}_Y(\ast Y_7)$. This follows from Lemma \ref{charSfn3}.
\end{proof}

\begin{lemma}\label{lem:no34in7} Let $i=3,4$. We have that $D_i$ is neither a submodule nor a quotient of $	H^3_{\ol{O}_7}(S)$.
\end{lemma}

\begin{proof}
We prove the result for $D_3$. The proof for $D_4$ is similar. We continue to write $Y=Y_{223}$ and $\pi=\pi_{223}$.

By Theorem \ref{thm:charMatrices} we have that $\bS_{(-1,-7)}A\oo \det(B\oo C)=\bS_{(-1,-7)}A\oo \bS_{(-4,-4)}B\oo \bS_{(-2,-2,-2,-2)}C$ is a subrepresentation of $D_3$, so that $\bS_{(1,-5)}A\oo \bS_{(-2,-2)}B\oo \bS_{(-1,-1,-1,-1)}C$	 is a subrepresentation of $D_3\otimes \det$. Since $D_3$ has no subrepresentations $\bS_{\alpha}A\oo \bS_{\beta}B\oo \bS_{\gamma}C$ satisfying $\alpha_2=-5$, it follows that $\bS_{(1,-5)}A\oo \bS_{(-2,-2)}B\oo \bS_{(-1,-1,-1,-1)}C$ belongs to the kernel of the multiplication map $D_3\otimes \det\to D_3$. Since $\mc{F}(D_3)=D_3$, it follows also that $\bS_{(-1,-7)}A\oo\bS_{(-4,-4)}B\oo \bS_{(-2,-2,-2,-2)}C$ is in the cokernel of the multiplication map $D_3\oo\det \to D_3$. 

We need to show that $\bS_{(1,-5)}A\oo \bS_{(-2,-2)}B\oo \bS_{(-1,-1,-1,-1)}C$ does not belong to $H^2_{\ol{O}_7}(\C[\ol{O}_8])\cong \mathbb{R}^1\pi_{\ast}\mathscr{H}^1_{Y_7}(\mc{O}_Y)$ and that $\bS_{(-1,-7)}A\oo\bS_{(-4,-4)}B\oo \bS_{(-2,-2,-2,-2)}C$ does not belong to $H^3_{\ol{O}_7}(\C[\ol{O}_8])\cong \mathbb{R}^2\pi_{\ast}\mathscr{H}^1_{Y_7}(\mc{O}_Y)$. By Bott's Theorem (Lemma \ref{BottLemma}), it suffices to show that neither $\bS_{(1,-5)}A\oo \bS_{(-2,-2)}B\oo \bS_{(-1,-1,-2)}\mc{Q}$ nor $\bS_{(-1,-7)}A\oo \bS_{(-4,-4)}B\oo \bS_{(-2,-3,-3)}\mc{Q}$ are subbundles of $\mc{O}_Y(\ast Y_7)$. Since $\mc{F}(\mc{O}_Y(\ast Y_7))=\mc{O}_Y(\ast Y_7)$ and $\mc{F}(\bS_{(1,-5)}A\oo \bS_{(-2,-2)}B\oo \bS_{(-1,-1,-2)}\mc{Q})=\bS_{(-1,-7)}A\oo \bS_{(-4,-4)}B\oo \bS_{(-2,-3,-3)}\mc{Q}$, it suffices to show that $\bS_{(1,-5)}A\oo \bS_{(-2,-2)}B\oo \bS_{(-1,-1,-2)}\mc{Q}$ is not a subbundle of $\mc{O}_Y(\ast Y_7)$.

Following the argument of Lemma \ref{charSfn3}, it suffices to show that $\bS_{(k+3,k-3)}A^{\ast}\oo \bS_{(k,k)}B^{\ast}$ is not a subrepresentation of $\mathbb{S}_{(2k-1,1,0,0)}(A^{\ast}\oo B^{\ast})$ for $k\gg 0$. To prove this, we apply \cite[Corollary 4.3b]{raicu2012secant}. Using the notation there $f=k$ and $e=2k-2$, so $e<2f$. Thus, the multiplicity is zero.
\end{proof}

\begin{lemma}\label{lem:D6primeisH3}
The following is true about $D_6'$.
\begin{enumerate}
\item We have $H^3_{\ol{O}_7}(S)=D_6'$.
\item The representation $\bS_{(-1,-3)}A\oo \bS_{(-2,-2)}B\oo \bS_{(-1,-1,-1,-1)}C$ appears with multiplicity one in the kernel of the multiplication map $D_6'\oo \det\to D_6'$. In particular, the representation $\bS_{(-3,-5)}A\oo \bS_{(-4,-4)}B\oo \bS_{(-2,-2,-2,-2)}C$ appears in $D_6'$ with multiplicity one.
\item If $\bS_{\alpha}A\oo \bS_{\beta}B\oo \bS_{\gamma}C$ belongs to $D_6'$ then $\gamma_2\geq -2$ and $\gamma_3\leq -2$.
\item The following representations do not appear in $D_6'$: 
$$
\bS_{(-1,-7)}A\oo \bS_{(-4,-4)}B\oo \bS_{(-2,-2,-2,-2)}C, \quad \bS_{(-4,-4)}A\oo \bS_{(-1,-7)}B\oo \bS_{(-2,-2,-2,-2)}C,
$$
$$
\bS_{(-2,-6)}A\oo \bS_{(-4,-4)}B\oo \bS_{(-2,-2,-2,-2)}C.
$$
\end{enumerate}
\end{lemma}

\begin{proof}
The kernel of the multiplication map $H^3_{\ol{O}_7}(S)\oo \det \to H^3_{\ol{O}_7}(S)$ is a subrepresentation of $H^2_{\ol{O}_7}(\C[\ol{O}_8])\cong \mathbb{R}^1\pi_{\ast}\mathscr{H}^1_{Y_7}(\mc{O}_Y)$. By Lemma \ref{BottLemma} we have that all subrepresentations $\bS_{\alpha}A\oo \bS_{\beta}B\oo \bS_{\gamma}C$ of this kernel satisfy $\gamma_3=-1$. By Lemma \ref{lem:kerCharlemma}, neither $D_0$, $D_1$, nor $D_2$ are submodules of $H^3_{\ol{O}_7}(S)$. By Lemma \ref{lem:2notsub7} and Lemma \ref{lem:factsD5lc7} we have that neither $D_5$ nor $D_6$ are submodules of $H^3_{\ol{O}_7}(S)$. By Lemma \ref{lem:no34in7}, neither $D_3$ nor $D_4$ are submodules. We conclude that the only possible simple submodule is $D_6'$. By Theorem \ref{thm:quiv4},  $H^3_{\ol{O}_7}(S)$  must be supported on the quiver
\[\xymatrix@C+1.6pc{(3) \ar@<0.5ex>[r] & \ar@<0.5ex>[l] (6') \ar@<0.5ex>[r] & \ar@<0.5ex>[l] (4) }\]
with all 2-cycles being zero. Suppose $M$ is a non-zero indecomposable summand of $H^3_{\ol{O}_7}(S)$ (with unique simple submodule $D_6'$). The only other possible composition factors of $M$ (and hence $H^3_{\ol{O}_7}(S)$)  are $D_3$ and $D_4$, and if they appeared, they would have to be quotients, by \cite[Theorem 2.13]{lHorincz2019categories}). By Lemma \ref{lem:no34in7} they cannot be quotients of  $H^3_{\ol{O}_7}(S)$, so we conclude that $H^3_{\ol{O}_7}(S)$ is a direct sum of copies of $D_6'$.	

By Lemma \ref{lem:wit17n4} the representation $\bS_{(-1,-3)}A\oo \bS_{(-2,-2)}B\oo \bS_{(-1,-1,-1,-1)}C$ does not appear in $D_7$ (since $(-1,-3)\neq (-2,-2)$). Thus, it does not appear in the cokernel of  $D_7\oo \det \to D_7$. Therefore, to complete the proof of (1) and (2), it suffices to show that $\bS_{(-1,-3)}A\oo \bS_{(-2,-2)}B\oo \bS_{(-1,-1,-1,-1)}C$ appears in $H^2_{\ol{O}_7}(\C[\ol{O}_8])\cong \mathbb{R}^1\pi_{\ast}\mathscr{H}^1_{Y_7}(\mc{O}_Y)$ with multiplicity one. By Bott's Theorem (Lemma \ref{BottLemma}) it suffices to show that  $\bS_{(-1,-3)}A\oo \bS_{(-2,-2)}B\oo \bS_{(-1,-1,-2)}\mc{Q}$ appears in $\mc{O}_Y(\ast Y_7)$ with multiplicity one. This is a consequence of Lemma \ref{charSfn3}.

(3) The upper bound on $\gamma_3$ follows from Lemma \ref{lem:upperBoundsOnWeights} and the fact that $H^3_{\ol{O}_7}(S)=D_6'$. The lower bound on $\gamma_2$ then follows from applying the Fourier transform, using that $\mc{F}(D_6')=D_6'$.

(4) If either of the first two representations appeared in $D_6'$, then by (3) we would have that $\bS_{(1,-5)}A\oo \bS_{(-2,-2)}B\oo \bS_{(-1,-1,-1,-1)}C$ or $\bS_{(-2,-2)}A\oo \bS_{(1,-5)}B\oo \bS_{(-1,-1,-1,-1)}C$ is in the kernel of the multiplication map $D_6'\oo \det \to D_6'$, and hence in  $H^2_{\ol{O}_7}(\C[\ol{O}_8])\cong \mathbb{R}^1\pi_{\ast}\mathscr{H}^1_{Y_7}(\mc{O}_Y)$. We showed that this does not happen during the proof of Lemma \ref{lem:no34in7}. Similarly, if  $\bS_{(-2,-6)}A\oo \bS_{(-4,-4)}B\oo \bS_{(-2,-2,-2,-2)}C$ was in $D_6'$, then $\bS_{(0,-4)}A\oo \bS_{(-2,-2)}B\oo \bS_{(-1,-1,-1,-1)}C$ would be in $H^2_{\ol{O}_7}(\C[\ol{O}_8])\cong \mathbb{R}^1\pi_{\ast}\mathscr{H}^1_{Y_7}(\mc{O}_Y)$. By Lemma \ref{BottLemma} and Lemma \ref{charSfn3}(5) this representation does not appear.
\end{proof}

\subsection{The case $n\geq 5$}\label{sec:lcO7ngeq5} Let $\pi'=\pi_{224}$ and $Y'=Y_{224}$. We set $Y'_7=\pi'^{-1}(\ol{O}_7)$. Then we have $\mathbb{R}\pi'_{\ast}\mathbb{R}\mathscr{H}^0_{Y_7'}(\mc{O}_{Y'})=\mathbb{R}\Gamma_{\ol{O}_7}(S)$, so there is a spectral sequence
\begin{equation}\label{localO7ngeq5}
E_2^{i,j}=\mathbb{R}^i\pi'_{\ast}\mathscr{H}^j_{Y_7'}(\mc{O}_{Y'})\implies H^{i+j}_{\ol{O}_7}(S).
\end{equation} 
We use the case $n=4$ of Theorem \ref{loc07} to prove the case $n\geq 5$. By Corollary \ref{cor:weightBoundsn4}, Lemma \ref{lem:D6primeisH3} and Lemma \ref{BottLemma} the nonzero modules on the $E_2$ page are given by
$$
E_2^{n-4,2}=\mathbb{R}^{n-4}\pi'_{\ast}D_7^{Y'} ,\quad E_2^{2n-8,3}=\mathbb{R}^{2n-8}\pi'_{\ast}D_6'^{Y'},\quad E_2^{3n-12,4}=\mathbb{R}^{3n-12}\pi'_{\ast}D_1^{Y'},\quad E_2^{4n-16,5}=\mathbb{R}^{4n-16}\pi'_{\ast}D_0^{Y'}.
$$
The arrows on the $E_r$ page decrease $j$ by $r-1$, and increase $i$ by $r$, so there are no nonzero differentials in the spectral sequence, and it degenerates on the $E_2$ page.

By \cite[Lemma 3.11]{lHorincz2019categories} and Theorem \ref{quiverN5} we have $\mathbb{R}^{n-4}\pi'_{\ast}D_7^{Y'}\cong D_7$. Since $\mathbb{R}^{4n-16}\pi'_{\ast}D_0^{Y'}\cong D_0$ (see \cite[Lemma 3.4]{socledegs}), to complete the proof of Theorem \ref{loc07}, we need to show that $\mathbb{R}^{2n-8}\pi'_{\ast}D_6'^{Y'}=D_6'$ and $\mathbb{R}^{3n-12}\pi'_{\ast}D_1^{Y'}\cong D_1$. We take care of this in Lemma \ref{lem:directimD1} and Lemma \ref{lem:pushD6prime} below.

\section{Characters and witness weights}\label{sec:characters}

Given a simple module $M\in \tn{mod}_{\GL}(\D_V)$, and a triple $(\alpha,\beta,\gamma)\in \mathbb{Z}^2_{\dom}\times \bb{Z}^2_{\dom}\times \bb{Z}^n_{\dom}$, we say that $\alpha\times \beta\times \gamma$ is a witness weight for $M$ if $\bS_{\alpha}A\otimes \bS_{\beta}B\otimes \bS_{\gamma}C$ appears in $M$ with multiplicity one, and does not appear in any other simple object of $\tn{mod}_{\GL}(\D_V)$. Such a weight may be used to identify the multiplicity of $M$ in a composition series of any $\D$-module in $\tn{mod}_{\GL}(\D_V)$. Furthermore, a witness weight for $M$ yields an explicit presentation of its projective cover in $\tn{mod}_{\GL}(\D_V)$, see  \cite[Section 2.1]{lHorincz2019categories}.

We write $\W(M)$ for the set of witness weights of $M$.

\subsection{Characters of the simple modules when $n=3$}\label{sec:charn3} 

\begin{theorem}\label{thm:witnessn3}
Let $n=3$. The following is true about the witness weights for simple modules in $\tn{mod}_{\GL}(\D_V)$:
$$
(-6,-6)\times (-6,-6)\times (-4,-4,-4)\in \W(D_0),\quad (0,0)\times (0,0)\times (0,0,0)\in \W(D_8),
$$	
$$
(t-6,-t-6)\times (t-6,-t-6)\times (-4,-4,-4)\in \W(D_1)\quad t>0,\quad (t,-t)\times (t,-t)\times (0,0,0)\in \W(D_7)\quad t>0,
$$
$$
(-4,-4)\times (-4,-4)\times (-2,-3,-3)\in \W(D_2),\quad (-2,-2)\times (-2,-2)\times (-1,-1,-2)\in \W(D_6),
$$
$$
(-1,-5)\times (-3,-3)\times (-2,-2,-2)\in \W(D_3),\quad (-3,-3)\times (-1,-5)\times (-2,-2,-2)\in \W(D_4),
$$
$$
(t-2,-t-4)\times (t-3,-t-3)\times (-2,-2,-2),\; (t-3,-t-3)\times (t-2,-t-4)\times (-2,-2,-2)\in \W(D_5)\quad t>0,
$$
$$
(-3,-3)\times (-3,-3)\times (-2,-2,-2)\in \W(D_6').
$$
\end{theorem}

\begin{lemma}\label{lem:helperLemmaWit17}
Let $A$ and $B$ be two-dimensional complex vector spaces, and let $t\geq 0$. For $k\gg  t$ we have that $\mathbb{S}_{(k-t-1,t+1)}A\otimes \mathbb{S}_{(k-t-1,t+1)}B$ appears in $\bS_{(k-2,1,1)}(A\otimes B)$ with multiplicity one.	
\end{lemma}

\begin{proof}
Let $a=k-2$. By Pieri's Rule \cite[Corollary 2.3.5]{weyman} we have $\bw^2\oo \Sym^a\cong \bS_{(a,1,1)}\oplus \bS_{(a+1,1)}$, which implies that we have the following equality in the Grothendieck group of representations of $\operatorname{GL}(A)\times \operatorname{GL}(B)$:
\begin{equation}\label{eqn:GGplethysm}
\big[\bS_{(a,1,1)}(A\otimes B)\big]=\big[\bw^2(A\oo B)\oo S^a(A\oo B)\big]-\big[\bS_{(a+1,1)}(A\otimes B)\big].
\end{equation}
By the Cauchy formulas we have
\begin{align*}
\bw^2(A\oo B)\oo S^a(A\oo B) & = \left(\left(\bw^2A\oo S^2B\right)\oplus \left(S^2A\oo \bw^2B\right)\right)\oo S^a(A\oo B)\\
& = \left(\left(\bw^2A\oo S^2B\right)\oplus \left(S^2A\oo \bw^2B\right)\right)\oo \left(\bigoplus_{\lambda \vdash a} \bS_{\lambda} A\otimes \bS_{\lambda} B\right)\\
& = \left( \bigoplus_{\lambda \vdash a} \bS_{(\lambda_1+1,\lambda_2+1)} A\otimes (S^2B\oo\bS_{\lambda} B)\right)\oplus \left( \bigoplus_{\lambda \vdash a}  (S^2A\oo\bS_{\lambda} A)\oo \bS_{(\lambda_1+1,\lambda_2+1)}B\right),
\end{align*}
where the last equality follows from the fact that $\dim A =\dim B=2$, so that $\bw^2A=\det A$ and $\bw^2B=\det B$. 

We are interested in the case when $\lambda_1=a-t$ and $\lambda_2=t$. Since $a \gg t$, Pieri's Rule implies that
$$
S^2B\oo\bS_{(a-t,t)} B=\bS_{(a-t+2,t)}B\oplus \bS_{(a-t+1,t+1)}B\oplus \bS_{(a-t,t+2)}B,
$$
and the similar identity holds for $S^2A\oo\bS_{(a-t,t)} A$. Therefore, for $a\gg t\geq 0$ we have
$$
\big[\bw^2(A\oo B)\oo S^a(A\oo B) : \mathbb{S}_{(a-t+1,t+1)}A\otimes \mathbb{S}_{(a-t+1,t+1)}B \big] =2.
$$

To complete the proof, we need to show that $[\bS_{(a+1,1)}(A\otimes B):\mathbb{S}_{(a-t+1,t+1)}A\otimes \mathbb{S}_{(a-t+1,t+1)}B ]=1$ when $a\gg t$. To do this, we apply \cite[Corollary 4.3b]{raicu2012secant}, which handles the plethysm $\bS_{\mu}(A\oo B)$ when $A$ and $B$ are two-dimensional, and $\mu$ has at most two parts. Using the notation there, $f=t+1$, $e=2t+3$, and $r=a+2$. Since $a\gg t$ we have $e<r-1$ and $e$ is odd. Therefore, the multiplicity in question is $(e+1)/2-f=(t+2)-(t+1)=1$, as required.
\end{proof}

For proofs in the remainder of this section, we will use material from \cite[Section 2]{raicu2016characters} to calculate Euler characteristics of $\D$-module pushforwards. See Section \ref{sec:relmatrices} for notation used below.

\begin{lemma}\label{lem:charD17n3}
Let $n=3$. The following is true about about the characters of $D_1$ and $D_7$.
\begin{enumerate}
\item For $t>0$ the representation $\bS_{(t-6,-t-6)}A\oo \bS_{(t-6,-t-6)}B \oo \bS_{(-4,-4,-4)}C$ appears in $D_1$ with multiplicity one, and no representation of the form $\bS_{\alpha}A\oo \bS_{\beta}B \oo \bS_{(\gamma_1,\gamma_2,\gamma_3)}C$ with $\gamma_2\geq-2$ appears in $D_1$.
\item For $t>0$ the representation $\bS_{(t,-t)}A\oo \bS_{(t,-t)}B \oo \bS_{(0,0,0)}C$ appears in $D_7$ with multiplicity one, and no representation of the form $\bS_{\alpha}A\oo \bS_{\beta}B \oo \bS_{(\gamma_1,\gamma_2,\gamma_3)}C$ with $\gamma_2\leq -2$ appears in $D_7$.
\end{enumerate}
\end{lemma}

\begin{proof}
If we prove (1), then (2) follows from application of the Fourier transform. Let $Y=Y_{111}$ and $\pi=\pi_{111}$. By Lemma \ref{pushdownlocallyclosed1} we have $\pi_{+}\mc{O}_Y(\ast Y_1)=\mathbb{R}\Gamma_{O_1}(S)[7]$. By Theorem \ref{BettiThm}, and the long exact sequence
$$
\cdots \to H^j_{\ol{O}_1}(S)\to H^j_{O_1}(S)\to H^{j-1}_{\ol{O}_0}(S)\to \cdots,
$$	
we have the following in the Grothendieck group of representations of $\GL$ (see Section \ref{sec:relmatrices}):
\begin{equation}\label{eqn:eulerCharinGG}
[\chi(\pi_{+}\mc{O}_Y(\ast Y_1))]=[H^7_{\ol{O}_1}(S)]+[H^9_{\ol{O}_1}(S)]+[H^{12}_{\ol{O}_0}(S)]=[D_1]+4[D_0].
\end{equation}
Let $\alpha,\beta\in \mathbb{Z}^2_{\dom}$ and $\gamma \in \mathbb{Z}^3_{\dom}$, and define 
$$
m_{\alpha,\beta,\gamma}=\big[ \chi(\pi_+ \mc{O}_Y(\ast Y_1)):\bS_{\alpha}A\oo \bS_{\beta}B \oo \bS_{\gamma}C\big].
$$
By \cite[Lemma 2.5, Proposition 2.10]{raicu2016characters} we have that 
$$
\big[\chi(\pi_+ \mc{O}_Y(\ast Y_1))\big]=\lim_{k\to \infty} \big[p_{1,k}(A)\oo p_{1,k}(B)\oo p_{1,k}(C)\oo S^{\ast}\otimes \det(A^{\ast}\oo B^{\ast}\oo C^{\ast})\big].
$$
Thus, by \cite[Lemma 2.3]{raicu2016characters} we have that
$$
m_{\alpha,\beta,\gamma}=\lim_{k\to \infty}\left(\sum_{\substack{I,J\in \binom{[2]}{1}\\ K\in \binom{[3]}{1}}}\big[S:\bS_{((-\alpha_2-6,-\alpha_1-6),k,I)}A\oo \bS_{((-\beta_2-6,-\beta_1-6),k,J)}B \oo \bS_{((-\gamma_3-4,-\gamma_2-4,-\gamma_1-4),k,K)}C \big]\right).
$$
Note $\bS_{\lambda}A\oo \bS_{\mu}B\oo \bS_{\nu}C$ appears in $S$ only if $\lambda,\mu,\nu$ are partitions. Note that, for $k\gg 0$ we have
$$
[\bS_{((-\gamma_3-4,-\gamma_2-4,-\gamma_1-4),k,K)}C]=\begin{cases}
[\bS_{(k-\gamma_3-4,-\gamma_2-4,-\gamma_1-4)}C] & \tn{if $K=\{1\}$},\\
-[\bS_{(k-\gamma_2-5,-\gamma_3-3,-\gamma_1-4)}C] & \tn{if $K=\{2\}$},\\
[\bS_{(k-\gamma_1-6,-\gamma_3-3,-\gamma_2-3)}C] & \tn{if $K=\{3\}$}.
\end{cases}
$$
If $\gamma_2\geq -2$, then $-\gamma_2\leq 2$ and $-\gamma_1\leq 2$, so $((-\gamma_3-4,-\gamma_2-4,-\gamma_1-4),k,K)$ is not a partition, and $m_{\alpha,\beta,\gamma}=0$.

It remains to calculate $m_{(t-6,-t-6),(t-6,-t-6),(-4,-4,-4)}$ for $t>0$. By Theorem \ref{thm:charMatrices} we have that $\bS_{(t-6,-t-6)}A\oo \bS_{(t-6,-t-6)}B \oo \bS_{(-4,-4,-4)}C$ belongs to $D_0$ if and only if $t=0$, in which case, it has multiplicity one. Thus, by (\ref{eqn:eulerCharinGG}), to complete the proof of (1), we need to show that
$$
m_{(t-6,-t-6),(t-6,-t-6),(-4,-4,-4)}=1\quad t>0.
$$
For $k\gg 0$ we have $[\bS_{((t,-t),k,I)}A]=[\bS_{(k+t,-t)}A]$ if $I=\{1\}$ and $[\bS_{((t,-t),k,I)}A]=-[\bS_{(k-t-1,t+1)}A]$ if $I=\{2\}$, and the similar result holds for $\bS_{((t,-t),k,J)}B$. Since $t>0$, only the case $I=J=\{2\}$ gives partitions. Note that for $k\gg 0$ we have
\begin{equation}
[\bS_{((0,0,0),k,K)}C]=\begin{cases}
[\bS_{(k,0,0)}C] & \tn{if $K=\{1\}$},\\
-[\bS_{(k-1,1,0)}C] & \tn{if $K=\{2\}$},\\
[\bS_{(k-2,1,1)}C] & \tn{if $K=\{3\}$}.
\end{cases}
\end{equation}
Let $t>0$. By the Cauchy formula applied to $\Sym^{k}((A\oo B)\oo C)$, we conclude that
$$
m_{(t-6,-t-6),(t-6,-t-6),(-4,-4,-4)}=[\bS_{(k,0,0)}(A\oo B):\bS_{(k-t-1,t+1)}A\oo \bS_{(k-t-1,t+1)}B]
$$
$$
-[\bS_{(k-1,1,0)}(A\oo B):\bS_{(k-t-1,t+1)}A\oo \bS_{(k-t-1,t+1)}B]+[\bS_{(k-2,1,1)}(A\oo B):\bS_{(k-t-1,t+1)}A\oo \bS_{(k-t-1,t+1)}B],
$$
where $k\gg t>0$. The first summand is one by the Cauchy formula, and the third summand is one by Lemma \ref{lem:helperLemmaWit17}. We use \cite[Corollary 4.3b]{raicu2012secant} to calculate $[\bS_{(k-1,1,0)}(A\oo B):\bS_{(k-t-1,t+1)}A\oo \bS_{(k-t-1,t+1)}B]$. Using the notation there, we have $f=t+1$, $e=2t+3$, and $r=k$. Since $k\gg t$ we have $e<r-1$ and $e$ is odd. Thus, $[\bS_{(k-1,1,0)}(A\oo B):\bS_{(k-t-1,t+1)}A\oo \bS_{(k-t-1,t+1)}B]=(e+1)/2-f=1$. We conclude that, for $t>0$,  we have $m_{(t-6,-t-6),(t-6,-t-6),(-4,-4,-4)}=1-1+1=1$, as desired.
\end{proof}

We remark that one can use the techniques above to verify that $\bS_{(-6,-6)}A\oo \bS_{(-6,-6)}B \oo \bS_{(-4,-4,-4)}C$ appears in $[\chi(\pi_{+}\mc{O}_Y(\ast Y_1))]$ with multiplicity four, which by (\ref{eqn:eulerCharinGG}) implies that this representation does not appear in $D_1$. This is confirmed using Lemma \ref{charSfn3}.

\begin{lemma}\label{lem:charD5n3}
Let $n=3$. The following is true about about the character of $D_5$.
\begin{enumerate}
\item For $t>0$ the following representations appear in $D_5$ with multiplicity one:
$$
\bS_{(t-2,-t-2)}A\oo \bS_{(t-2,-t-2)}B \oo \bS_{(-1,-1,-2)}C,\quad \bS_{(t-4,-t-4)}A\oo \bS_{(t-4,-t-4)}B \oo \bS_{(-2,-3,-3)}C.
$$

\item For $t>0$ the following representations appear in $D_5$ with multiplicity one:
$$
\bS_{(t-2,-t-4)}A\oo \bS_{(t-3,-t-3)}B \oo \bS_{(-2,-2,-2)}C,\quad \bS_{(t-3,-t-3)}A\oo \bS_{(t-2,-t-4)}B \oo \bS_{(-2,-2,-2)}C.
$$
\item If $\gamma_1\leq -3$ or $\gamma_3\geq -1$, then no representation of the form $\bS_{\alpha}A\oo \bS_{\beta}B \oo \bS_{(\gamma_1,\gamma_2,\gamma_3)}C$ appears in $D_5$. 

\item The following representations do not appear in $D_5$:
$$
\bS_{(-4,-4)}A\oo \bS_{(-4,-4)}B \oo \bS_{(-2,-3,-3)}C,\quad \bS_{(-2,-2)}A\oo \bS_{(-2,-2)}B \oo \bS_{(-1,-1,-2)}C,
$$
$$
\bS_{(-1,-5)}A\oo \bS_{(-3,-3)}B \oo \bS_{(-2,-2,-2)}C,\quad \bS_{(-3,-3)}A\oo \bS_{(-1,-5)}B \oo \bS_{(-2,-2,-2)}C,
$$
$$
\bS_{(-3,-3)}A\oo \bS_{(-3,-3)}B \oo \bS_{(-2,-2,-2)}C.
$$
\end{enumerate}

\end{lemma}

We remark that the representations in (1) are used for the local cohomology calculation of $H^i_{\ol{O}_7}(S)$ (in the proof of Lemma \ref{lem:factsD5lc7}), but are not witness weights. Indeed, one can show using the techniques from the proof of Lemma \ref{lem:charD17n3} that $\bS_{(t-4,-t-4)}A\oo \bS_{(t-4,-t-4)}B \oo \bS_{(-2,-3,-3)}C$ belongs to $D_1$, and $\bS_{(t-2,-t-2)}A\oo \bS_{(t-2,-t-2)}B \oo \bS_{(-1,-1,-2)}C$ belongs to $D_7$.

\begin{proof}
Let $Y=Y_{222}$ and $\pi=\pi_{222}$. By Lemma \ref{pushdownlocallyclosed1} we have $\pi_{+}\mc{O}_Y(\ast Y_5)=\mathbb{R}\Gamma_{O_6}(S)[2]$, so by Proposition \ref{locorbit06n3} we have the following in the Grothendieck group of representations of $\GL$ (see Section \ref{sec:relmatrices}):
\begin{equation}\label{eqn:gglco6n3}
[\chi(\pi_{+}\mc{O}_Y(\ast Y_5))]=[D_6]+[D_5]+[D_0].
\end{equation}
Let $\alpha,\beta\in \mathbb{Z}^2_{\dom}$ and $\gamma \in \mathbb{Z}^3_{\dom}$, and define 
$$
m_{\alpha,\beta,\gamma}=\big[ \chi(\pi_+ \mc{O}_Y(\ast Y_5)):\bS_{\alpha}A\oo \bS_{\beta}B \oo \bS_{\gamma}C\big].
$$
By \cite[Lemma 2.5, Proposition 2.10]{raicu2016characters} we have that 
$$
\big[\chi(\pi_+ \mc{O}_Y(\ast Y_5))\big]=\lim_{k\to \infty} \big[\bS_{(2k,2k)}A\oo \bS_{(2k,2k)}B\oo p_{2,2k}(C)\oo 
 S^{\ast}\otimes \det(A^{\ast}\oo B^{\ast}\oo C^{\ast})\big].
$$
Thus, by \cite[Lemma 2.3]{raicu2016characters} we have that
$$
m_{\alpha,\beta,\gamma}=\lim_{k\to \infty}\left(\sum_{J\in \binom{[3]}{2}}\big[S:\bS_{(2k-\alpha_2-6,2k-\alpha_1-6)}A\oo \bS_{(2k-\beta_2-6,2k-\beta_1-6)}B \oo \bS_{((-\gamma_3-4,-\gamma_2-4,-\gamma_1-4),2k,J)}C \big]\right),
$$
where for $k\gg0$ we have
$$
[\bS_{((-\gamma_3-4,-\gamma_2-4,-\gamma_1-4),2k,J)}C]=\begin{cases}
[\bS_{(2k-\gamma_3-4,2k-\gamma_2-4,-\gamma_1-4)}C] & \tn{if }J=\{1,2\},\\
-[\bS_{(2k-\gamma_3-4,2k-\gamma_1-5,-\gamma_2-3)}C] & \tn{if }J=\{1,3\},\\
[\bS_{(2k-\gamma_2-5,2k-\gamma_1-5,-\gamma_3-2)}C] & \tn{if }J=\{2,3\}.	
\end{cases}
$$
Note $\bS_{\lambda}A\oo \bS_{\mu}B\oo \bS_{\nu}C$ appears in $S$ only if $\lambda,\mu,\nu$ are partitions. If $\gamma_3\geq -1$, then $-\gamma_i\leq 1$ for $i=1,2,3$. Thus, $((-\gamma_3-4,-\gamma_2-4,-\gamma_1-4),2k,J)$ is not a partition, so $m_{\alpha,\beta,\gamma}=0$. Since $\mc{F}(D_5)=D_5$, we conclude also that, if $\gamma_1\leq -3$ then the multiplicity of $\bS_{\alpha}A\oo \bS_{\beta}B \oo \bS_{\gamma}C$ in $D_5$ is zero. This proves (3).

(1) We first consider	$\bS_{(t-2,-t-2)}A\oo \bS_{(t-2,-t-2)}B \oo \bS_{(-1,-1,-2)}C$ for $t\geq 0$. If $J$ is equal to $\{1,2\}$ or $\{1,3\}$, then $((-2,-3,-3),2k,J)$ is not a partition, so 
$$
m_{\alpha,\beta,\gamma}=\lim_{k\to \infty}\big[S,\bS_{(2k+t-4,2k-t-4)}A\oo \bS_{(2k+t-4,2k-t-4)}B \oo \bS_{(2k-4,2k-4,0)}C \big].
$$
We apply \cite[Corollary 4.3b]{raicu2012secant}. Using notation there, we have $f=2k-4$, $e=6k-2t-12$, and $r=4k-8$. For $k\gg 0$ we have $e\geq r-1$ and $e$ even, so $m_{\alpha,\beta,\gamma}=\lfloor r/2\rfloor -f+1=1$. By Theorem \ref{thm:charMatrices} we have that $\bS_{(t-2,-t-2)}A\oo \bS_{(t-2,-t-2)}B \oo \bS_{(-1,-1,-2)}C$ belongs to $D_6$ if and only if $t=0$, and it does not belong to $D_0$. By (\ref{eqn:gglco6n3}) we conclude that $\bS_{(t-2,-t-2)}A\oo \bS_{(t-2,-t-2)}B \oo \bS_{(-1,-1,-2)}C$ has multiplicity one in $D_5$ if $t>0$, and multiplicity zero if $t=0$. Since $\mc{F}(D_5)=D_5$, and $\mc{F}$ swaps the representations in (1), we have completed the proof of (1), and proved that the representations in the first row of (4) do not appear in $D_5$.

(2) By symmetry, it suffices to prove the result for $\bS_{(t-2,-t-4)}A\oo \bS_{(t-3,-t-3)}B \oo \bS_{(-2,-2,-2)}C$. Again, if $J$ is equal to $\{1,2\}$ or $\{1,3\}$, then $((-2,-2,-2),2k,J)$ is not a partition, so 
$$
m_{\alpha,\beta,\gamma}=\lim_{k\to \infty}\big[S,\bS_{(2k+t-2,2k-t-4)}A\oo \bS_{(2k+t-3,2k-t-3)}B \oo \bS_{(2k-3,2k-3,0)}C \big].
$$
We apply \cite[Corollary 4.3b]{raicu2012secant}. Using notation there, we have $f=2k-3$, $e=6k-2t-10$, $r=4k-6$. Thus, for $k\gg 0$ we have $e\geq r-1$, and $e$ even, so $m_{\alpha,\beta,\gamma}=(2k-3)-(2k-3)+1=1$. Since $\bS_{(t-2,-t-4)}A\oo \bS_{(t-3,-t-3)}B \oo \bS_{(-2,-2,-2)}C$ does not appear in $D_6$ nor $D_0$, the result follows from (\ref{eqn:gglco6n3}).

(4) It remains to prove the assertion about the second and third rows of (4). For the second row, by symmetry, it suffices to consider $\bS_{(-1,-5)}A\oo \bS_{(-3,-3)}B \oo \bS_{(-2,-2,-2)}C$. Again, if $J$ is equal to $\{1,2\}$ or $\{1,3\}$, then $((-2,-2,-2),2k,J)$ is not a partition, so 
$$
m_{\alpha,\beta,\gamma}=\lim_{k\to \infty}\big[S,\bS_{(2k-1,2k-5)}A\oo \bS_{(2k-3,2k-3)}B \oo \bS_{(2k-3,2k-3,0)}C \big].
$$
By \cite[Corollary 4.3b]{raicu2012secant}, this is zero. A similar argument can be used to verify that $\bS_{(-3,-3)}A\oo \bS_{(-3,-3)}B \oo \bS_{(-2,-2,-2)}C$ does not appear in $D_5$.
\end{proof}

\begin{proof}[Proof of Theorem \ref{thm:witnessn3}]
By Theorem \ref{thm:charMatrices} we have  $(-6,-6)\times (-6,-6)\times (-4,-4,-4)$ appears with multiplicity one in $D_0$, and does not appear in $D_2$, $D_3$, $D_4$, $D_6$, $D_8$. By Lemma \ref{charSfn3}, this weight appears with multiplicity one in $S_f$, so it does not appear in $D_1$, $D_6'$, $D_7$. By Lemma \ref{lem:charD5n3} it does not appear in $D_5$. Therefore, it is a witness weight for $D_0$. Applying $\mc{F}$, we conclude that $(0,0)\times (0,0)\times (0,0,0)$ is a witness weight for $D_8$.

Let $t>0$. By Lemma \ref{lem:charD17n3} we have that $(t-6,-t-6)\times (t-6,-t-6)\times (-4,-4,-4)$ appears with multiplicity one in $D_1$, and does not appear in $D_7$. By Theorem \ref{thm:charMatrices}, these weights do not appear in $D_0$, $D_2$, $D_3$, $D_4$, $D_6$, $D_8$, and by Lemma \ref{lem:charD5n3} they do not appear in $D_5$. By Lemma \ref{charSfn3} they have multiplicity one in $S_f$, so they do not appear in $D_6'$. Therefore, $(t-6,-t-6)\times (t-6,-t-6)\times (-4,-4,-4)$ ($t>0$) are witness weights for $D_1$. Applying $\mc{F}$, we conclude also that $(t,-t)\times (t,-t)\times (0,0,0)$ ($t>0$) are witness weights for $D_7$. 

By Theorem \ref{thm:charMatrices} we have that $(-4,-4)\times (-4,-4)\times (-2,-3,-3)$ appears with multiplicity one in $D_2$, and does not appear in $D_0$, $D_3$, $D_4$, $D_6$, $D_8$. By Lemma \ref{lem:charD5n3}, this weight does not appear in $D_5$. By Lemma \ref{charSfn3} it does not appear in $S_f$, so it does not appear in $D_1$, $D_6'$, $D_7$. Therefore, $(-4,-4)\times (-4,-4)\times (-2,-3,-3)$ is a witness weight for $D_2$. Applying $\mc{F}$, we conclude also that $(-2,-2)\times (-2,-2)\times (-1,-1,-2)$ is a witness weight for $D_6$.

By Theorem \ref{thm:charMatrices} we have that $(-1,-5)\times (-3,-3)\times (-2,-2,-2)$ appears with multiplicity one in $D_3$, and does not appear in $D_0$, $D_2$, $D_4$, $D_6$, $D_8$. By Lemma \ref{lem:charD5n3}, this weight does not appear in $D_5$. By Lemma \ref{charSfn3}, this weight does not appear in $S_f$, so it does not appear in $D_1$, $D_6'$, $D_7$. Therefore, it is a witness weight for $D_3$. The case of $D_4$ is proved similarly.

Let $t>0$. By Lemma \ref{lem:charD5n3} we have that $(t-2,-t-4)\times (t-3,-t-3)\times (-2,-2,-2)$ and $(t-3,-t-3)\times (t-2,-t-4)\times (-2,-2,-2)$ appear with multiplicity one in $D_5$. By Theorem \ref{thm:charMatrices}, these weights do not appear in $D_0$, $D_2$, $D_3$, $D_4$, $D_6$, $D_8$. By Lemma \ref{charSfn3}, these weights do not appear in $S_f$, so they do not appear in $D_1$, $D_6'$, $D_7$. Therefore, they are witness weights for $D_5$.

By Lemma \ref{charSfn3}, the representation $(-3,-3)\times (-3,-3)\times (-2,-2,-2)$ appears in $S_f$ with multiplicity one (spanned by $1/f$). By Lemma \ref{lem:charD17n3}, this weight does not appear in $D_1$ nor $D_7$, and by Theorem \ref{thm:charMatrices} this representation does not appear in $D_0$, $D_2$, $D_3$, $D_4$, $D_6$, $D_8$. By Lemma \ref{lem:charD5n3}, this representation does not appear in $D_5$. Therefore, we have that $(-3,-3)\times (-3,-3)\times (-2,-2,-2)$ appears in $D_6'$ with multiplicity one, and does not appear in any other simple module.
\end{proof}

\subsection{Characters for $n\geq 4$}\label{sec:charactersn4}

\begin{theorem}\label{thm:witnessngeq4}
Let $n\geq 4$. The following is true about the witness weights for simple modules in $\tn{mod}_{\GL}(\D_V)$:
$$
(-2n,-2n)\times (-2n,-2n)\times (-4^n)\in \W(D_0),\quad (0,0)\times (0,0)\times (0^n)\in \W(D_9),
$$	
$$
(-n-1,-2n+1)\times (-n-1,-2n+1)\times (-3^n)\in \W(D_1),\quad (-1,1-n)\times (-1,1-n)\times (-1^n)\in \W(D_7),
$$
$$
(-n-2,-2n+2)\times (-n-2,-2n+2)\times (-3^n)\in \W(D_2),\quad (-2,2-n)\times (-2,2-n)\times (-1^n)\in \W(D_8),
$$
$$
(-1,1-2n)\times (-n,-n)\times (-2^n)\in \W(D_3),\quad (-n,-n)\times (-1,1-2n)\times (-2^n)\in \W(D_4),
$$
$$
(-2,2-2n)\times (-n,-n) \times (-2^n)\in \W(D_5),
$$
$$
(-3,3-2n)\times (-n,-n)\times (-2^n)\in \W(D_6').
$$
$$
(-4,4-2n)\times (-n,-n)\times (-2^n)\in \W(D_6),
$$
\end{theorem}

We note that, in addition to the calculations in this subsection, we found information about the characters of $D_1$ and $D_7$ in Lemma \ref{lem:wit17n4}.

We start by investigating the case $n=4$.

\begin{lemma}\label{lem:charD5D6ngeq4}
Let $n=4$ and $p\geq 4$, and let $\alpha,\beta \in \mathbb{Z}^2_{\dom}$ with $|\alpha|=|\beta|=-2p$. The following is true about $D_5$ and $D_6$.
\begin{enumerate}
\item Representations of the form $\bS_{\alpha}A\oo \bS_{\beta}B\oo \bS_{(-2,-2,2-p,2-p)}C$	appear in $D_6$ if and only if $\alpha_1+\beta_1\leq -4-p$ and $p-\alpha_1-\beta_1$ is even, in which case the multiplicity is one.
\item Representations of the form $\bS_{\alpha}A\oo \bS_{\beta}B\oo \bS_{(-2,-2,2-p,2-p)}C$	appear in $D_5$ if and only if $\alpha_1+\beta_1\geq -3-p$ and $p-\alpha_1-\beta_1$ is even, in which case the multiplicity is one.
\end{enumerate}
\end{lemma}

\begin{proof}
(1) By Theorem \ref{thm:charMatrices} we have that $\bS_{\alpha}A\oo \bS_{\beta}B\oo \bS_{(-2,-2,2-p,2-p)}C$ appears in $D_6$ if and only if $\bS_{\alpha}A\oo \bS_{\beta}B$ is a subrepresentation of $\bS_{(-2,-2,2-p,2-p)}(A\oo B)$. Dualizing and twisting by $\det(A^{\ast}\oo B^{\ast})^{\otimes -2}$, this is equivalent to $\bS_{(-\alpha_2-4,-\alpha_1-4)}A^{\ast}\oo \bS_{(-\beta_2-4,-\beta_1-4)}B^{\ast}$ being a subrepresentation of $\bS_{(p-4,p-4,0,0)}(A^{\ast}\oo B^{\ast})$. Using the notation of \cite[Corollary 4.3b]{raicu2012secant}, since $-\alpha_1\leq p$ and $-\beta_1\leq p$, we have $r=2p-8$, $f=p-4$, and $e=p-\alpha_1-\beta_1-12$. Then $e\geq 2f$ if and only if $\alpha_1+\beta_1\leq -4-p$. Thus, if $\alpha_1+\beta_1\geq -3-p$ then the desired multiplicity is zero. If $\alpha_1+\beta_1\leq -4-p$, then since $r=2f$, we have that the desired multiplicity is $r/2-f+1$ if $e$ is even, and it is $r/2-f$ if $e$ is odd. Thus, the multiplicity is one if and only if $p-\alpha_1-\beta_1$ is even, and zero otherwise.

(2) Let $\pi=\pi_{222}$, $Y=Y_{222}$, and $Y_5=\pi^{-1}(\ol{O}_5)$. By Lemma \ref{lem:decompn4} we have the following in the Grothendieck group of representations of $\GL$ (see Section \ref{sec:relmatrices}):
\begin{equation}\label{eq:ggpusheuler}
\big[\chi(\pi_{+}\mc{O}_Y(\ast Y_5))\big]=[D_6]+[D_5]+[D_2]+2[D_0].
\end{equation}
Let $\alpha,\beta\in \mathbb{Z}^2_{\dom}$ and $\gamma=(-2,-2,2-p,2-p)$ with $|\alpha|=|\beta|=-2p$, and define
$$
m_{\alpha,\beta,\gamma}=\big[ \chi(\pi_+ \mc{O}_Y(\ast Y_5)):\bS_{\alpha}A\oo \bS_{\beta}B \oo \bS_{\gamma}C\big].
$$
By \cite[Lemma 2.5, Proposition 2.10]{raicu2016characters} we have that 
$$
\big[\chi(\pi_+ \mc{O}_Y(\ast Y_5))\big]=\lim_{k\to \infty} \big[\bS_{(2k,2k)}A\oo \bS_{(2k,2k)}B\oo p_{2,2k}(C)\oo 
 S^{\ast}\otimes \det(A^{\ast}\oo B^{\ast}\oo C^{\ast})\big].
$$
Thus, by \cite[Lemma 2.3]{raicu2016characters} we have that
$$
m_{\alpha,\beta,\gamma}=\lim_{k\to \infty}\left(\sum_{J\in \binom{[4]}{2}}\big[S:\bS_{(2k-\alpha_2-8,2k-\alpha_1-8)}A\oo \bS_{(2k-\beta_2-8,2k-\beta_1-8)}B \oo \bS_{((p-6,p-6,-2,-2),2k,J)}C \big]\right).
$$
Note $\bS_{\lambda}A\oo \bS_{\mu}B\oo \bS_{\nu}C$ appears in $S$ only if $\lambda,\mu,\nu$ are partitions. We have for $k\gg 0$ that $((p-6,p-6,-2,-2),2k,J)$ is a partition if and only if $J=\{3,4\}$, in which case
$$
((p-6,p-6,-2,-2),2k,\{3,4\})=(2k-4,2k-4, p-4, p-4).
$$
By the Cauchy Formula applied to $S=\Sym((A\oo B)\oo C)$ we have
$$
m_{\alpha,\beta,\gamma}=\lim_{k\to \infty}\big[\bS_{(2k-4,2k-4, p-4, p-4)}(A\oo B):\bS_{(2k-\alpha_2-8,2k-\alpha_1-8)}A\oo \bS_{(2k-\beta_2-8,2k-\beta_1-8)}B]
$$
Twisting by $\det(A\oo B)^{\otimes 4-p}$ we want the multiplicity of $\bS_{(2k-\alpha_2-2p,2k-\alpha_1-2p)}A\oo \bS_{(2k-\beta_2-2p,2k-\beta_1-2p)}B$ in $\bS_{(2k-p,2k-p, 0, 0)}(A\oo B)$ for $k\gg 0$. We apply \cite[Corollary 4.3b]{raicu2012secant}, noting that $-\alpha_1\leq p$ and $-\beta_1\leq p$. Using the notation there, for $k\gg 0$ we have $r=4k-2p$, $f=2k-p$, and $e=6k-5p-\alpha_1-\beta_1$. Thus, for $k\gg 0$ we have $e\geq r-1$, $e\geq 2f$, $r$ even. Then $e$ is even if and only if $p-\alpha_1-\beta_1$ is even, in which case $m_{\alpha,\beta,\gamma}=r/2-f+1=1$. Otherwise, $m_{\alpha,\beta,\gamma}=r/2-f=0$. By part (1) and (\ref{eq:ggpusheuler}) we obtain (2), using that  no representation in the statement of the lemma belongs to $D_0$ or $D_2$ (Theorem \ref{thm:charMatrices}).
\end{proof}

\begin{lemma}\label{lem:witD_6primen4}
Let $n=4$ and $p\geq 4$. The following is true about weights of $D_6'$.
\begin{enumerate}
\item The representation $\bS_{(-3,3-2p)}A\oo \bS_{(-p,-p)}B\oo \bS_{(-2,-2,2-p,2-p)}C$ appears in $D_6'$ with multiplicity one.
\item The following representations do not appear in $D_6'$:
$$
\bS_{(-1,1-2p)}A\oo \bS_{(-p,-p)}B\oo \bS_{(-2,-2,2-p,2-p)}C,\quad \bS_{(-p,-p)}A\oo \bS_{(-1,1-2p)}B\oo \bS_{(-2,-2,2-p,2-p)}C,
$$ 
$$
\bS_{(-2,2-2p)}A\oo \bS_{(-p,-p)}B\oo \bS_{(-2,-2,2-p,2-p)}C,\quad \bS_{(-4,4-2p)}A\oo \bS_{(-p,-p)}B\oo \bS_{(-2,-2,2-p,2-p)}C.
$$
\end{enumerate}

\end{lemma}

\begin{proof}
Let $\pi=\pi_{222}$, $Y=Y_{222}$, and $Y_5=\pi^{-1}(\ol{O}_5)$. Let $\mc{L}=\bS_{(1,1)}A\oo \bS_{(1,1)}B\oo \bS_{(1,1)}\mc{Q}$ and let $f:Y\to \mathbb{G}(2;C)$ be the structure map for the vector bundle $Y$. Then $\mc{O}_Y(\ast Y_5)\otimes f^{\ast}(\mc{L})$ is the relative version of the $\D$-module $S_h\cdot \sqrt{h}$ studied in \cite{perlman2020equivariant}, so that $\mc{O}_Y(\ast Y_5)\otimes f^{\ast}(\mc{L})$ has composition factors $D_1^Y$, $D_2^Y$, $D_3^Y$, $D_4^Y$, $D_6'^Y$, each with multiplicity one. Since $D_6'$ is supported on $\ol{O}_6$ and $\pi$ is an isomorphism on $O_6$, we have the following in the Grothendieck group of $\GL$-representations:
\begin{equation}\label{eq:eulerD6prime}
\big[\chi(\pi_+(\mc{O}_Y(\ast Y_5)\otimes f^{\ast}(\mc{L})))\big]=c_0[D_0]+c_1[D_1]+c_2[D_2]+c_3[D_3]+c_4[D_4]+c_5[D_5]+[D_6'],
\end{equation}
for some $c_0,c_1,c_2,c_3,c_4,c_5\in \mathbb{Z}$.
Let $\alpha,\beta\in \mathbb{Z}^2_{\dom}$ and $\gamma \in \mathbb{Z}^4_{\dom}$ with $|\alpha|=|\beta|=|\gamma|$, and define
$$
m_{\alpha,\beta,\gamma}=\big[ \chi(\pi_+(\mc{O}_Y(\ast Y_5)\otimes f^{\ast}(\mc{L}))):\bS_{\alpha}A\oo \bS_{\beta}B \oo \bS_{\gamma}C\big].
$$
By \cite[Lemma 2.5, Proposition 2.10]{raicu2016characters} we have that 
$$
\big[\chi(\pi_+(\mc{O}_Y(\ast Y_5)\otimes f^{\ast}(\mc{L})))\big]=\lim_{k\to \infty} \big[\bS_{(2k+1,2k+1)}A\oo \bS_{(2k+1,2k+1)}B\oo p_{2,2k+1}(C)\oo 
 S^{\ast}\otimes \det(A^{\ast}\oo B^{\ast}\oo C^{\ast})\big].
$$
Let $\gamma=(-2,-2,2-p,2-p)$, so that $|\alpha|=|\beta|=-2p$. By \cite[Lemma 2.3]{raicu2016characters} we have that
$$
m_{\alpha,\beta,\gamma}=\lim_{k\to \infty}\left(\sum_{J\in \binom{[4]}{2}}\big[S:\bS_{(2k-\alpha_2-7,2k-\alpha_1-7)}A\oo \bS_{(2k-\beta_2-7,2k-\beta_1-7)}B \oo \bS_{((p-6,p-6,-2,-2),2k+1,J)}C \big]\right).
$$
Note $\bS_{\lambda}A\oo \bS_{\mu}B\oo \bS_{\nu}C$ appears in $S$ only if $\lambda,\mu,\nu$ are partitions. For $k\gg 0$ we have that $((p-6,p-6,-2,-2),2k,J)$ is a partition if and only if $J=\{3,4\}$, in which case
$$
((p-6,p-6,-2,-2),2k,\{3,4\})=(2k-3,2k-3, p-4, p-4).
$$
By the Cauchy Formula applied to $S=\Sym((A\oo B)\oo C)$ we have
$$
m_{\alpha,\beta,\gamma}=\lim_{k\to \infty}\big[\bS_{(2k-3,2k-3, p-4, p-4)}(A\oo B):\bS_{(2k-\alpha_2-7,2k-\alpha_1-7)}A\oo \bS_{(2k-\beta_2-7,2k-\beta_1-7)}B].
$$
Twisting by $\det(A\oo B)^{\otimes 4-p}$ we have
$$ 
m_{\alpha,\beta,\gamma}=\lim_{k\to \infty}\big[\bS_{(2k-p+1,2k-p+1)}(A\oo B):\bS_{(2k-2p-\alpha_2+1,2k-2p-\alpha_1+1)}A\oo \bS_{(2k-2p-\beta_2+1,2k-2p-\beta_1+1)}B].
$$
We apply \cite[Corollary 4.3b]{raicu2012secant}, noting that $-\alpha_1\leq p$ and $-\beta_1\leq p$. Using the notation there, for $k\gg 0$ we have $r=4k-2p+2$, $e=6k-5p-\alpha_1-\beta_1+3$, $f=2k-p+1$. Thus, $e\geq 2f$, $e\geq r-1$, and $r$ is even. Therefore, $e$ is even if and only if $p-\alpha_1-\beta_1$ is odd, in which case, $m_{\alpha, \beta, \gamma}=r/2-f+1=1$. Otherwise, when $p-\alpha_1-\beta_1$ is even, we have $m_{\alpha, \beta, \gamma}=r/2-f=0$.

By Theorem \ref{thm:charMatrices} we have that none of the representations in the statement of the lemma belong to $D_0$, $D_2$, and by Corollary \ref{cor:weightBoundsn4} none of these representations appear in $D_1$. If $p=4$ and $\alpha=(1,-7)$ and $\beta=(-4,-4)$ or vice versa we have  that $p-\alpha_1-\beta_1=7$, which is odd, so $m_{\alpha,\beta, \gamma}=1$. By Lemma \ref{lem:D6primeisH3} the representations $\bS_{(-1,-7)}A\oo \bS_{(-4,-4)}B\oo \bS_{(-2,-2,-2,-2)}C$, $\bS_{(-4,-4)}A\oo \bS_{(-1,-7)}B\oo \bS_{(-2,-2,-2,-2)}C$ do not appear in $D_6'$, and by Lemma \ref{lem:charD5D6ngeq4} they do not appear in $D_5$. Therefore, $c_3=c_4=1$ in (\ref{eq:eulerD6prime}). Similarly, by Lemma \ref{lem:D6primeisH3} and Theorem \ref{thm:charMatrices} the representation $\bS_{(-2,-6)}A\oo \bS_{(-4,-4)}B\oo \bS_{(-2,-2,-2,-2)}C$ ($p=4$) does not appear in $D_6'$, $D_3$, nor $D_4$. Since $p-\alpha_1-\beta_1=10$ is even, we conclude that $c_5=0$.

Therefore, given $\alpha,\beta, \gamma$ as in the statement of the lemma, we have
$$
\big[ D_6':\bS_{\alpha}A\oo \bS_{\beta}B\oo \bS_{\gamma}C\big]=m_{\alpha,\beta, \gamma}-\big[ D_3:\bS_{\alpha}A\oo \bS_{\beta}B\oo \bS_{\gamma}C\big]-\big[ D_4:\bS_{\alpha}A\oo \bS_{\beta}B\oo \bS_{\gamma}C\big].
$$
The proof of (1) and (2) then follows from a case-by-case analysis of the parity of $p-\alpha_1-\beta_1$. The representations in the first row of (2) belong to $D_3$ (resp. $D_4$), and the representations in the second row of (2) do not belong to $D_3$ nor $D_4$.
\end{proof}

\begin{lemma}\label{lem:directimD1}
Let $n\geq 4$. The following is true about $D_1$ and $D_7$.
\begin{enumerate}
\item $\mathbb{R}\pi'_{\ast}D_1^{Y'}$ has only cohomology in degree $3n-12$, isomorphic as representations of $\GL$ to $D_1$.
\item $\mathbb{R}\pi'_{\ast}D_7^{Y'}$ has only cohomology in degree $n-4$, isomorphic as representations of $\GL$ to $D_7$.	

\item The representation $\bS_{(-n-1,-2n+1)}A\oo \bS_{(-n-1,-2n+1)}B \oo \bS_{(-3^n)}C$ appears with multiplicity one in $D_1$, and does not appear in $D_2$.
\item The representation $\bS_{(-n-2,-2n+2)}A\oo \bS_{(-n-2,-2n+2)}B \oo \bS_{(-3^n)}C$ does not appear in $D_1$, and appears with multiplicity one in $D_2$.

\item The representation $\bS_{(-1,1-n)}A\oo \bS_{(-1,1-n)}B \oo \bS_{(-1^n)}C$ appears with multiplicity one in $D_7$, and does not appear in $D_8$.
\item The representation $\bS_{(-2,2-n)}A\oo \bS_{(-2,2-n)}B \oo \bS_{(-1^n)}C$ does not appear in $D_7$, and appears with multiplicity one in $D_8$.		
\end{enumerate}

\end{lemma}

\begin{proof}
(1) We consider the spectral sequence	
$$
\mathbb{R}^i\pi'_{\ast}\mathscr{H}^j_{Y_1'}(\mc{O}_{Y'})\implies H^{i+j}_{\ol{O}_1}(S).
$$
By Theorem \ref{BettiThm}(2) for $n=4$, the $E_2$ page of this spectral sequence has only four nonzero columns, corresponding to $j=10$, $11$, $13$, $15$. Note that column ten corresponds to the cohomology of $\mathbb{R}\pi'_{\ast}D_1^{Y'}$, whereas the other columns correspond to direct images of (direct sums of) $D_0^{Y'}$. We recall that $\mathbb{R}\pi'_{\ast}D_0^{Y'}$ has only cohomology in degree $4n-16$, the highest possible degree. Thus, by Theorem \ref{BettiThm}(2) for $n>4$, the spectral sequence does not converge properly unless the cohomology of $\mathbb{R}\pi'_{\ast}D_1^{Y'}$ is as claimed.

(2) Proven in Section \ref{sec:lcO7ngeq5}.

By application of the Fourier transform, it remains to prove (3) and (4). Let $\pi'=\pi_{224}$ and $Y'=Y_{224}$. By Lemma \ref{lem:wit17n4} we have that $\bS_{(-n-1,-2n+1)}A\oo \bS_{(-n-1,-2n+1)}B \oo \bS_{(-3,1-n,1-n,1-n)}\mc{Q}$ belongs to $D_1^{Y'}$ and  $\bS_{(-n-2,-2n+2)}A\oo \bS_{(-n-2,-2n+2)}B \oo \bS_{(-3,1-n,1-n,1-n)}\mc{Q}$ does not belong to $D_1^{Y'}$. By (1) and Lemma \ref{BottLemma} we have that $\bS_{(-n-1,-2n+1)}A\oo \bS_{(-n-1,-2n+1)}B \oo \bS_{(-3^n)}C$ belongs to $D_1$ with multiplicity one and $\bS_{(-n-2,-2n+2)}A\oo \bS_{(-n-2,-2n+2)}B \oo \bS_{(-3^n)}C$ does not appear in $D_1$.

Next, we consider the character of $D_2$. By Theorem \ref{thm:charMatrices} a representation of the form $\bS_{\lambda}(A\oo B)\oo \bS_{(-3^n)}C$ appears in $D_2$ if and only if $\lambda=(-3,1-n,1-n,1-n)$. By Cauchy's formula applied to $\bS_{(-3,1-n,1-n,1-n)}(A\oo B)=\bS_{(n-4,0,0,0)}(A\oo B)\oo \det(A\oo B)^{\otimes 2-2n}$ we have that $\bS_{(-n-2,-2n+2)}A\oo \bS_{(-n-2,-2n+2)}B$ appears in $\bS_{(-3,1-n,1-n,1-n)}(A\oo B)$ with multiplicity one, and $\bS_{(-n-1,-2n+1)}A\oo \bS_{(-n-1,-2n+1)}B$ does not appear. \end{proof}

\begin{lemma}\label{lem:pushD5}
Let $n\geq 4$. The following is true about $D_5$ and $D_6$.
\begin{enumerate}
\item $\mathbb{R}\pi'_{\ast}D_5^{Y'}$ has only cohomology in degree $2n-8$, isomorphic as representations of $\GL$ to $D_5$.
\item  $\mathbb{R}\pi'_{\ast}D_6^{Y'}$ has only cohomology in degree $2n-8$, isomorphic as representations of $\GL$ to $D_6$.

\item Let $\alpha,\beta \in \mathbb{Z}^2_{\dom}$ with $|\alpha|=|\beta|=|\gamma|=-2n$. Representations of the form $\bS_{\alpha}A\oo \bS_{\beta}B\oo \bS_{(-2^n)}C$	appear in $D_6$ if and only if $\alpha_1+\beta_1\leq -4-n$ and $n-\alpha_1-\beta_1$ is even, in which case the multiplicity is one.
\item Let $\alpha,\beta \in \mathbb{Z}^2_{\dom}$ with $|\alpha|=|\beta|=|\gamma|=-2n$. Representations of the form $\bS_{\alpha}A\oo \bS_{\beta}B\oo \bS_{(-2^n)}C$	appear in $D_5$ if and only if $\alpha_1+\beta_1\geq -3-n$ and $n-\alpha_1-\beta_1$ is even, in which case the multiplicity is one.\end{enumerate}
		
\end{lemma}

\begin{proof}
By discussion in Section \ref{sec:o5ngeq5} we have (1). Using that $D_6=L_{Z_2}$ we have that (2) can be deduced from \cite{raicu2016characters, raicuWeymanLocal} (see \cite[Lemma 3.4]{socledegs} for the explicit statement). By (1), (2), Lemma \ref{BottLemma}, and Lemma \ref{lem:charD5D6ngeq4} we obtain (3),(4). \end{proof}

\begin{lemma}\label{lem:pushD6prime}
Let $n\geq 4$. The following is true about $D_6'$.
\begin{enumerate}
\item $\mathbb{R}\pi'_{\ast}D_6'^{Y'}$ has only cohomology in degree $2n-8$, isomorphic as representations of $\GL$ to $D_6'$. 

\item The representation $\bS_{(-3,3-2n)}A\oo \bS_{(-n,-n)}B\oo \bS_{(-2^n)}C$ appears in $D_6'$ with multiplicity one.	
\item The following representations do not appear in $D_6'$:
$$
\bS_{(-1,1-2n)}A\oo \bS_{(-n,-n)}B\oo \bS_{(-2^n)}C,\quad \bS_{(-n,-n)}A\oo \bS_{(-1,1-2n)}B\oo \bS_{(-2^n)}C,
$$ 
$$
\bS_{(-2,2-2n)}A\oo \bS_{(-n,-n)}B\oo \bS_{(-2^n)}C,\quad \bS_{(-4,4-2n)}A\oo \bS_{(-n,-n)}B\oo \bS_{(-2^n)}C.
$$

\end{enumerate}

\end{lemma}

\begin{proof}
 For $n=4$ we have that (2) and (3) are the case $p=4$ of Lemma \ref{lem:witD_6primen4}, so we assume $n\geq 5$. Let $\pi'=\pi_{224}$ and $Y'=Y_{224}$. By Lemma \ref{lem:D6primeisH3} we have that if $\bS_{\alpha}A\oo \bS_{\beta}B\oo \bS_{\gamma}\mc{Q}$ belongs to $D_6'^{Y'}$ then $\gamma_2\geq -2$ and $\gamma_3\leq -2$. Thus, by Bott's Theorem (Lemma \ref{BottLemma}) we have that $\mathbb{R}\pi'_{\ast}D_6'^{Y'}$ has only cohomology in degree $2n-8$. By the discussion in Section \ref{sec:lcO7ngeq5}, since the spectral sequence $E_r^{i,j}$ there is degenerate, it follows that $\mathbb{R}^{2n-8}\pi'_{\ast}D_6'^{Y'}$ is isomorphic as representation of $\GL$ to an equivariant holonomic $\D_V$-module. 

By Lemma \ref{BottLemma} it follows that if $\bS_{\alpha}A\oo \bS_{\beta}B\oo \bS_{\gamma}C$ belongs to $\mathbb{R}^{2n-8}\pi'_{\ast}D_6'^{Y'}$ then 
$$
\gamma_3=\cdots=\gamma_{n-2}=-2.
$$
Thus, by Theorem \ref{thm:charMatrices} we have that none of $D_0$, $D_2$, $D_8$, $D_9$ are subrepresentations of $\mathbb{R}^{2n-8}\pi'_{\ast}D_6'^{Y'}$. By Lemma \ref{lem:directimD1}(1)(2) we have that subrepresentations $\bS_{\alpha}A\oo \bS_{\beta}B\oo \bS_{\gamma}C$ of $D_1$ satisfy $\gamma_{n-2}\leq -3$, and those of $D_7$ satisfy $\gamma_3\geq -1$. Therefore, neither $D_1$ nor $D_7$ are subrepresentations of $\mathbb{R}^{2n-8}\pi'_{\ast}D_6'^{Y'}$.

  By Lemma \ref{lem:witD_6primen4} and Bott's Theorem, the representation of Lemma \ref{lem:pushD6prime}(2) appears in $\mathbb{R}^{2n-8}\pi'_{\ast}D_6'^{Y'}$ with multiplicity one, and the representations of Lemma \ref{lem:pushD6prime}(3) do not appear. By Theorem \ref{thm:charMatrices} the representations in the first row of Lemma \ref{lem:pushD6prime}(3) belong to $D_3$ resp. $D_4$, and by Lemma \ref{lem:pushD5} the representations in the second row of Lemma \ref{lem:pushD6prime}(3) belong to $D_5$ resp. $D_6$. Thus, none of $D_3$, $D_4$, $D_5$, $D_6$ are subrepresentations of $\mathbb{R}^{2n-8}\pi'_{\ast}D_6'^{Y'}$, so by process of elimination and Lemma \ref{lem:witD_6primen4}, we have $\mathbb{R}^{2n-8}\pi'_{\ast}D_6'^{Y'}\cong D_6'$ and (2) and (3) hold.
\end{proof}

\begin{corollary}\label{cor:bottcorsimple}
Let $n\geq 4$. The following is true about the characters of $D_1$, $D_5$, $D_6'$, $D_7$. 
\begin{enumerate}
\item if $\bS_{\alpha}A\oo \bS_{\beta}B\oo \bS_{\gamma}C$ belongs to $D_1$ then $\gamma_1\geq -3$, $\gamma_{n-2}\leq -3$ and
$$
\gamma_2=\cdots =\gamma_{n-3}=-3.
$$
\item if $\bS_{\alpha}A\oo \bS_{\beta}B\oo \bS_{\gamma}C$ belongs to $D_5$ or $D_6'$ then $\gamma_2\geq -2$, $\gamma_{n-1}\leq -2$ and 
$$
\gamma_3=\cdots=\gamma_{n-2}=-2.
$$
\item if $\bS_{\alpha}A\oo \bS_{\beta}B\oo \bS_{\gamma}C$ belongs to $D_7$ then $\gamma_3\geq -1$, $\gamma_n\leq -1$ and
$$
\gamma_4=\cdots =\gamma_{n-1}=-1.
$$
\end{enumerate}

\end{corollary}

\begin{proof}
The case $n=4$ follows from Corollary \ref{cor:weightboundsD5n4}, Corollary \ref{cor:weightBoundsn4}, and Lemma \ref{lem:D6primeisH3}. The case $n\geq 5$ follows from Lemma \ref{BottLemma}, Lemma \ref{lem:directimD1}, Lemma \ref{lem:pushD5}, Lemma \ref{lem:pushD6prime}.	
\end{proof}

\begin{proof}[Proof of Theorem \ref{thm:witnessngeq4}]
By Theorem \ref{thm:charMatrices} the weight $(-2n,-2n)\times (-2n,-2n)\times (-4^n)$ appears with multiplicity one in $D_0$, and does not appear in $D_2$, $D_3$, $D_4$, $D_6$, $D_8$, $D_9$.	By Corollary \ref{cor:bottcorsimple} this weight does not appear in $D_1$, $D_5$, $D_6'$, $D_7$. Therefore, it is a witness weight for $D_0$. Applying $\mc{F}$ we conclude that $(0,0)\times (0,0)\times (0^n)$ is a witness weight for $D_9$.

By Lemma \ref{lem:directimD1} the weight $(-n-1,-2n+1)\times (-n-1,-2n+1)\times (-3^n)$ appears with multiplicity one in $D_1$, and does not appear in $D_2$. By Theorem \ref{thm:charMatrices} this weight does not appear in $D_0$, $D_3$, $D_4$, $D_6$, $D_8$, $D_9$. By Corollary \ref{cor:bottcorsimple} this weight does not appear in $D_5$, $D_6'$, $D_7$. Therefore, it is a witness weight for $D_1$. Applying $\mc{F}$ we conclude that $(-1,1-n)\times (-1,1-n)\times (-1^n)$ is a witness weight for $D_7$.

By Lemma \ref{lem:directimD1}(4) the weight $(-n-2,-2n+2)\times (-n-2,-2n+2)\times (-3^n)$ appears with multiplicity one in $D_2$, and by Theorem \ref{thm:charMatrices} it does not appear in $D_0$, $D_3$, $D_4$, $D_6$, $D_8$, $D_9$. By Lemma \ref{lem:directimD1} it does not appear in $D_1$, and by Corollary \ref{cor:bottcorsimple} it does not appear in $D_5$, $D_6'$, $D_7$. Therefore, it is a witness weight for $D_2$. Applying $\mc{F}$ we conclude that $(-2,2-n)\times (-2,2-n)\times (-1^n)$ is a witness weight for $D_8$.

By Theorem \ref{thm:charMatrices} the weight $(-1,1-2n)\times (-n,-n)\times (-2^n)$ appears with multiplicity one in $D_3$, and does not appear in $D_0$, $D_2$, $D_4$, $D_8$, $D_9$. By Lemma \ref{lem:pushD5} it does not appear in $D_5$, $D_6$, and by Corollary \ref{cor:bottcorsimple} it does not appear in $D_1$, $D_7$. By Lemma \ref{lem:pushD6prime} it does not appear in $D_6'$. Therefore, it is a witness weight for $D_3$. By symmetry, we also obtain the witness weight for $D_4$.

By Lemma \ref{lem:pushD5} the weight $(-2,2-2n)\times (-n,-n) \times (-2^n)$ appears with multiplicity one in $D_5$, and does not appear in $D_6$. By Lemma \ref{lem:pushD6prime} it does not appear in $D_6'$, and by Corollary \ref{cor:bottcorsimple} it does not appear in $D_1$, $D_7$. By Theorem \ref{thm:charMatrices} it does not appear in $D_0$, $D_2$, $D_3$, $D_4$, $D_8$, $D_9$. Therefore, it is a witness weight for $D_5$. By an identical argument, $(-4,4-2n)\times (-n,-n)\times (-2^n)$ is a witness weight for $D_6$.

By Lemma \ref{lem:pushD6prime} the weight $(-3,3-2n)\times (-n,-n)\times (-2^n)$ appears with multiplicity one in $D_6'$. By Theorem \ref{thm:charMatrices} it does not appear in $D_0$, $D_2$, $D_3$, $D_4$, $D_8$, $D_9$. By Lemma \ref{lem:pushD5} it does not appear in $D_5$, $D_6$, and by Corollary \ref{cor:bottcorsimple} it does not appear in $D_1$, $D_7$. Therefore it is a witness weight for $D_6'$.
\end{proof}

\section{Lyubeznik numbers and intersection cohomology}\label{sec:lyub}

In this section we explain how to obtain the Lyubeznik number of orbit closures, and $H^i_{\{0\}}(M)$ for each simple equivariant $\D$-module $M$. By Proposition \ref{prop:intcoho}, we thus also obtain all the intersection cohomology groups for each orbit closure.

As the determinantal case is done in \cite{lorcla}, we only need to deal with $\ol{O}_i$ for $i=1,5,7$, as well as calculate $H^j_{\{0\}}(D_6')$, for $j\geq 0$. For $i=1,5,7$, let $R_i=\C[\ol{O}_i]_{\mathfrak{m}}$, where $\mathfrak{m}$ denotes the homogeneous maximal ideal.

\begin{prop}\label{prop:loc0D1}
For all $n\geq 3$ the nonzero local cohomology modules of $D_1$ with support in $\{0\}$ are given by	
$$
H^{n-2}_{\{0\}}(D_1)=E,\quad H^{n}_{\{0\}}(D_1)=E^{\oplus 2},\quad H^{n+2}_{\{0\}}(D_1)=E.
$$
The nonzero Lyubeznik numbers of $R_1$ are:
\begin{itemize}
\item[(a)] $n=3$: $\lambda_{0,3}(R_1)= 2, \, \lambda_{3,5}(R_1)=2, \, \lambda_{5,5}(R_1)=1$; 
\item[(b)] $n\geq 4$: $\lambda_{0,3}(R_1) =2 , \lambda_{0,5}(R_1) = 1, \lambda_{n-2,n+2}(R_1)=1, \lambda_{n,n+2}(R_1)=2, \lambda_{n+2,n+2}(R_1)=1.$
\end{itemize}
\end{prop}

\begin{proof}
The claim about the Lyubeznik numbers follows readily from Theorem \ref{BettiThm} by looking at the spectral sequence 
\[H^i_{\{0\}} H^j_{\ol{O}_1} (S) \implies H^{i+j}_{\ol{O}_1}(S).\]
When $n\geq 4$, this also yields the first part, while for $n=3$ it follows from the long exact sequence obtained by applying the functor $H_{\{0\}}^0(-)$ to the short exact sequence in Theorem \ref{BettiThm}(1).
\end{proof}

As $\mc{F}(D_1)\cong D_7$, we readily obtain the following by Proposition \ref{prop:intcoho}.

\begin{corollary}\label{cor:loc0D7}
The nonzero local cohomology modules of $D_7$ with support in $\{0\}$ are given by	
$$
H^{3n-2}_{\{0\}}(D_7)=E,\quad H^{3n}_{\{0\}}(D_7)=E^{\oplus 2},\quad H^{3n+2}_{\{0\}}(D_7)=E.
$$
\end{corollary}

\begin{theorem}
For all $n\geq 3$ the nonzero local cohomology modules of $D_5$ with support in $\{0\}$ are given by	
$$
H^{2n-3}_{\{0\}}(D_5)=E,\quad H^{2n-1}_{\{0\}}(D_5)=E,\quad H^{2n+1}_{\{0\}}(D_5)=E,\quad H^{2n+3}_{\{0\}}(D_5)=E.
$$
The nonzero Lyubeznik numbers of $R_5$ are:
\begin{itemize}
\item[(a)] $n=3$: $\lambda_{9,9}(R_5)=1$; 
\item[(b)] $n=4$: $\lambda_{3,10}(R_5) = \lambda_{5,10}(R_5) = \lambda_{7,10}(R_5)= \lambda_{5,11}(R_5)=\lambda_{7,11}(R_5)= \lambda_{9,11}(R_5)= \lambda_{11,11}(R_5)=1$;
\item[(c)] $n\geq 5$: $\lambda_{0,10}(R_5) = \lambda_{n-3,n+6}(R_5) = \lambda_{n-1,n+6}(R_5)= \lambda_{n+1,n+6}(R_5)=\lambda_{n+3,n+6}(R_5)= \lambda_{2n-3,2n+3}(R_5)$ $= \lambda_{2n-1,2n+3}(R_5)=\lambda_{2n+1,2n+3}(R_5)=\lambda_{2n+3,2n+3}(R_5)=1$.
\end{itemize}
\end{theorem}

\begin{proof}
We analyze the spectral sequence 
\[H^i_{\{0\}} H^j_{\ol{O}_5} (S) \implies H^{i+j}_{\ol{O}_5}(S).\]
First, let $n=3$. By Theorem \ref{loc05}, the spectral sequence degenerates, and the claim about the Lyubeznik numbers follows. By \cite[Theorem 1.1]{lorcla}, the only non-zero local cohomology modules of $D_2$ supported in $\{0\}$ are in degrees $2, 4, 6$ (in which case they are all equal to $E$). From the long exact sequence obtained by applying the functor $H_{\{0\}}^0(-)$ to the short exact sequence in Theorem \ref{loc05}(1), we obtain the first part of the statement in this case.

Next, let $n=4$. By Theorem \ref{loc05}, the module $H^6_{\ol{O}_5}(S)$ is isomorphic to what is called $Q_1$ in \cite[Theorem 1.6]{lorcla}, therefore its only non-zero local cohomology modules supported in $\{0\}$ are in degrees $3, 5, 7$ (in which case they are all equal to $D_0$). The spectral sequence now gives the result.

Lastly, let $n\geq 5$. Again, by \cite[Theorem 1.1]{lorcla}, the only non-zero local cohomology modules of $D_2$ supported in $\{0\}$ are in degrees $n-3, n-1, n+1, n+3$ (in which case they are all equal to $E$). In the spectral sequence, the term $H^0_{\{0\}} H^{4n-10}_{\ol{O}_5}(S)$ can be canceled out by either the term $H^{n-3}_{\{0\}} H^{3n-6}_{\ol{O}_5}(S)$ or by the term $H^{2n-6}_{\{0\}} H^{2n-3}_{\ol{O}_5}(S)$, in principle. To see that the latter is not possible, it is enough to see that  $H^{2n-6}_{\{0\}}(D_5) = 0$, which follows from Proposition \ref{prop:intcoho} as $c= 2n-3$. The spectral sequence gives now the required result.
\end{proof}

\begin{theorem}
For $n\geq 3$, the nonzero local cohomology modules of $D_6'$ with support in $\{0\}$ are given by	
$$
H^{2n-2}_{\{0\}}(D_6')=E^{\oplus 2},\quad H^{2n}_{\{0\}}(D_6')=E^{\oplus 2},\quad H^{2n+2}_{\{0\}}(D_6')=E^{\oplus 2}.
$$
The nonzero Lyubeznik numbers of $R_7$ are
\item[(a)] $n=3$: $\lambda_{11,11}(R_7)= 1$; 
\item[(b)] $n\geq 4$: $\lambda_{0,11}(R_7) =1 , \lambda_{n-2,n+2}(R_7) = 1, \lambda_{n,n+2}(R_7)=2, \lambda_{n+2,n+2}(R_7)=1, \lambda_{2n-2,2n+5}(R_7)=2,$ 

$\lambda_{2n,2n+5}(R_7)=2, \lambda_{2n+2,2n+5}(R_7)=2, \lambda_{3n-2,3n+2}(R_7)=1, \lambda_{3n,3n+2}(R_7)=2, \lambda_{3n+2,3n+2}(R_7)=1.$
\end{theorem}

\begin{proof}
We analyze the spectral sequence 
\[H^i_{\{0\}} H^j_{\ol{O}_7} (S) \implies H^{i+j}_{\ol{O}_7}(S).\]
First, let $n=3$. Then the spectral sequence degenerates and the claim on the Lyubeznik numbers is clear. The claim on the local cohomology modules of $D_6'$ with support in $\{0\}$ follows by tracing the long exact sequences associated to the short exact sequences in Theorem \ref{loc07} (1). Namely, by Proposition \ref{prop:loc0D1} we get that nonzero local cohomology modules of $S_f/\D_Vf^{-1}$ with support in $\{0\}$ are $H^3_{\{0\}}(S_f/\D_Vf^{-1})=E^{\oplus 2}$ and $H^5_{\{0\}}(S_f/\D_Vf^{-1})=E$. Then the nonzero local cohomology modules of $\D_Vf^{-1}/S$ with support in $\{0\}$ are $H^4_{\{0\}}(\D_Vf^{-1}/S)=E^{\oplus 2}$, $H^6_{\{0\}}(\D_Vf^{-1}/S)=H^{11}_{\{0\}}(\D_Vf^{-1}/S)=E$. By the analogue of Proposition \ref{prop:intcoho} for non-trivial local systems, since $D_6'$ is self-dual and $\mc{F}(D_6')=D_6'$, we see that $H^{11}(D_6')=0$. The claim now follows from the long exact sequence induced by $0\to D_7 \to \D_Vf^{-1}/S \to D_6'\to 0$ and Corollary \ref{cor:loc0D7}.

Now let $n\geq 4$. Again, by (the analogue of) Proposition \ref{prop:intcoho} we see that $H_{\{0\}}^{2n-5}(D_6') = 0$ as $2n-5 < \op{codim}(O_6, V)=2n-4$. Therefore, by Proposition \ref{prop:loc0D1} and Corollary \ref{cor:loc0D7} in the spectral sequence the term $E=H^0_{\{0\}} H^{4n-11}_{\ol{O}_7}(S)$ must be canceled out by $E=H^{n-2}_{\{0\}} H^{3n-8}_{\ol{O}_7}(S)$. The rest of the cancelations are straightforward, which yields the local cohomology modules of $D_6'$, and therefore the Lyubeznik numbers of $R_7$.
\end{proof}

\section*{Acknowledgments}
We thank Vic Reiner for a helpful conversation regarding the character calculations in Section \ref{sec:characters}.

\end{document}